\documentclass[12pt]{amsart}
\usepackage[utf8]{inputenc}
\usepackage{hyperref}
\usepackage[letterpaper,centering,text={6in,9in}]{geometry}
\usepackage[all]{xy}
\usepackage{cp2-conic-line}
\usepackage{graphicx}

\usepackage{setspace}

\usepackage{datetime}
\newdateformat{ymddate}{\THEYEAR--\twodigit{\THEMONTH}--\twodigit{\THEDAY}}
\ymddate

\title[Mirror of the projective plane]{Floer cohomology in the mirror of the projective plane and a binodal cubic curve}
\author{James Pascaleff}
\date{Revised 2013--12--19}
\address{Department of Mathematics, MIT, Cambridge, MA 02139, USA}
\curraddr{Department of Mathematics, The University of Texas at
  Austin, Austin, TX 78712, USA}
\email{jpascaleff@math.utexas.edu}

\begin{document}
\begin{abstract}
  We construct a family of Lagrangian submanifolds in the
  Landau--Ginzburg mirror to the projective plane equipped with a
  binodal cubic curve as anticanonical divisor. These objects
  correspond under mirror symmetry to the powers of the twisting sheaf
  $\cO(1)$, and hence their Floer cohomology groups form an algebra
  isomorphic to the homogeneous coordinate ring. An interesting
  feature is the presence of a singular torus fibration on the mirror,
  of which the Lagrangians are sections. This gives rise to a
  distinguished basis of the Floer cohomology and the homogeneous
  coordinate ring parametrized by fractional integral points in the
  singular affine structure on the base of the torus fibration. The
  algebra structure on the Floer cohomology is computed using the
  symplectic techniques of Lefschetz fibrations and the TQFT counting
  sections of such fibrations. We also show that our results agree
  with the tropical analog proposed by
  Abouzaid--Gross--Siebert. Extensions to a restricted class of
  singular affine manifolds and to mirrors of the complements of
  components of the anticanonical divisor are discussed.
\end{abstract}

\maketitle
\tableofcontents

%\begin{doublespace}
\section{Introduction}
\label{sec:intro}

This article is concerned with a case of mirror symmetry relating the
algebraic geometry of a Fano manifold to the symplectic geometry of a
Landau--Ginzburg model. The Fano manifold is the projective plane
$\CP^2$ with homogeneous coordinates $x, y, z$, and additionally we choose as anticanonical divisor $D = C
\cup L$, the union of a conic $C = \{xz-y^2\}$ and a line $L = \{y=0\}$ (a binodal cubic
curve). According to Auroux \cite{auroux07}, the mirror of this pair
$(\CP^2,D)$ is the Landau--Ginzburg model $(X^\vee,W)$
\begin{equation}
  \begin{split}
    X^\vee = \{(u,v) \in \C^2 \mid uv \neq 1\}, \quad
    W = u + \frac{e^{-\Lambda}v^2}{uv-1}
  \end{split}
\end{equation}
The function $W$ is known as the superpotential. This example is
interesting in the context of the Strominger-Yau-Zaslow picture of
mirror symmetry \cite{syz} in terms of dual torus fibrations, since
the spaces $\CP^2\setminus D$ and $X^\vee$ admit special Lagrangian
torus fibrations with a singularity. The presence of singularities in
the torus fibration is known to make the mirror duality vastly more
complicated, but in this article we develop techniques to deal with it
in the above case (and other cases with similar properties, see \S
\ref{sec:other-manifolds}).

\subsection{Summary of results}
\label{sec:outline}
The piece of symplectic geometry in the Landau--Ginzburg model
$(X^\vee,W)$ we study is the Floer cohomology of Lagrangian
submanifolds, while on the mirror side $\CP^2$ we consider the
cohomology of the coherent sheaves $\cO_{\CP^2}(d)$. We construct a
symplectic manifold $X(B)$, that serves as our symplectic model for
$X^\vee$, and a collection of Lagrangian submanifolds $\{L(d)\}_{d\in
  \Z}$ in $X(B)$ that is mirror to the collection
$\{\cO_{\CP^2}(d)\}_{d\in \Z}$. The constructions of $X(B)$ and $L(d)$
are based on the SYZ picture of mirror symmetry in terms of torus
fibrations and affine manifolds, and also use the theory of Lefschetz
fibrations. The manifold $X(B)$ carries two fibrations: one is a
fibration over a singular affine manifold $B$ whose fibers are Lagrangian
tori, while the other is a Lefschetz fibration whose fibers are
symplectic submanifolds. See section \ref{sec:main-construction} for
the definitions.

To state the main result, let $A_n = H^0(\CP^2,\cO_{\CP^2}(n))$, and
choose a basis $x,y,z$ of $A_1$. Thus $A = \bigoplus_{n=0}^\infty A_n
\cong \C[x,y,z]$ is the homogeneous coordinate ring of $\CP^2$. The Floer cohomology of two Lagrangian submanifolds $L,L'$ is a $\Z$-graded $\C$-vector space denoted $HF^*(L,L')$.
\begin{theorem}
  \label{thm:triangle-products}
  For each $d \in \Z$ and $n \geq 0$, there is an isomorphism
  \begin{equation}
    \psi_{d,n} : HF^0(L(d),L(d+n)) \to A_n
  \end{equation}
  carrying the basis of intersection points $L(d) \cap L(d+n)$ to the
  basis of $A_n$ consisting of the polynomials of the form
  \begin{equation}
    \{x^a(xz-y^2)^iy^{n-a-2i}\} \cup
    \{z^a(xz-y^2)^iy^{n-a-2i}\}
  \end{equation}
  (where we require $a \geq 0, i \geq 0, n-a-2i \geq 0$).
  The system of isomorphisms $\psi_{d,n}$ intertwines the Floer triangle product
  \begin{equation}
    \mu^2: HF^0(L(d+n),L(d+n+m)) \otimes HF^0(L(d),L(d+n)) \to HF^0(L(d),L(d+n+m))
  \end{equation}
  and the product of polynomials $A_m \otimes A_n \to A_{n+m}$.
\end{theorem}

Let us remark on the formulation of the theorem. From the construction
of the Lagrangians $L(d)$ it is easy to count the number of
intersection points and show that the map $\psi_{d,n}$ exists as an
isomorphism of vector spaces. Such a map exists for any choice of
basis of $A_n$. The last statement equating Floer triangle products
and the products of polynomials is the nontrivial bit that ties the
choices together and shows that we have chosen the right basis for
$A_n$. 

This basis we obtain is related to the choice of divisor $D = C\cup L$, as it consists of monomials in the defining section $p =xz-y^2$ of $C$, the defining section $y$ of $L$, and the homogeneous variables $x$ or $z$ (but not both $x$ and $z$ in the same monomial). Another point of view has to do with the fact that the ring of functions on $\CP^2\setminus D$ is a cluster algebra \cite{fomin-zelevinsky-I}. There are two triples of homogeneous forms on $\CP^2$, $(x,p,y)$ and $(z,p,y)$, which are related by the so-called exchange relation $xz = y^2 + p$, and our basis consists of sections that are monomials in either triple. 

For the proof of this theorem, the majority of our technical efforts
are aimed at computing the Floer triangle product. This occupies
section \ref{sec:degeneration}. At this point in the argument, the
picture of $X(B)$ as a Lefschetz fibration is the focus. The
holomorphic triangles we need to find can be represented as sections
of this Lefschetz fibration. The problem of counting sections of
Lefschetz fibrations has a TQFT-type structure, developed by
Seidel. This structure, where one breaks a given problem into simpler
ones by degenerating the base of the Lefschetz fibration, provides the
basis of our technique.

There are several variations on Theorem \ref{thm:triangle-products} that
are accessible using the same holomorphic curve analysis. We consider
the complement of (some components of) the anticanonical divisor in
$\CP^2$, and on the Landau--Ginzburg side, we must change the
superpotential $W$ and consider a certain form of wrapped Floer
cohomology. Our techniques allow us to treat three cases:
\begin{enumerate}
\item the mirror of $(\CP^2 \setminus L, C \setminus (C\cap L))$,
\item the mirror of $(\CP^2 \setminus C, L \setminus (L\cap C))$, and 
\item the mirror of $(\CP^2 \setminus (C\cup L), \emptyset)$.
\end{enumerate}
In each case the mirror space is the same manifold $X^\vee$, and the
Lagrangians mirror to line bundles are the same $L(d)$, but in each
case there is a different prescription for wrapping the Lagrangian
submanifolds. The first two cases require ``partially wrapped'' Floer
cohomology, while the third uses the more standard ``fully wrapped''
Floer cohomology. In each case we denote wrapped Floer cohomology by
$HW^*(L_1,L_2)$.  Let $U = \CP^2 \setminus L$, $\CP^2 \setminus C$, or
$\CP^2 \setminus (C\cup L)$, and let $A_n(U) = H^0(U,\cO_U(n))$. The
space $A_n(U)$ consists of rational functions in the variables
$x,y,z$ that are regular on $U$ and have degree $n$. The wrapped version of Theorem
\ref{thm:triangle-products} is as follows (see
\S\ref{sec:complements})
\begin{theorem}
  \label{thm:wrapped-products}
  
  For $d \in \Z$ and $n \geq 0$, there is an isomorphism
  \begin{equation}
    \psi_{d,n} : HW^0(L(d),L(d+n)) \to A_n(U)
  \end{equation}
  carrying a certain distinguished basis of $HW^0(L(d),L(d+n))$ to the
  basis of $A_n(U)$ consisting of rational functions of the form 
  \begin{equation}
    \{x^a(xz-y^2)^iy^{n-a-2i}\} \cup
    \{z^a(xz-y^2)^iy^{n-a-2i}\}
  \end{equation}
  where the exponents $a$ and $i$ are restricted to those values which
  actually give elements of $A_n(U)$. The system of isomorphisms
  $\psi_{d,n}$ intertwines the Floer triangle product on wrapped Floer
  cohomology and the product of rational functions $A_m(U) \otimes A_n(U) \to
  A_{n+m}(U)$.
\end{theorem}

While our holomorphic triangle counts use the structure of a Lefschetz
fibration on $X(B)$, we can also study the geometry of the $X(B)$ as a
Lagrangian torus fibration over the base affine manifold $B$. One of
the expectations of SYZ philosophy is that much of the geometry of the
space $X(B)$ can be seen tropically, in terms of the geometry of the
affine base $B$.

In fact, this is how we arrive at the construction of the manifold
$X(B)$. In section \ref{sec:fiber-superpotential}, we construct a
tropicalization of the fiber of the superpotential $W^{-1}(c)$ over a
large positive real value, with respect to a torus fibration on the
manifold $X^\vee$ with a single singularity. This gives a tropical
curve in the base of our torus fibration. It bounds a compact region
$B$ in the base, which agrees with the affine base of the torus
fibration on $\CP^2\setminus D$. The purpose of this section is to
motivate the use of the singular affine manifold $B$ as the basis for
our main symplectic constructions in section
\ref{sec:main-construction}.

We also are able to verify a conjectural tropical description of Floer
cohomology in the cases we study. This
description comes from a recent proposal of Abouzaid, Gross and
Siebert for a tropical Fukaya category associated to a singular affine
manifold. The Lagrangian submanifolds are taken to be sections of the
torus fibration, but all we see of them tropically are their
intersection points, which map to fractional integral points of the
affine manifold $B$. The Floer triangle product corresponds to a
``tropical triangle product'' counting certain tropical curves in $B$
joining the fractional integral points. Since we do not say anything
about degenerating holomorphic polygons to tropical ones, we merely
verify the equivalence by matching bases and computing on both
sides. See section \ref{sec:trop-fuk} for the precise definitions of
the terms.
\begin{theorem}
  \label{thm:tropical-trangles}
  There is a bijection between the basis of intersection points
  $L(d)\cap L(d+n)$ for $HF^0(L(d),L(d+n))$ and the set of
  $(\frac{1}{n})$-integral points of the affine manifold $B$. Under
  this bijection, the counts of pseudo-holomorphic triangles
  contributing to the Floer triangle product are equal to counts of
  tropical curves in $B$ joining $(\frac{1}{n})$-integral points.
\end{theorem}

The techniques developed in this article apply to a larger but still
restricted class of $2$--dimensional singular affine manifolds, where
the main restriction is that all singularities have \emph{parallel
  monodromy-invariant directions}. The generalization to these types
of manifolds is discussed in \ref{sec:other-manifolds}.

\subsection{Context and related work}
\label{sec:context}

% Mirror symmetry is the name given to the phenomenon of deep,
% non-trivial, and sometimes even spectacular equivalences between the
% geometries of certain pairs of spaces. Such a pair $(X,X^\vee)$ is
% called a mirror pair, and we say that $X^\vee$ is the mirror to $X$
% and vice--versa. A byword for mirror symmetry is Since this discovery, the study of mirror symmetry has
% expanded in many directions, both in physics and mathematics, allowing
% for generalization of the class of spaces considered, providing new
% algebraic ideas for how the equivalence ought to be conceptualized,
% and giving geometric insight into how a given space determines its
% mirror partner. 

% In this introduction we provide some orientation and
% context that we hope will enable the reader to situate our work within
% this constellation of ideas.

\subsubsection{Manifolds with effective anticanonical divisor and Landau--Ginzburg models}
\label{sec:intro-fano}

The class of spaces originally considered in mirror symmetry were
Calabi--Yau manifolds. Roughly speaking, the mirror to a compact
Calabi--Yau manifold $X$ is another compact Calabi--Yau manifold
$X^\vee$ of the same dimension. For example, it is in this context
that we have the equivalence, discovered by Candelas--de la
Ossa--Green--Parkes \cite{candelas-etal} and proven mathematically by
Givental \cite{givental-equivariant-GW} and Lian--Liu--Yau \cite{lly},
between the Gromov--Witten theory of the quintic threefold $V_5
\subset \PP^4$ and the theory of period integrals on a family of
Calabi--Yau threefolds known as ``mirror quintics.'' However, mirror
symmetry can be considered for other classes of manifolds such as
manifolds of general type ($\Omega^n_X$ ample), for which a proposal
has recently been made by Kapustin--Katzarkov--Orlov--Yotov
\cite{hms-general-type}, and (our present concern) manifolds with an
effective anticanonical divisor. In both of these latter cases the
mirror is not a manifold of the same class.

Let $X$ be an $n$--dimensional K\"{a}hler manifold with an effective
anticanonical divisor $D$. We regard $D$ as part of the data, and
write $(X,D)$ for the pair. We also choose a meromorphic $(n,0)$--form
$\Omega$ with no zeros and polar locus equal to $D$. According to
Hori--Vafa \cite{hori-vafa} and Givental, the mirror to $(X,D)$ is a
\emph{Landau--Ginzburg model} $(X^\vee, W)$, consisting of a
K\"{a}hler manifold $X^\vee$, together with a holomorphic function $W:
X^\vee \to \C$, called the superpotential.

The mirrors of toric Fano manifolds were derived by Hori--Vafa
\cite[\S 5.3]{hori-vafa} based on physical considerations. Let $X$ be
an $n$--dimensional toric Fano manifold, and let $D$ be the complement
of the open torus orbit. The mirror is then $X^\vee = (\C^*)^n$ with a superpotential $W$ given by a sum of monomials corresponding to the one-dimensional cones in the fan for $X$. 
 Choose a polarization $\cO_X(1)$ with corresponding moment
polytope $P$, a lattice polytope in $\R^n$. For each facet $F$ of $P$,
let $\nu(F)$ to be the primitive integer inward-pointing normal
vector, and let $\alpha(F)$ be such that $\langle \nu(F),x\rangle +
\alpha(F) = 0$ is the equation for the hyperplane containing $F$. Then
mirror Landau-Ginzburg model is given by
\begin{equation}
  \label{eq:hori-vafa}
  X^\vee = (\C^*)^n, \quad W = \sum_{F \text{ facet}} e^{-2\pi\alpha(F)}z^{\nu(F)},
\end{equation}
where $z^{\nu(F)}$ is a monomial in multi-index notation.

 In the case
where $X$ is toric but not necessarily Fano, a similar formula for the
mirror superpotential is expected to hold, which differs by the
addition of ``higher order'' terms \cite[Theorems 4.5, 4.6]{fooo-toric-I}.

The Hori--Vafa formula contains the case of the projective plane
$\CP^2$ with the toric boundary as anticanonical divisor. If $x,y,z$
denote homogeneous coordinates, then $D_{\text{toric}}$ can be taken to be
$\{xyz=0\}$, the union of the coordinate lines. We then have
\begin{equation}
  \label{eq:toric-lg-model}
    X^\vee_{\text{toric}} = (\C^*)^2, \quad
    W_{\text{toric}} = z_1+z_2+\frac{e^{-\Lambda}}{z_1z_2}
\end{equation}
where $\Lambda$ is a parameter that measures the cohomology class of
the K\"{a}hler form $\omega$ on $\CP^2$.

The example with which we are primarily concerned in this paper is
also $\CP^2$, but with respect to a different, non-toric boundary
divisor. Consider the meromorphic $(2,0)$--form $\Omega = dx\wedge dz
/ (xz-1)$, whose polar locus is the binodal cubic curve $D = \{xyz-y^3
= 0\}$. Thus $D = L \cup C$ is the union of a conic $C = \{xz-y^2 =
0\}$ and a line $\{y=0\}$. The construction of the mirror to this pair
$(\CP^2,D)$ is due to Auroux \cite{auroux07}, and we have
\begin{equation}
  \label{eq:main-lg-model}
  \begin{split}
    X^\vee = \{(u,v) \in \C^2 \mid uv \neq 1\}, \quad
    W = u + \frac{e^{-\Lambda}v^2}{uv-1}
  \end{split}
\end{equation}

A direct computation shows that both superpotentials
$W_{\text{toric}}$ and $W$ have the same critical values
% \eqref{eq:main-lg-model} has three critical points
% \begin{equation}
%   \label{eq:W-crit-pts}
%   \Crit(W) = \{(v = e^{\Lambda/3}e^{2\pi i (n/3)}, w = 1) \mid n = 0, 1, 2\},
% \end{equation}
% and corresponding critical values
\begin{equation}
  \label{eq:W-crit-vals}
  \{3e^{-\Lambda/3}e^{-2\pi i (n/3)} \mid n = 0, 1, 2\}.
\end{equation}
It is also easy to see that any regular fiber of
$W_{\text{toric}}^{-1}(c) \subset X^\vee_{\text{toric}}$ is a
three-times-punctured elliptic curve, while any regular fiber
$W^{-1}(c) \subset X^\vee$ is a twice-punctured elliptic curve. This
is an example of the general expectation that partially smoothing the
anticanonical divisor corresponds to partially compactifying the total
space of the Landau--Ginzburg model.

\subsubsection{Torus fibrations and affine manifolds}
\label{sec:intro-SYZ}

One justification that
\eqref{eq:toric-lg-model}--\eqref{eq:main-lg-model} are appropriate
mirrors is found in the Strominger--Yau--Zaslow proposal, which
expresses mirror symmetry geometrically in terms of dual torus
fibrations, a relationship also known as T--duality. Ideally, one
would expect that $X \setminus D$ and $X^\vee$ are dual special
Lagrangian torus fibrations over the same base $B$. When this holds
true, the mirror $X^\vee$ can be constructed as the \emph{complexified
  moduli space} of special Lagrangian fibers of $X\setminus D$
\cite{mclean-calibrated}, \cite{hitchin-slag},\cite[\S
2]{auroux07}. The superpotential $W$ can be expressed as a function
on this moduli space counting Maslov index two disks with boundary on
the Lagrangian fibers of $X\setminus D$, weighted by symplectic area
and the holonomy of a local system \cite{auroux07}.

% An important insight into the geometric nature of mirror symmetry is
% the proposal by Strominger--Yau--Zaslow (SYZ) \cite{syz} to view two mirror
% manifolds $X$ and $X^\vee$ as dual special Lagrangian torus fibrations
% over the same base $B$. This relationship is called T--duality.

For our purposes, a Lagrangian submanifold $L$ of a K\"{a}hler manifold
$X$ with meromorphic $(n,0)$--form $\Omega$ is called \emph{special}
of phase $\phi$ if $\im(e^{-i\phi}\Omega)|L = 0$. Obviously this only
makes sense in the complement of the polar locus $D$. The
infinitesimal deformations of a special Lagrangian submanifold are
given by $H^1(L;\R)$, and are unobstructed \cite{mclean-calibrated}. If $L \cong T^n$ is a
torus, $H^1(L;\R)$ is an $n$--dimensional space, and in good cases the
special Lagrangian deformations of $L$ are all disjoint, and form the
fibers of a fibration $\pi: X\setminus D\to B$, where $B$ is the global parameter
space for the deformations of $L$.

Assuming this, define the \emph{complexified moduli space} of
deformations of $L$ to be the space $\cM_L$ consisting of pairs
$(L_b,\cE_b)$, where $L_b = \pi^{-1}(b)$ is a special Lagrangian
deformation of $L$, and $\cE_b$ is a $U(1)$ local system on
$L_b$. There is an obvious projection $\pi^\vee : \cM_L \to B$ given
by forgetting the local system. The fiber $(\pi^\vee)^{-1}(b)$ is the
space of $U(1)$ local systems on the given torus $L_b$, which is
precisely the dual torus $L_b^\vee$. In this sense, the fibrations
$\pi$ and $\pi^\vee$ are dual torus fibrations, and the SYZ proposal
can be taken to mean that the mirror $X^\vee$ is precisely this
complexified moduli space: $X^\vee = \cM_L$. The picture is completed
by showing that $\cM_L$ naturally admits a complex structure $J^\vee$,
a K\"{a}hler form $\omega^\vee$, and a holomorphic $(n,0)$--form
$\Omega^\vee$. One finds that $\Omega^\vee$ is constructed from
$\omega$, while $\omega^\vee$ is constructed from $\Omega$, thus
expressing the interchange of symplectic and complex structures
between the two sides of the mirror pair. For details we refer the
reader to \cite{hitchin-slag},\cite[\S 2]{auroux07}.

However, this picture of mirror symmetry cannot be correct as stated,
as it quickly hits upon a major stumbling block: the presence of
singular fibers in the original fibration $\pi: X\setminus D \to
B$. These singularities make it impossible to obtain the mirror
manifold by a fiberwise dualization, and generate ``quantum
corrections'' that complicate the T-duality prescription. Attempts to
overcome this difficulty led to the work of Kontsevich and Soibelman
\cite{ks-torus-fibrations,ks-nonarchimedean} and Gross and Siebert
\cite{gs-affine-log-mirror,gs-log-degen-I,gs-log-degen-II} that
implements the SYZ program in an algebro-geometric context. In the
case of K3 surfaces, the work of Fukaya--Oh--Ohta--Ono
\cite{fooo-ch10} relates this problem to the wall--crossing of
holomorphic disk moduli spaces. It is also this difficulty which
motivates us to consider the case of $\CP^2$ relative to a binodal
cubic curve, where the simplest type of singularity arises.

In the case of $X$ with effective anticanonical divisor $D$, we can
see these corrections in action if we include the superpotential $W$
into the SYZ picture. As $W$ is to be a function on $X^\vee$, which is
naively $\cM_L$, $W$ assigns a complex number to each pair
$(L_b,\cE_b)$. This number is a count of holomorphic disks with
boundary on $L_b$, of Maslov index $2$, weighted by symplectic area
and the holonomy of $\cE_b$:
\begin{equation}
  \label{eq:disk-superpotential}
  W(L_b,\cE_b) = \sum_{\beta \in \pi_2(X,L_b), \mu(\beta) =
    2}n_\beta(L_b) \exp(-\int_\beta \omega) \hol(\cE_b,\partial \beta)
\end{equation}
where $n_\beta(L_b)$ is the count of holomorphic disks in the class
$\beta$ passing through a general point of $L_b$.

In the toric case, $X\setminus D \cong (\C^*)^n$, and we the special
Lagrangian torus fibration is simply the map $\Log: X\setminus D \to
\R^n$, $\Log(z_1,\dots,z_n) = (\log|z_1|,\dots,\log|z_n|)$. This
fibration has no singularities, and the above prescriptions work as
stated. In the toric Fano case, we recover the Hori--Vafa
superpotential \eqref{eq:hori-vafa}.

In the case of $\CP^2$ with the non-toric divisor $D$, the torus
fibration has one singular fiber, which is a pinched torus (a
focus-focus singularity). The above prescription breaks down: one
finds that the superpotential counting disks is not a continuous
function on the moduli space of special Lagrangians. This leads one to
redefine $X^\vee$ by breaking it into pieces and re-gluing so as to
make $W$ continuous, leading to \eqref{eq:main-lg-model} \cite{auroux07}. We find that $X^\vee$ also admits a
special Lagrangian torus fibration with one singular fiber.

% \subsection{Affine manifolds}
% \label{sec:intro-affine}

In general, the structure of a special Lagrangian torus fibration
$\pi: X\to B$ yields the structure of an \emph{tropical affine
  manifold} on the base $B$. This is a manifold with a distinguished
collection of affine coordinate charts, such that the transition maps
between affine coordinate charts lie in $\Aff_\Z(\R^n) = \GL(n,\Z)
\rtimes \R^n$, the group of affine linear transformations with
integral linear part. When singular fibers are present in the torus
fibration, we simply regard the affine structure as being undefined at
the singular fibers and call the resulting structure on the base a
\emph{singular tropical affine manifold}. In fact, the base $B$
inherits two affine structures, one from the symplectic form $\omega$,
and one from the holomorphic $(n,0)$--form $\Omega$. The former is
called the \emph{symplectic affine structure}, and the latter is
called the \emph{complex affine structure}.

Let us recall briefly how the local affine coordinates are
defined. For the symplectic affine structure, we choose a collection
of loops $\gamma_1,\dots,\gamma_n$ that form a basis of
$H_1(L_b;\Z)$. Let $X \in T_bB$ be a tangent vector to the base, and
take $\tilde{X}$ be any vector field along $L_b$ which lifts it. Then
\begin{equation}
  \alpha_i(X) = \int_0^{2\pi} \omega_{\gamma_i(t)}(\dot{\gamma}_i(t),\tilde{X}(\gamma_i(t)))\,dt
\end{equation}
defines a $1$-form on $B$: since $L_b$ is Lagrangian, the integrand is
independent of the lift $\tilde{X}$, and $\alpha_i$ only depends on
the class of $\gamma_i$ in homology. In fact, the collection
$(\alpha_i)_{i=1}^n$ forms a basis of $T^*_bB$, and there is a
coordinate system $(y_i)_{i=1}^n$ such that $dy_i = \alpha_i$; these
are the affine coordinates. This definition actually shows us that
there is a canonical isomorphism $T^*_bB \cong H_1(L_b;\R)$. This isomorphism induces an
integral structure on $T^*_bB$: $(T^*_bB)_\Z \cong H_1(L_b;\Z)$, which is
preserved by all transition functions between coordinate charts. Thus,
when an affine manifold arises as the base of a torus fibration in
this way, the structural group is reduced to $\Aff_\Z(\R^n) =
\GL(n,\Z) \rtimes \R^n$, the group of affine linear transformations
with integral linear part.

The complex affine structure follows exactly the same pattern, only
that we take $\Gamma_1,\dots,\Gamma_n$ to be $(n-1)$--cycles forming a
basis of $H_{n-1}(L_b;\Z)$, and in place of $\omega$ we use
$\im(e^{-i\phi}\Omega)$. Now we have an isomorphism $T^*_bB \cong
H_{n-1}(L_b;\R)$, or equivalently $T_bB \cong H_1(L_b;\R)$, which
induces the integral structure.

The affine manifolds we consider in this article satisfy a stronger
integrality condition, which requires the translational part of each
transition function to be integral as well. We use the term
\emph{integral affine manifold} to denote an affine manifold whose
structural group has been reduced to $\Aff(\Z^n) = \GL(n,\Z)\rtimes
\Z^n$. Such affine manifolds are ``defined over $\Z$'' and have an
intrinsically defined lattice of integral points.

For an integral affine manifold, it makes sense to speak of tropical
subvarieties. These are certain piecewise linear complexes contained
in $B$, which in some way correspond to holomorphic or Lagrangian
submanifolds of the total space of the torus fibration. Tropical
geometry has played a role in much work on mirror symmetry,
particularly in the program of Gross and Siebert, and closer to this
paper, in Abouzaid's work on mirror symmetry for toric varieties
\cite{abouzaid06,abouzaid09}. See \cite{tropical-AG} for a general
introduction to tropical geometry. Though most of the methods in this
paper are explicitly symplectic, tropical geometry appears in section
\ref{sec:fiber-superpotential}, where we compute the tropicalization
of the fiber of the superpotential, and in section \ref{sec:trop-fuk},
where a class of tropical curves corresponding to holomorphic polygons
is considered.

% \begin{remark} This notion of special Lagrangian is a weakening of the
%   original notion of special Lagrangian submanifold, which is a
%   submanifold calibrated by $\re(e^{-i\phi}\Omega)$, where $\Omega$ is
%   a parallel $(n,0)$--form on a manifold of $\SU(n)$ holonomy.
% \end{remark}

\subsubsection{Homological mirror symmetry}
\label{sec:intro-HMS}

The results on Floer cohomology that we prove fall under the heading
homological mirror symmetry (HMS) \cite{hms}, which holds that mirror
symmetry can be interpreted as an equivalence of categories associated to the complex or algebraic
geometry of $X$, and the symplectic geometry of $X^\vee$, and
vice--versa. The categories which are appropriate depend somewhat on
the situation, so let us focus on the case of the a manifold $X$ with
anticanonical divisor $D$, and its mirror Landau--Ginzburg model
$(X^\vee,W)$. 

Associated to $(X,D)$, we take the derived category of coherent
sheaves $D^b(\Coh X)$, while to $(X^\vee,W)$ we associate a Fukaya-type A$_\infty$-category
$\cF(X^\vee,W)$ whose objects are certain Lagrangian submanifolds of
$X^\vee$, morphism spaces are generated by intersection points, and the
A$_\infty$ product structures are defined by counting pseudo-holomorphic
polygons with boundary on a collection of Lagrangian submanifolds. Our
main reference for Floer cohomology and Fukaya categories is the book
of Seidel \cite{seidel-book}.

The superpotential $W$ enters the definition of $\cF(X^\vee,W)$ by
restricting the class of objects to what are termed \emph{admissible
  Lagrangian submanifolds}. In this paper, it will mean that we allow Lagrangian submanifolds that are not closed but which have boundary on a reducible hypersurface whose components are defined by setting one term of the superpotential equal to a constant. Historically, there have been several attempts to formulate the notion of admissibility. Originally, 
Kontsevich \cite{kontsevich-notes} and Hori--Iqbal--Vafa \cite{hori-iqbal-vafa} considered those Lagrangian
submanifolds $L$, not necessarily compact, which outside of a compact
subset are invariant with respect to the gradient flow of $\re(W)$. An
alternative formulation, due to Abouzaid \cite{abouzaid06,abouzaid09},
trades the non-compact end for a boundary on a fiber $\{W = c\}$ of
$W$, together with the condition that, near the boundary, the $L$ maps
by $W$ to a curve in $\C$. A further reformulation, which is more
directly related to the SYZ picture, replaces the fiber $\{W = c\}$
with the union of hypersurfaces $\bigcup_{\beta}\{z_\beta = c\}$, where $z_\beta$ is
the term in the superpotential \eqref{eq:disk-superpotential}
corresponding to the class $\beta \in \pi_2(X,\pi^{-1}(b))$, and
admissibility means that near $\{z_\beta = c\}$, $L$ maps by $z_\beta$
to a curve in $\C$. The admissibility condition we use in this paper is closest to this last formulation.

With these definitions, homological mirror symmetry amounts to an
equivalence of categories $D^\pi \cF(X^\vee,W) \to D^b(\Coh X)$, where
$D^\pi$ denotes the split-closed derived category of the
A$_\infty$--category. This piece of mirror symmetry has been addressed
many times
\cite{ako08,ako06,seidel01b,abouzaid06,abouzaid09,fltz-T-duality},
including results for the projective plane and its toric mirror.

However, in this article, we emphasize less the equivalences of
categories themselves, and focus more on geometric structures which
arise from a combination of the HMS equivalence with the SYZ picture.
When dual torus fibrations are present on the manifolds in a mirror
pair, one expects the correspondence between coherent sheaves and
Lagrangian submanifolds to be expressible in terms of a Fourier--Mukai
transform with respect to the torus fibration \cite{lyz}. In
particular, Lagrangian submanifolds $L \subset X^\vee$ that are
sections of the torus fibration correspond to line bundles on $X$, and
the Lagrangians $L(d)$ we consider are of this type.
 
In this context, our Theorem \ref{thm:triangle-products} can be
interpreted as yielding an embedding (at the cohomology level) of the
subcategory of $\cF(X(B),W)$ containing the Lagrangians $L(d)$ into
$D^b(\Coh \CP^2)$.

\subsubsection{Distinguished bases}
\label{sec:intro-bases}

The combination of SYZ and HMS also gives rise to the expectation
that, at least in favorable situations, the spaces of sections of
coherent sheaves on $X$ can be equipped with distinguished
bases. Suppose that $F: \Fuk(X^\vee) \to D^b(X)$ is a functor
implementing the HMS equivalence. Let $L_1, L_2 \in \Ob(\Fuk(X^\vee))$
be two objects of the Fukaya category supported by transversely
intersecting Lagrangian submanifolds. Then
\begin{equation}
  HF(L_1,L_2) \cong \RHom(F(L_1),F(L_2)).
\end{equation}
Suppose furthermore that the differential on the Floer cochain complex
$CF(L_1,L_2)$ vanishes, so that $HF(L_1,L_2) \cong CF(L_1,L_2)$. As
$CF(L_1,L_2)$ is defined to have a basis in bijection with the set of
intersection points $L_1\cap L_2$, one obtains a basis of
$\RHom(F(L_1),F(L_2))$ parametrized by the same set via the above
isomorphisms. If $\cF$ is some sheaf of interest, and by convenient
choice of $L_1$ and $L_2$ we can ensure $F(L_1) \cong \cO_X$ and
$F(L_2) \cong \cF$, then we will obtain a basis for $H^i(X,\cF)$.

When $E$ and $E^\vee$ are mirror dual elliptic curves, this phenomenon
is evident in the work of Polishchuk--Zaslow \cite{polishchuk-zaslow}
and especially M.~Gross \cite[Ch.~8]{d-branes-book}. Both $E$ and
$E^\vee$ may be written as special Lagrangian $S^1$-fibrations over
the same base $B \cong S^1$. The base has an integral affine structure
as $\R/\Z$. The Lagrangians $L(d) \subset E^\vee$ are sections of this
torus fibration with slope $d$, and their intersection points project
precisely to the fractional integral points of the base $B$.
\begin{equation}
  \label{eq:int-points}
  L(0) \cap L(d) \leftrightarrow B\left(\frac{1}{d}\Z\right) :=
  \frac{1}{d} \text{-integral points of } B
\end{equation}
% The notation $B((1/d)\Z)$ is in analogy with the functor-of-points
% notation.
Under HMS, we obtain $F(L(0)) = \cO_E$ and $F(L(d)) = \cO_E(d)$. the
basis of intersection points $L(0)\cap L(d)$ corresponds to a basis of
$\Gamma(E,\cO_E(d))$ consisting of theta functions.

Another illustration is the case of toric varieties and their mirror
Landau-Ginzburg models \cite{abouzaid06,abouzaid09}. In this case,
Abouzaid constructs a family of Lagrangian submanifolds $L(d)$ mirror
to the powers of the ample line bundle $\cO_X(d)$. These Lagrangian
submanifolds are topologically discs with boundary on a level set of
the superpotential, $W^{-1}(c)$ for some $c$. For $d>0$, the Floer
complex $CF^*(L(0),L(d))$ is concentrated in degree zero. Hence
\begin{equation}
  CF^0(L(0),L(d)) = HF^0(L(0),L(d)) = H^0(X,\cO_X(d)).
\end{equation}
The basis of intersection points $L(0)\cap L(d)$ corresponds to the
basis of characters of the algebraic torus $T = (\C^*)^n$ which appear
in the $T$-module $H^0(X,\cO_X(d))$. The same formula
\eqref{eq:int-points} is valid in the case of toric varieties, where
the base $B$ is the moment polytope $P$ of the toric variety
$X$. Abouzaid interprets $P$ as a subset of the base of the torus
fibration on $X^\vee = (\C^*)^n$ (the fibration given by the Clifford
tori), which moreover appears as a chamber bounded by a tropical
variety corresponding to a level set $W^{-1}(c)$ of the
superpotential. 

As is explained in section \ref{sec:fiber-superpotential}, these
features are also present in our case. The base $B$ is identified with
the base of the torus fibration on $\CP^2\setminus D$, with its
symplectic affine structure, and also identified as a subset of base
of the torus fibration on $X^\vee$ with its complex affine structure,
bounded by the tropicalization of a level set of the superpotential. We once again find a bijection between the intersection points of our Lagrangians $L(d)$ and the fractional integral points of the affine base.

Ongoing work of Gross--Hacking--Keel \cite{gross-hacking-keel} seeks
to extend these constructions to other manifolds, such as K3 surfaces,
using a purely algebraic and tropical framework. In this paper we are
concerned with extensions to cases that are tractable from the point
of view of symplectic geometry, although the tropical analog
of our results is described in section \ref{sec:trop-fuk}.

\subsection{Acknowledgments}
\label{sec:acknowledgements}
This paper is a version of my doctoral thesis. It represents research done
at MIT under the supervision of Denis Auroux, whose support and
generosity contributed greatly to the completion of this work. I thank
him, Paul Seidel, and Mohammed Abouzaid for suggestions that proved to be invaluable in its development. I thank Tom Mrowka for
his interest in this work and several helpful discussions. I thank the referees who provided comments and helpful suggestions that improved the paper in many ways.

This work was partially supported by NSF grants DMS-0600148 and
DMS-0652630. Part of it was done during visits to the University of
California at Berkeley and the Mathematical Sciences Research
Institute, whom I thank for their hospitality.

\section{The fiber of $W$ and its tropicalization}
\label{sec:fiber-superpotential}

\subsection{Torus fibrations on $\CP^2 \setminus D$ and its mirror}
\label{sec:torus-fibrations}

Let 
\begin{equation}
D = \{xyz-y^3 = 0\} \subset \CP^2  
\end{equation}
 be a binodal cubic
curve. We equip $\CP^2$ and $\CP^2 \setminus D$ with standard Fubini-Study symplectic forms. Both $\CP^2 \setminus D$ and its mirror $X^\vee = \{(u,v)\in
\C^2\mid uv \neq 1\}$ admit special Lagrangian torus fibrations. In
fact, these spaces are diffeomorphic, each being $\C^2$ minus a
conic. The torus fibrations are essentially the same on both sides,
but we are interested in the \emph{symplectic} affine structure
associated to the fibration on $\CP^2 \setminus D$ and the
\emph{complex} affine structure associated to $X^\vee$.

The construction is of the torus fibrations is taken from \cite[\S
5]{auroux07}. We have that
$\CP^2\setminus D$ is an affine algebraic variety with coordinates $x$
and $z$, where $xz\neq 1$. Hence we can define a map 
\begin{equation}
f: \CP^2\setminus D \to \C^* \quad f(x,z) = xz-1
\end{equation}
This map is a Lefschetz
fibration with critical point $(0,0)$ and critical value $-1$. The
fibers are affine conics, and the map is invariant under the $S^1$
action $e^{i\theta}(x,z) = (e^{i\theta}x,e^{-i\theta}z)$ that rotates
the fibers. Each fiber contains a distinguished $S^1$--orbit, namely
the vanishing cycle $\{|x| = |z|\}$. We can parametrize the other
$S^1$--orbits by the function $\delta(x,z)$ which denotes the signed
symplectic area between the vanishing cycle and the orbit through
$(x,z)$. The function $\delta$ is a moment map for the
$S^1$-action. Symplectic parallel transport in every direction
preserves the circle at level $\delta = \lambda$, and so by choosing
any loop $\gamma \subset \C^*$, and $\lambda \in (-\Lambda,\Lambda)$
(where $\Lambda = \int_{\CP^1} \omega$ is the area of a line), we
obtain a Lagrangian torus $T_{\gamma,\lambda} \subset \CP^2 \setminus
D$. If we let $T_{R,\lambda}$ denote the torus at level $\lambda$ over
the circle of radius $R$ centered at the origin in $\C^*$, we find
that $T_{R,\lambda}$ is special Lagrangian with respect to the form
\begin{equation}
\Omega = dx\wedge dz/(xz-1).
\end{equation}

The torus fibration on $X^\vee$ is essentially the same, except that
the coordinates $(x,z)$ are changed to $(u,v)$. For the rest of the
paper, we denote by 
\begin{equation}
w = uv-1
\end{equation} the quantity to which we project in
order to obtain the Lagrangian tori $T_{R,\lambda}$ (and later the
Lagrangian sections $L(d)$) as fibering over paths. For the time
being, and in order to enable the explicit computations in section
\ref{sec:tropicalization}, we will equip $X^\vee$ with the standard
symplectic form in the $(u,v)$--coordinates, so the quantity
\begin{equation}\delta(u,v) =|u|^2-|v|^2\end{equation}  is the standard moment map. In summary,
for $X^\vee$, we have 
\begin{equation}
T_{R,\lambda} = \{(u,v) \mid |w| = |uv-1| = R,
|u|^2-|v|^2 = \lambda\}.
\end{equation}

Each torus fibration has a unique singular fiber $T_{1,0}$, which is a
pinched torus.

Figure \ref{fig:torus-fibration} shows several fibers of the Lefschetz
fibration, with a Lagrangian torus that maps to a circle in the
base. The two marked points in the base represent a Lefschetz critical
value (filled-in circle), and a puncture (open circle).
\begin{figure}
\includegraphics[width=2.5in]{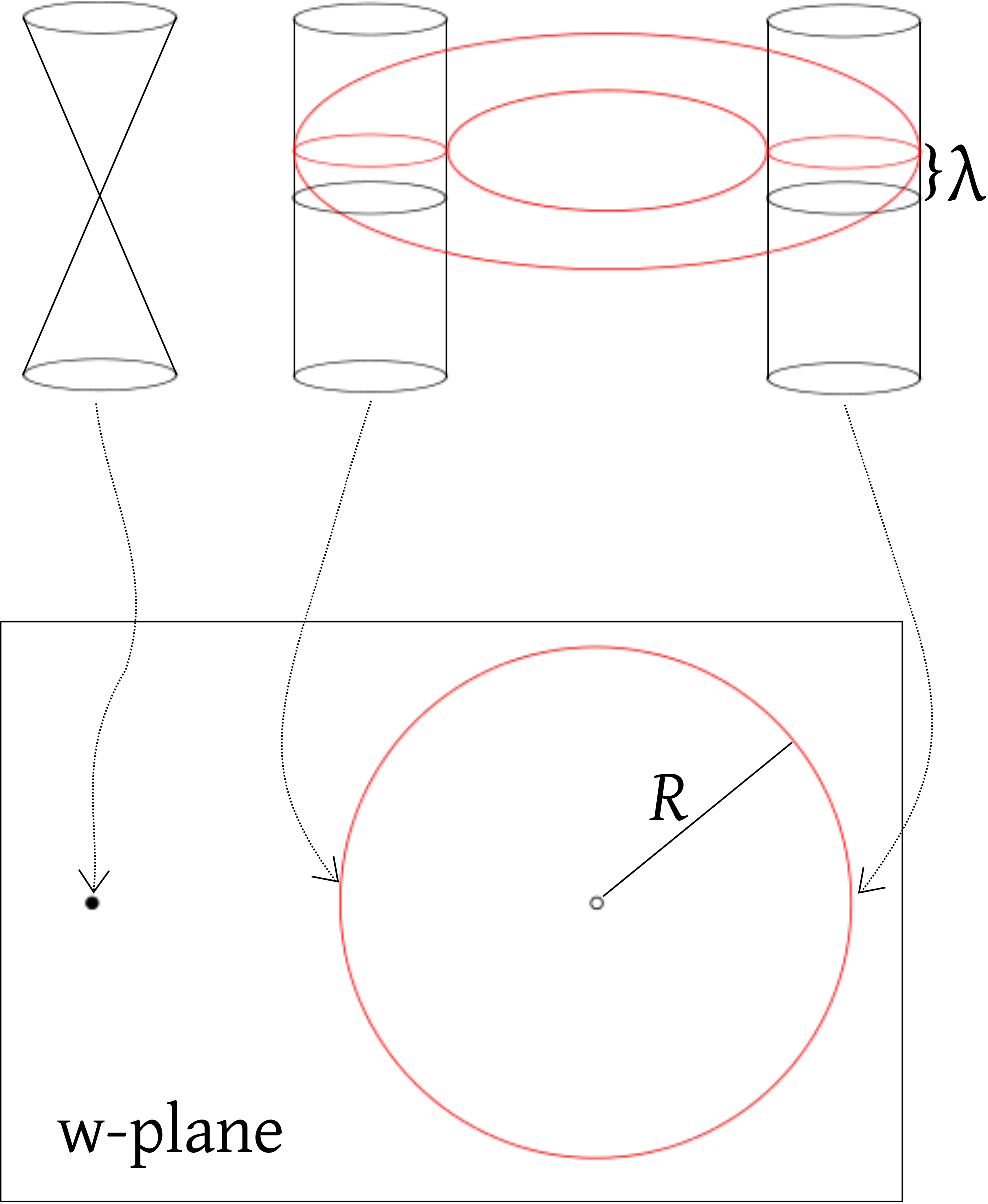}
\caption{The Lefschetz fibration with a torus that maps to a circle.}
\label{fig:torus-fibration}
\end{figure}

For the symplectic affine structure, the affine coordinates are
obtained by integrating the symplectic two-form $\omega$ over
one-cycles in the torus fibers to obtain one-forms on the base, which
may be integrated to functions. For the complex affine structure, we
instead integrate the imaginary part of the holomorphic volume form
$\im \Omega$ over $(n-1)$ cycles in the fiber to obtain one-forms on
the base, where in our case $n = 2$. See \cite[\S 2]{gross-syz-survey}
for an explanation of this construction.

We begin by describing the symplectic affine structure on the base $B$
of the torus fibration on $\CP^2 \setminus D$.  Recall that an
integral affine structure has a canonically defined local system of
integral tangent vectors. In the case when the affine manifold has
singularities, this local system may have monodromy around the
singular locus. The following result can be extracted from the
analysis of \cite[\S 5.2]{auroux07}.

\begin{proposition}
  The integral affine manifold $B$ is topologically a disk. It has one
  singular point, with monodromy conjugate to
  $\begin{pmatrix}1&0\\1&1\end{pmatrix}$. In affine coordinates, the
  boundary of $B$ consists of two line segments that are straight with
  respect to the affine structure. The corners are locally equivalent, by an affine linear transformation, to the standard quadrant $\R_{\geq 0}^2 \subset \R^2$.
\end{proposition}

\begin{proof}
  % This proposition can be extracted from the analysis in \cite[\S
  % 5.2]{auroux07}. 
 The symplectic affine coordinates are the
  symplectic areas of disks in $\CP^2$ with boundary on
  $T_{R,\lambda}$. Let $H$ denote the class of a line. We consider the cases $R > 1$ and $R < 1$. 

  On the $R > 1$ side, we take $\beta_1,\beta_2 \in
  H_2(\CP^2,T_{R,\lambda})$ to be the classes of two sections over the
  disk bounded by the circle of radius $R$ in the base, where
  $\beta_1$ intersects the $z$-axis and $\beta_2$ the $x$-axis. Then
  the torus fiber collapses onto line $\{y=0\}$ when
  \begin{equation}
    \langle [\omega], H-\beta_1-\beta_2 \rangle = 0.
  \end{equation}

  On the $R < 1$ side, we take $\alpha, \beta \in
  H_2(\CP^2,T_{R,\lambda})$, where $\beta$ is now the unique class of
  sections over the disk bounded by the circle of radius $R$, and
  $\alpha$ is the class of a disk connecting an $S^1$-orbit to the
  vanishing cycle within the conic fiber and capping off with the
  thimble. The torus fiber collapses onto the conic $\{xz-y^2 = 0\}$
  when
  \begin{equation}
    \langle [\omega], \beta \rangle = 0.
  \end{equation}

  The two sides $R > 1$ and $R < 1$ are glued together along the wall
  at $R = 1$, but
  the gluing is different for $\lambda > 0$ than for $\lambda < 0$,
  leading to the monodromy. Let us take $\eta = \langle
  [\omega],\alpha\rangle$ and $\xi = \langle [\omega],\beta\rangle$ as
  affine coordinates in the $R < 1$ region. We continue these across
  the $\lambda > 0$ part of the wall using correspondence between
  homology classes:
  \begin{equation}
    \begin{split}
      \alpha &\leftrightarrow \beta_1-\beta_2\\
      \beta &\leftrightarrow \beta_2\\
      H-2\beta-\alpha &\leftrightarrow H-\beta_1-\beta_2
      \end{split}
  \end{equation}
  Thus, in the $\lambda > 0$ part of the base, the conic appears as
  $\xi = 0$, while the line appears as 
  \begin{equation}
    0 = \langle [\omega],H - 2\beta - \alpha\rangle = \Lambda - 2\xi - \eta
  \end{equation}
  which is a line of slope of $-1/2$ with respect to the coordinates $(\eta,\xi)$. The pair of functions $\xi$ and $\Lambda -2\xi -\eta$ also form an affine coordinate system, and in this system the corner appears as a standard quadrant.

  In the $\lambda < 0$ part of the base, we instead use
    \begin{equation}
    \begin{split}
      \alpha &\leftrightarrow \beta_1-\beta_2\\
      \beta &\leftrightarrow \beta_1\\
      H-2\beta+\alpha &\leftrightarrow H-\beta_1-\beta_2
      \end{split}
  \end{equation}
  Hence in this region the conic appears as $\xi = 0$ again, while the
  line appears as 
  \begin{equation}
       0 = \langle [\omega],H - 2\beta + \alpha\rangle = \Lambda - 2\xi + \eta
  \end{equation}
  which is a line of slope $1/2$ with respect to the coordinates
  $(\eta,\xi)$. The pair of functions $\xi$ and $\Lambda -2\xi +\eta$ also form an affine coordinate system, and in this system the corner appears as a standard quadrant.
  
  The discrepancy between the two gluings represents the monodromy. As
  we pass from $\{R > 1, \lambda > 0\} \to \{R< 1, \lambda > 0\} \to
  \{R < 1,\lambda < 0\} \to \{R > 1, \lambda < 0\} \to \{R > 1,\lambda
  > 0\}$, the coordinates $(\eta,\xi)$ under go the transformation
  $(\eta,\xi) \to (\eta, \xi-\eta)$, which is indeed a simple shear.

\end{proof}

Figure \ref{fig:the-bigon} shows the affine manifold $B$. The marked
point is a singularity of the affine structure, and the dotted line is
a cut in the affine coordinates. The affine coordinates
$\eta,\xi,\psi$ are indicated. The function $\xi$ is undefined on the
cut below the singularity, while $\psi$ is undefined at points
directly above the singularity. Going around the singularity
counterclockwise, the monodromy of the local system of integral
tangent vectors $\begin{pmatrix}1&0\\1&1\end{pmatrix}$, which also
serves as the gluing map along the cut. The edges of the picture are
boundaries. The upper edge corresponds to points where the torus
collapses onto the line $\{y=0\}$. It has slope zero in this
picture. The lower edge corresponds to points where the torus
collapses onto the conic $\{xz-y^2 = 0\}$. On the left portion the
lower edge has slope $-1/2$, while on the right portion it has slope
$1/2$. The lower edge is actually straight with respect to the affine
structure, and the nontrivial gluing is what compensates for the
apparent bend.
\begin{figure}
\includegraphics[width=3in]{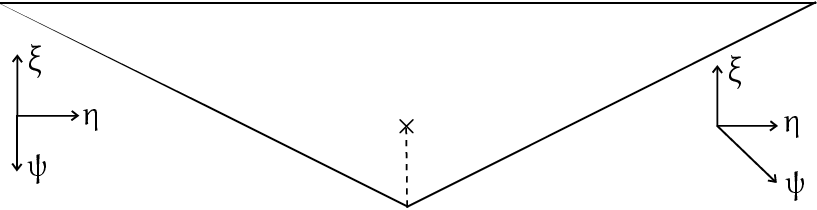}
\caption{The affine manifold $B$.}
\label{fig:the-bigon}
\end{figure}

On the mirror side, we compute the complex affine structure
(determined by the holomorphic volume form) on the base of the torus
fibration on $X^\vee$. The holomorphic volume form is
\begin{equation}
  \label{eq:vol-form-uv}
  \Omega = \frac{du\wedge dv}{uv-1} = \frac{du\wedge dv}{w}.
\end{equation}
Differentiating the defining equation $uv = 1 + w$ and substituting
gives the other formulas
\begin{equation}
 \label{eq:vol-form-uw}
 \Omega = \frac{du}{u}\wedge \frac{dw}{w}, \quad \text{when}\ u \neq 0,
\end{equation}
\begin{equation}
  \label{eq:vol-form-vw}
  \Omega = -\frac{dv}{v}\wedge \frac{dw}{w}, \quad \text{when}\ v \neq 0.
\end{equation}
The special Lagrangian fibration on $X^\vee$ to consider is constructed
in \cite{auroux07}. The fibers are the tori
\begin{equation}
  \label{eq:slags}
  T_{R,\lambda} = \{(u,v)\in X^\vee \mid |uv-1| = R, |u|^2-|v|^2 =
  \lambda\}, \quad (R,\lambda) \in (0,\infty)\times (-\infty,\infty),
\end{equation}
and the fiber $T_{1,0}$ is a pinched torus. Thus $(R,\lambda)$ are
coordinates on the base of this fibration. They are not affine
coordinates, which must be computed using $\im \Omega$. Due to the
simple algebraic form of this fibration, it is possible to find an
integral representation of the complex affine coordinates
explicitly.

\begin{proposition}
  \label{prop:affine-coords}
  The following functions are affine linear with respect to affine structure induced by $\im \Omega$.
  \begin{equation}
    \label{eq:affine-coords-uw}
    \begin{split}
      \eta &= \log|w| = \log R\\
      \xi &= \frac{1}{2\pi}\int_{T_{R,\lambda}\cap \{u \in \R_+\}}
      \log|u|\,
      d\arg(w)\\
      &= \frac{1}{2\pi}\int_0^{2\pi}
      \frac{1}{2}\log\left(\frac{\lambda +\sqrt{\lambda^2 + 4\cdot
            |1+Re^{i\theta}|^2}}{2}\right)\,d\theta\\
\psi&= \frac{1}{2\pi}\int_{T_{R,\lambda}\cap \{v\in \R_+\}} \log|v|\, d\arg(w)\\
      &= \frac{1}{2\pi}\int_0^{2\pi}
      \frac{1}{2}\log\left(\frac{-\lambda +\sqrt{\lambda^2 +
            4\cdot |1+Re^{i\theta}|^2}}{2}\right)\, d\theta
    \end{split}
  \end{equation}
  ($\eta$ is defined everywhere, $\xi$ is defined where $u \neq 0$, and $\psi$ is defined where $v \neq 0$).
  They satisfy the following relations on their common domain of definition.
  \begin{enumerate}
  \item At every point, $\eta$ is independent from $\xi$ and from $\psi$.
  \item In the subset where $R < 1$, the relation $\xi + \psi = 0$ holds.
  \item In the subset where $R > 1$, the relation $\xi + \psi = \eta$
    holds.
  \end{enumerate}
\end{proposition}

% \begin{proposition}
%   \label{prop:small-R-affine-coords}
%   In the subset of the base where $R < 1$, a set of affine coordinates is
%   \begin{equation}
%     \label{eq:affine-coords-uw}
%     \begin{split}
%       \eta &= \log|w| = \log R\\
%       \xi &= \frac{1}{2\pi}\int_{T_{R,\lambda}\cap \{u \in \R_+\}}
%       \log|u|\,
%       d\arg(w)\\
%       &= \frac{1}{2\pi}\int_0^{2\pi}
%       \frac{1}{2}\log\left(\frac{\lambda +\sqrt{\lambda^2 + 4\cdot
%             |1+Re^{i\theta}|^2}}{2}\right)\,d\theta
%     \end{split}
%   \end{equation}
%   Another set is
%   \begin{equation}
%     \label{eq:affine-coords-vw}
%     \begin{split}
%       \eta &= \log|w| = \log R\\
%       \psi&= \frac{1}{2\pi}\int_{T_{R,\lambda}\cap \{v\in \R_+\}} \log|v|\, d\arg(w)\\
%       &= \frac{1}{2\pi}\int_0^{2\pi}
%       \frac{1}{2}\log\left(\frac{-\lambda +\sqrt{\lambda^2 +
%             4\cdot |1+Re^{i\theta}|^2}}{2}\right)\, d\theta
%     \end{split}
%   \end{equation}
%   These coordinates satisfy
%   \begin{equation}
%     \label{eq:small-R-coord-relation}
%     \xi + \psi = 0.
%   \end{equation}
% \end{proposition}

\begin{proof}
  This is essentially straightforward so we only give an example of
  the computation.

  The general procedure for computing affine coordinates from the flux
  of the holomorphic volume form is as follows: we choose, over a
  local chart on the base, a collection of $(2n-1)$--manifolds
  $\{\Gamma_i\}_{i=1}^n$ in the total space $X$ such that the torus
  fibers $T_b$ intersect each $\Gamma_i$ in an $(n-1)$--cycle, and
  such that these $(n-1)$--cycles $T_b \cap \Gamma_i$ form a basis of
  $H_{n-1}(T_b;\Z)$. The affine coordinates $(y_i)_{i=1}^n$ are
  defined up to constant shift by the property that
  \begin{equation}
    y_i(b')-y_i(b) = \frac{1}{2\pi}\int_{\Gamma_i \cap \pi^{-1}(\gamma)} \im \Omega
  \end{equation}
  where $\gamma$ is any path in the local chart on the base connecting
  $b$ to $b'$. Because $\Omega$ is closed, this integral does not
  depend on the choice of $\gamma$.

  To get the coordinate $\eta$, start with the submanifold
  \begin{equation}
    \Gamma_1= \{w \in \R_+\}
  \end{equation}

  % To get the coordinate system \eqref{eq:affine-coords-uw}, we start
  % with the submanifolds defined by
  % \begin{equation}
  %   \Gamma_1= \{w \in \R_+\},\ \Gamma_2 = \{u \in \R_+\}
  % \end{equation}
  
  The intersection $\Gamma_1 \cap T_{R,\lambda}$ is a loop on
  $T_{R,\lambda}$. The function $\arg(u)$ gives a coordinate on this
  loop (briefly, $w\in \R_+$ and $|w| = R$ determine $uv$, along with
  $|u|^2 - |v|^2 = \lambda$ this determines the $|u|$ and $|v|$; the
  only parameter left is $\arg(u)$ since $\arg(v) = - \arg(u)$), and
  we declare the loop to be oriented so that $-d\arg(u)$ restricts to
  a positive volume form on it. Using \eqref{eq:vol-form-uw} we see
  \begin{equation}
    \label{eq:im-vol-form-uw}
    \im \Omega = d\arg(u) \wedge d\log|w| + d\log|u| \wedge d\arg(w) 
  \end{equation}
  Using the fact that $\arg(w)$ is constant on $\Gamma_1$, we see
  that for any path $\gamma$ in the subset of the base where $R < 1$
  connecting $b = (R,\lambda)$ to $b'=(R',\lambda')$, we have 
  \begin{equation}
    \int_{\Gamma_1\cap \pi^{-1}(\gamma)} \im \Omega =
    \int_{\Gamma_1\cap \pi^{-1}(\gamma)} d\arg(u) \wedge d\log|w|.
  \end{equation}
  But $d\arg(u) \wedge d\log|w| = d(-\log|w|\,d\arg(u)))$, so
  the integral above equals
  \begin{equation}
    \int_{\Gamma_1 \cap T_{b'}} -\log|w|\,d\arg(u) -
    \int_{\Gamma_1\cap T_b} -\log|w|\,d\arg(u) = 2\pi(\log R' - \log R)
  \end{equation}
  (the minus signs within the integrals are absorbed by the
  orientation convention for $\Gamma_1 \cap T_b$). Thus $\eta = \log
  R$ is the affine coordinate corresponding to $\Gamma_1$.

  The same idea applied to $\Gamma_2 = \{u \in \R_+\}$ yields the affine coordinate 
\begin{equation}
\xi = \frac{1}{2\pi} \int_{\Gamma_2 \cap T_b} \log|u|\,d\arg(w)
\end{equation}
To arrive at the second formula for $\xi$, we must solve for $|u|$
  in terms of $R,\lambda$, and $\theta = \arg(w)$. The equations $uv =
  1 + Re^{i\theta}$ and $|u|^2 - |v|^2 = \lambda$ imply $|u|^4 -
  \lambda|u|^2 = |1+Re^{i\theta}|^2$. Solving for $|u|^2$ by the
  quadratic formula and taking logarithms gives the result.

The formulas for $\psi$ are obtained by applying the same method to $\Gamma'_2 = \{v \in \R_+\}$. 

  % The intersection $\Gamma_2 \cap T_{R,\lambda}$ is a loop on
  % $T_{R,\lambda}$; together with the loop $\Gamma_1 \cap
  % T_{R,\lambda}$ it gives a basis of $H_1(T_{R,\lambda};\Z)$. The
  % function $\arg(w)$ gives a coordinate on this loop, and we orient
  % the loop so that $d\arg(w)$ restricts to a positive volume
  % form. Using \eqref{eq:im-vol-form-uw}, the fact that $\arg(u)$ is
  % constant on $\Gamma_2$, and the same reasoning as above, we see that
  % \begin{equation}
  %   \int_{\Gamma_2\cap \pi^{-1}(\gamma)}\im \Omega = \int_{\Gamma_2
  %     \cap T_{b'}} \log|u|\,d\arg(w) - \int_{\Gamma_2 \cap T_b} \log|u|\,d\arg(w).
  % \end{equation}
  % Thus $\xi =\frac{1}{2\pi} \int_{\Gamma_2 \cap T_b} \log|u|\,d\arg(w)$ is the
  % affine coordinate correspond to $\Gamma_2$.

  % To get the coordinate system \eqref{eq:affine-coords-vw}, we must
  % consider now the subset $\Gamma_2' = \{v\in \R_+\}$. This
  % intersects each fiber in a loop along which $\arg(w)$ is once again
  % a coordinate. Due to the minus sign in \eqref{eq:vol-form-vw}, we
  % must orient the loop so that $-d\arg(w)$ is a positive volume form
  % in order to get the formula we want. Otherwise, the derivation of
  % $\psi$ is entirely analogous to the the derivation of $\xi$ from $\Gamma_2$.

  To prove the linear relations, we find that $\xi + \psi$ reduces by the law of logarithms and the Cauchy integral formula to
  \begin{equation}
    \frac{1}{2\pi}\int_0^{2\pi} \log|1+Re^{i\theta}|\,d\theta = \begin{cases}0, & R<1\\ \log R, &R > 1\end{cases}
  \end{equation}

\end{proof}

Proposition \ref{prop:affine-coords} determines the monodromy around
the singular point (at $\eta = \xi = \psi = 0$) of the base, and show
that the affine structure is in fact integral. Once again, the
monodromy is a shear.

\subsection{The topology of the map $W$}
\label{sec:topology-W}

A direct computation shows that the superpotential $W$ given by
\eqref{eq:main-lg-model} has three critical points
\begin{equation}
  \label{eq:W-crit-pts}
  \Crit(W) = \{(v = e^{\Lambda/3}e^{2\pi i (n/3)}, w = 1) \mid n = 0, 1, 2\},
\end{equation}
and corresponding critical values
\begin{equation}
  \label{eq:W-crit-vals}
  \Critv(W) = \{3e^{-\Lambda/3}e^{-2\pi i (n/3)} \mid n = 0, 1, 2\}.
\end{equation}
As expected, $\Critv(W)$ is the set of eigenvalues of quantum
multiplication by $c_1(T\CP^2)$ in $QH^*(\CP^2)$, that is,
multiplication by $3h$ in the ring $\C[h]/\langle h^3 = e^{-\Lambda}\rangle$.

\begin{proposition}
  Any regular fiber $W^{-1}(c) \subset X^\vee$ is a twice-punctured elliptic curve.  
\end{proposition}

\begin{proof}
  In the $(u,v)$ coordinates, $W^{-1}(c)$ is defined by the equation
  \begin{equation}
    u + \frac{e^{-\Lambda}v^2}{uv-1} = c,
  \end{equation}
  \begin{equation}
    u(uv-1) + e^{-\Lambda}v^2 = c(uv-1).
  \end{equation}
  This is an affine cubic plane curve, and it is disjoint from the
  affine conic $\{uv - 1=0\}$. It is smooth as long as $c$ is a regular value.

  The projective closure of $W^{-1}(c)$ in $(u,v)$ coordinates is given
  by the homogeneous equation (with $\xi$ as the third coordinate)
  \begin{equation}
    u(uv-\xi^2) + e^{-\Lambda}v^2\xi = c\xi(uv-\xi^2).
  \end{equation}
  This is a projective cubic plane curve, hence elliptic, and it
  intersects the line at infinity $\{\xi = 0\}$ when $u^2v = 0$. So it
  is tangent to the line at infinity at $(u:v:\xi) = (0:1:0)$ and
  intersects it transversely at $(u:v:\xi) = (1:0:0)$. Hence the
  affine curve is the projective curve minus these two points.
\end{proof}

\begin{remark}
  The function $W$ above is to be compared to the ``standard''
  superpotential for $\CP^2$, namely,
  \begin{equation}
    \label{eq:toric-W}
    W = x + y + \frac{e^{-\Lambda}}{xy}
  \end{equation}
  corresponding to the choice of the toric boundary divisor, a union
  of three lines, as anticanonical divisor. This $W$ has the same
  critical values, and its regular fibers are all thrice-punctured
  elliptic curves. Hence smoothing the anticanonical divisor to the
  union of a conic and a line corresponds to compactifying one of the
  punctures of $W^{-1}(c)$. 
\end{remark}

\subsection{Tropicalization in a singular affine structure}
\label{sec:tropicalization}

Our goal is to find the same affine manifold $B$ (that comes from
symplectic structure of $\CP^2\setminus D$) embedded in the base of
the torus fibration on $X^\vee$, equipped with the \emph{complex}
affine structure. We shall see that it is obtained as a bounded
chamber inside a particular tropical curve, the tropicalization of the
fiber of $W$. 

 For some real number $\epsilon > 0$, consider the curve
$W^{-1}(e^{\epsilon \Lambda})$:
\begin{equation}
  \label{eq:fiber-to-trop}
  W = u + \frac{e^{-\Lambda}v^2}{w} = e^{\epsilon \Lambda}
\end{equation}
 The tropicalization corresponds to the limit $\Lambda \to \infty$, or
$t := e^{-\Lambda} \to 0$. 
The amoeba of $A(W^{-1}(t^{-\epsilon}))$ is the image of the curve in the affine base of the torus fibration on $X^\vee$.
We want to produce a tropical curve in the base that reflects the asymptotic geometry of these amoebas $A(W^{-1}(t^{-\epsilon}))$ as $t \to 0$. In the standard situation, this is done by rescaling the amoebas by $\log t$ and taking the Hausdorff limit. Our situation is not standard because the affine manifold in which our tropical curve is to live has a singularity. Thus, none of the standard tropicalization techniques \cite{tropical-AG} apply directly. Since we do not know of any general theory of tropicalization when the affine structure is singular, we will here content ourselves with an ad hoc method involving explicit computation in two coordinate charts, and checking that the results fit together in the singular affine structure.

In the standard picture of tropicalization, one considers a family
of subvarieties of an algebraic torus $V_t \subset (\C^*)^n$. The map
$\Log: (\C^*)^n \to \R^n$ given by $\Log(z_1,\dots,z_n) =
(\log|z_1|,\dots,\log|z_n|)$ projects these varieties to their amoebas
$\Log(V_t)$, and the rescaled limit of these amoebas is the
tropicalization of the family $V_t$. The tropicalization is also given
as the non-archimedean amoeba of the defining equation of $V_t$.

The map $\Log: (\C^*)^n \to \R^n$ is special Lagrangian fibration. Its
fibers are the tori $\{|z_1| = r_1, \dots,|z_n| = r_n\}$. These tori
are Lagrangian with respect to the standard symplectic form, and they
are special with respect to the holomorphic volume form
\begin{equation}
  \label{eq:toric-vol-form}
  \Omega_{\text{toric}} = \frac{dz_1}{z_1} \wedge \cdots \wedge \frac{dz_n}{z_n},
\end{equation}
which has logarithmic poles along the coordinate hyperplanes in
$\C^n$. The complex affine coordinates are $\log |z_i|$.

The torus fibration on $X^\vee$ is approximated by this standard
structure as follows. Equation \eqref{eq:vol-form-uw} shows that in the
$(w,u)$ coordinates (where $u \neq 0$), the holomorphic volume form is
standard. If the special Lagrangian fibration were also standard, the
affine coordinates would be $(\log|w|,\log|u|)$. Proposition
\ref{prop:affine-coords} shows that, while $\eta = \log|w|$ is still
an affine coordinate (reflecting the fact that there is still an
$S^1$-symmetry), the other affine coordinate $\xi$ is the average
value of $\log|u|$ along a loop in the fiber. We also see that as
$|\lambda|$ becomes large, the approximation $\xi \approx \log|u|$
holds with increasing accuracy. Similarly, in the $(w,v)$ coordinates
(when $v \neq 0$), the holomorphic volume form is standard, and $\eta
= \log|w|$ and $\psi \approx \log|v|$ form affine coordinates.

We shall see that as $t = e^{-\Lambda} \to 0$, the amoebas
$A(W^{-1}(t^{-\epsilon}))$ move farther away from the
singularity, where the approximations $\xi \approx \log|u|$ and $\psi \approx
\log|v|$ hold with increasing accuracy, while $\eta = \log|w|$ holds exactly
everywhere.

% In the case at hand, we have a pencil of curves $W^{-1}(c)$ in $X^\vee \cong \C^2
% \setminus \{uv - 1 = 0\}$. The total space $X^\vee$ must now play the role that
% $(\C^*)^2 = \C^2 \setminus \{uv = 0\}$ plays in ordinary tropical
% geometry. 

% Now we consider the fiber $W^{-1}(c)
% \subset X^\vee$ of the superpotential, where the torus fibration on
% $X^\vee$ plays the role of the log map.

% This means that at the level
% of tropical amoebas, we can actually identify the tropical coordinates
% $\xi$ and $\log|u|$, $\psi$ and $\log|v|$, in appropriate regions on
% the base of the torus fibration, 

The approximations $\xi \approx \log|u|$ and $\psi \approx \log|v|$ and the identity $\eta = \log|w|$ suggest the approach to finding the tropicalization. We represent the equation $W = u + e^{-\Lambda}v^2/w = e^{\epsilon \Lambda}$ as a polynomial equation in the coordinates $(w,u) \in (\C^*)^2$. We then compute the corresponding tropical curve, using the standard procedures \cite{tropical-AG}, and we plot the result in the affine plane whose coordinates are $(\log|w|, \log|u|)$. This is shown in Figure \ref{fig:tropical-fiber}(a). We repeat the process in the coordinate system $(w,v)$, and plot the resulting tropical curve in the affine plane whose coordinates are $(\log|w|, \log|v|)$. This is shown in Figure \ref{fig:tropical-fiber}(b). We then transfer these tropical curves into the singular affine manifold by simply identifying $\xi$ with $\log|u|$ and $\psi$ with $\log|v|$. 

We now observe that the two curves actually match up, at least away from the vertical line passing through the singularity, and we claim that the result is as depicted in Figure \ref{fig:tropical-fiber}(c). This is evident in comparing the left-hand portions of parts (a) and (b) of Figure \ref{fig:tropical-fiber}. On the right-hand portion, we must take into account that the coordinate $\psi \approx \log|v|$ has been affected by a shear in passing from (b) to (c).

 There is apparently a problem along the vertical line passing through the singularity, since the curves depicted in (a) and (b) have vertical legs there, which we claim do not appear in (c). Our reasoning is this: the coordinate system $(w,u) \in (\C^*)^2$ only covers the locus where $u \neq 0$, which is also where the function $\xi$ is well defined. Thus we cannot expect the curve in (a) to be valid near the vertical going down from the singularity, it is valid above the singularity. Conversely, the figure in (b) is only valid in the region below the singularity. Thus the extra vertical legs are illusory. 

We call the resulting tropical curve $T_\epsilon$.

\begin{proposition}
  \label{prop:tropical-fiber}
  For $\epsilon > 0$, $T_\epsilon$ is a trivalent graph with two
  vertices, a cycle of two finite edges, and two infinite edges.

%    Furthermore, this graph is the Hausdorff limit of the amoebas $A_t(W^{-1}(t^\epsilon))$.
\end{proposition}

% Figure \ref{fig:tropical-fiber} shows the tropicalization of the fiber
% of $W$ for $\epsilon > 0$. The horizontal coordinate is $\eta = \log|w|$, and the vertical coordinate is $\xi \approx \log|u|$. The marked point is the singularity of the
% affine structure, and the dotted line is a cut in the affine
% coordinates. The lower edge of the figure is in fact straight. 
% In the case where $(-1/3) < \epsilon < 0$, the horizontal edge moves below the singular point of the affine structure, and so it is no longer straight. This creates a new vertex in the middle of this edge, and a new edge going upwards in order to make the graph balanced. Since graph can only have two infinite ends (being the tropicalization of a curve with two punctures), this edge will terminate at the singular point.
\begin{figure}
\includegraphics[width=5.75in]{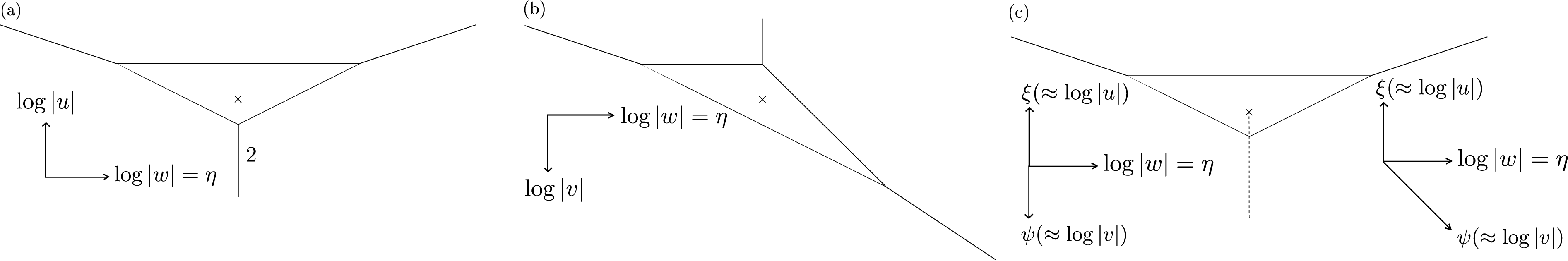}
\caption{The tropical fiber of $W$ for $\epsilon > 0$. (a) $(w,u)$ coordinates. (b) $(w,v)$ coordinates. (c) The tropical fiber $T_\epsilon$.}
\label{fig:tropical-fiber}
\end{figure}

\begin{proposition}
  \label{prop:tropical-fiber-chamber}
  For $\epsilon > 0$, the complement of $T_\epsilon$ has a bounded
  component that is an integral affine manifold with singularities
  that is isomorphic, after rescaling, to the base $B$ of the special Lagrangian
  fibration on $\CP^2 \setminus D$ with the affine structure coming
  from the symplectic form.
\end{proposition}

\begin{remark}
 The topology of the tropical curve $T_\epsilon$ corresponds to that of a
  twice-punctured elliptic curve.
\end{remark}

\begin{remark}
  We can apply the same patching procedure for $\epsilon$ in the range $(-1/3) < \epsilon < 0$. This changes Figure \ref{fig:tropical-fiber} so that the triangles in (a) and (b) become smaller and they lie entirely below point marked with $\times$. Thus, assuming that (a) is valid above the singularity, while (b) is valid below the singularity, we find that $T_\epsilon$ is essentially the same as what appears in (b), but that the vertical leg must terminate at the singularity, in order to be consistent with the fact that in (a) no part of the curve appears above the singularity. Thus we find that $T_\epsilon$ is a trivalent graph with
  three vertices, a cycle of three finite edges, two infinite edges
  and one edge connecting a vertex to the singular point of the
  affine structure.

  We note that the value $\epsilon = -1/3$ corresponds to the critical values of
  $W$.
  
\end{remark}

\begin{remark}
  This proposition is another case of the phenomenon, described in
  Abouzaid's paper \cite{abouzaid06}, that for toric varieties, the
  bounded chamber of the fiber of the superpotential is isomorphic to
  the moment polytope. In the general case of a manifold $X$ with
  effective anticanonical divisor $D$, the boundary of the symplectic
  affine base of the torus fibration on $X\setminus D$ corresponds to
  a torus fiber collapsing onto $D$, a particular class of holomorphic
  disks having vanishing area, and the corresponding term of the
  superpotential having unit norm. On the other hand, the
  tropicalization of the fiber of the superpotential has some parts
  corresponding to one of the terms having unit norm, and it is
  expected that these bound a chamber which is isomorphic to the base
  of the original torus fibration.
\end{remark}

\section{Symplectic constructions}
\label{sec:main-construction}

Let $B$ the affine manifold which is the bounded chamber of the
tropicalization of the fiber of $W$, constructed in the previous
section. In this section we construct a symplectic manifold $X(B)$,
which is a torus fibration over $B$, together with a Lefschetz
fibration $w: X(B) \to X(I)$, where $X(I)$ is an annulus. We also
construct a collection $\{L(d)\}_{d\in \Z}$ of Lagrangian
submanifolds. These submanifolds are sections of the torus fibration
$X(B) \to B$, and they fiber over paths with respect to the Lefschetz
fibration $w: X(B)\to X(I)$. The Lagrangians have an admissibility
property governing their behavior at the boundary of
$X(B)$. Corresponding to the two sides of $B$, and hence to the two
terms of $W = u + e^{-\Lambda}v^2/(uv-1)$, we have horizontal boundary
faces $\partial^h X(B)$, along each of which the symplectic connection
for the Lefschetz fibration defines a foliation. Choosing a leaf of
the foliation on each face defines a boundary condition (corresponding
to the fiber of $W$) for our Lagrangian submanifolds.

The motivation for these constructions is existence of the map
$w = uv-1: X^\vee \to \C^*$, which is a Lefschetz fibration with
general fiber an affine conic and a single critical value. The tori in
the SYZ fibration considered in section \ref{sec:fiber-superpotential}
fiber over loops in this projection, so it is natural to attempt to
use it to understand as much of the geometry as possible. In
particular it will allow us to apply the techniques of
\cite{seidel-book}, \cite{seidel01a}, \cite{seidel01b}.

\subsection{Symplectic monodromy associated to a Hessian metric}
\label{sec:monodromy-hessian}

Let $B$ be a two-dimensional affine manifold with affine coordinates
$(\eta,\xi)$ that embed $B$ as a subset of $\R^2$. In this subsection
$B$ does not have any singularities. Suppose that $\eta: B \to \R$ is
a submersion over some interval $I \subset \R$, and that the fibers of
this map are connected intervals. For our purposes, we consider the
case where $B$ is a quadrilateral, bounded on two opposite sides by
line segments of constant $\eta$ (the \emph{vertical} boundary
$\partial^v B$), and on the other two sides by line segments that are
transverse to the projection to $\eta$ (the \emph{horizontal} boundary
$\partial^h B$).

%This setup is a tropical model of a Lefschetz fibration. 

Associated to $B$ and $I$, we define complex manifolds $X(B)$ and
$X(I)$. The space $X(B)$ is the subset of the complex torus $(\C^*)^2$
with coordinates $w$ and $z$ such that $(\eta,\xi) =
(\log|w|,\log|z|)$ lies in $B$, and $X(I)$ is the subset of $(\C^*)$
with coordinate $w$ such that $\eta = \log|w|$ lies in $I$. Thus
$X(I)$ is an annulus, and we have a map $w : X(B) \to X(I)$, which is
a non-singular fibration with fibers isomorphic to annuli. Philosophically speaking, the map
$\eta: B \to I$ is a tropical model of the map $w: X(B)\to X(I)$.

In this situation, the most natural way to prescribe a K\"{a}hler
structure on $X(B)$ is through a Hessian metric on the base $B$. This
is a metric $g$ such that locally $g = \Hess K$ for some function $K:
B\to \R$, where the Hessian is computed with respect to an affine
coordinate system. If $\pi: X(B) \to B$ denotes the projection, then
$\phi = K\circ \pi$ is a K\"{a}hler potential on $X(B)$, and the
positivity of $g = \Hess K$ corresponds to the positivity of the real
closed $(1,1)$-form $\omega = -dd^c \phi$. Explicitly, if $y_1,\dots,
y_n$ are affine coordinates corresponding to complex coordinates
$z_1,\dots,z_n$ on $(\C^*)^n$, then
\begin{equation}
  \label{eq:hessian-metric}
  g = \sum_{i,j=1}^n \frac{\partial^2 K}{\partial y_i\partial y_j}dy_idy_j
\end{equation}
\begin{equation}
  \label{eq:hessian-kahler-form}
  \omega = -dd^c \phi = \frac{\sqrt{-1}}{2}\sum_{i,j=1}^n  \frac{\partial^2 K}{\partial y_i\partial y_j}\frac{dz_i}{z_i} \wedge \frac{d\bar{z}_j}{\bar{z}_j}
\end{equation}

Any such K\"{a}hler structure is invariant under the $S^1$--action
$e^{i\theta}(z,w) = (e^{i\theta}z,w)$ that rotates the fibers of the
map $w: X(B)\to X(I)$.

We now have a K\"{a}hler structure on $X(B)$ such that the fibration
$w: X(B) \to X(I)$ has symplectic fibers. Thus there is a symplectic
connection on this fibration whose horizontal subspaces are the
symplectic orthogonal spaces to the fibers. This connection defines a
notion of symplectic parallel transport along paths in the base
$X(I)$, and symplectic monodromy around loops in the base. Since our
fibers have boundary we must show that the symplectic parallel
transport preserves the boundary. Given that, symplectic parallel
transport defines symplectomorphisms between the fibers.

We now compute the symplectic connection. Let $X \in T_{(z,w)}X(B)$
denote a tangent vector. Let $Y \in \ker dw$ denote the general
vertical vector. The relation defining the horizontal distribution is
$\omega(X,Y) = 0$, or,
\begin{equation}
\begin{split}
 0 &=  \left\{K_{\eta\eta}\frac{dw}{w} \wedge \frac{d\bar{w}}{\bar{w}} + K_{\eta\xi}\left(\frac{dw}{w}
 \wedge \frac{d\bar{z}}{\bar{z}} + \frac{dz}{z} \wedge \frac{d\bar{w}}{\bar{w}}\right) + K_{\xi\xi}
   \frac{dz}{z} \wedge \frac{d\bar{z}}{\bar{z}}\right\}(X,Y)\\
&= K_{\eta\xi}\left(\frac{dw(X)}{w}\frac{d\bar{z}(Y)}{\bar{z}} -
  \frac{d\bar{w}(X)}{\bar{w}}\frac{dz(Y)}{z}\right) +
K_{\xi\xi}\left(\frac{dz(X)}{z}\frac{d\bar{z}(Y)}{\bar{z}}-
  \frac{dz(Y)}{z}\frac{d\bar{z}(X)}{\bar{z}}\right)\\
&= \left(K_{\eta\xi}\frac{dw(X)}{w} + K_{\xi\xi}
  \frac{dz(X)}{z}\right)\frac{d\bar{z}(Y)}{\bar{z}} - \text{ complex conjugate}
\end{split}
\end{equation}
Since $dz(Y)$ can have any phase, this shows that the quantity in
parentheses on the last line must vanish:
\begin{equation}
  \label{eq:symplectic-connection}
  d\log z (X) = -\frac{K_{\eta\xi}}{K_{\xi\xi}}d\log w (X)
\end{equation}
Tropically, this formula has the following interpretation: In the
$(\eta,\xi)$ coordinates, the vertical tangent space is spanned by the
vector $(0,1)$. The $g$-orthogonal to this space is spanned by the
vector $(K_{\xi\xi}, -K_{\eta\xi})$, whose slope with respect to the
affine coordinates is the factor $-K_{\eta\xi}/K_{\xi\xi}$ appearing in the
formula for the connection.

\begin{lemma}
  \label{lem:monodromy-hessian}
   Let $K$ be such that  $\partial^h B$ is $g$-orthogonal to the fibers of the map $\eta: B\to I$, and assume that the slope $\sigma_F$ of each boundary face $F$ of $B$ is rational. Then 
   \begin{enumerate}
   \item parallel transport around the loop $\{|w|= R\}$ acts on the fiber over $w = R$ by rotating each circle of constant $|z|$ through a phase $2\pi (-K_{\eta\xi}/K_{\xi\xi})$, and 
   \item for each boundary face $F$, we have $2\pi(-K_{\eta\xi}/K_{\xi\xi}) = 2\pi\sigma_F$, the part of $\partial^h X(B)$ lying over $F$ is foliated by multi-sections of the fibration that are horizontal with respect to the symplectic connection.
   \end{enumerate}
\end{lemma}

\begin{proof}
Consider the parallel transport of the connection around the loop $\{|w|
= R\}$, which is a generator of $\pi_1(X(I))$.
%This loop cannot be seen tropically.
As $w$ traverses the path $R\exp(it)$, the initial
condition $(z,w) = (r\exp(i\theta), R)$ generates the solution $(r\exp
(i \theta + (-K_{\eta\xi}/K_{\xi\xi})it), R\exp(it))$, where the
expression $-K_{\eta\xi}/K_{\xi\xi}$ is constant along the solution
curve. As a self-map of the fiber over $w = R$, this monodromy
transformation maps circles of constant $|z|$ to themselves, but
rotates each by the phase $2\pi (-K_{\eta\xi}/K_{\xi\xi})$.

We now consider the behavior of the symplectic connection near the
horizontal boundary $\partial^h X(B)$. Let $F$
be a component of $\partial^h B$. Since $F$ is a straight line segment
$g$-orthogonal to the fibers of $\eta$, the function
$-K_{\eta\xi}/K_{\xi\xi}$ is constant on $F$ and equal to its slope,
which we denote $\sigma = \sigma_F$. This slope is rational by assumption. The part of $X(B)$ lying over $F$ is defined by the
condition $\log|z| = \sigma \log|w| + C$. Let $w = w_0\exp(\rho(t) + i
\phi(t))$ describe an arbitrary curve in the base annulus $X(I)$. If
$(z_0,w_0)$ is an initial point that lies over $F$, then
\begin{equation}
  z = z_0\exp\left\{\sigma(\rho(t)+i\phi(t))\right\}, w = w_0\exp(\rho(t)+i\phi(t))
\end{equation}
is a path in $X(B)$ that lies entirely over $F$, and which by virtue
of this fact also solves the symplectic parallel transport
equation. Thus the part of $\partial^h X(B)$ that lies over $F$ is
foliated by horizontal sections of the fibration, namely $(w/w_0)^\sigma =
(z/z_0)$ where $\pi(z_0,w_0) \in F$. Take note that $\sigma$ is merely
rational, so these horizontal sections may actually be multi-sections.
\end{proof}

Examples of the Hessian metrics such that $\partial^hB$ is
$g$-orthogonal to the fibers of $\eta$ may be constructed by starting
with the function
\begin{equation}
  F(x,y) = x^2 + \frac{y^2}{x} 
\end{equation}
\begin{equation}
  \Hess F = 
  \begin{pmatrix}
    2+ 2\frac{y^2}{x^3} & -2\frac{y}{x^2}\\
    -2 \frac{y}{x^2} & 2\frac{1}{x}
  \end{pmatrix}
\end{equation}
\begin{equation}
  \frac{-F_{xy}}{F_{yy}} = \frac{y}{x}
\end{equation}
Thus the families of lines $x = c$ and $y = \sigma x$ form an
orthogonal net for $\Hess F$. By taking $x$ and $y$ to be shifts of
the affine coordinates $\eta$ and $\xi$ on $B$, we can obtain a
Hessian metric on $B$ such that the vertical boundary consists of lines
of the form $x = c$, while the horizontal boundary consists of
lines of the form $y = \sigma x$.

\subsection{Focus-focus singularities and Lefschetz singularities}
\label{sec:focus-focus-lefschetz}

Now we consider the case where the affine structure on $B$ contains a
focus-focus singularity, that is, a singularity with monodromy conjugate to $\begin{pmatrix}1& 0\\ 1 &1\end{pmatrix}$, and the monodromy invariant direction of this
singularity is parallel to the fibers of the map $\eta: B \to I$. The
goal is to construct a symplectic manifold $X(B)$, along with a
Lefschetz fibration $w: X(B) \to X(I)$. Topologically $X(B)$ is a
torus fibration over $B$ with a single pinched torus fiber over the
singular point.

Let $B$ be an affine manifold with a single focus-focus
singularity, and $\eta: B \to I$ a globally defined affine
coordinate. Suppose $B$ has vertical boundary consisting of fibers of
$\eta$, as before, and suppose that the horizontal boundary consists
of line segments of rational slope. If we draw the singular affine
structure with a branch cut, one side will appear straight while the
other appears bent, though the bending is compensated by the monodromy
of the focus-focus singularity. Figure \ref{fig:tropical-fibration} shows the projection $B \to I$,
with the fibers drawn as vertical lines.
\begin{figure}
\includegraphics[width=3in]{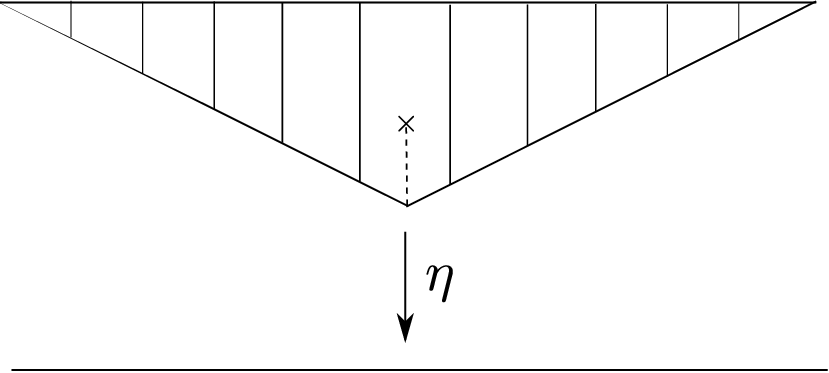}
\caption{The fibration $B \to I$.}
\label{fig:tropical-fibration}
\end{figure}

\begin{proposition}
  \label{prop:focus-focus-lefschetz}
  The manifold $X(B)$ admits an exact symplectic structure making $w : X(B)\to X(I)$ a symplectic Lefschetz fibration, with a single Lefschetz critical point. This critical point coincides with focus-focus singularity of the torus fibration $\pi : X(B)\to B$. There is a neighborhood $U \subset B$ of the fiber of $\eta$ through the singularity such that on the complement of $\pi^{-1}(U)$, the symplectic form is locally given by a Hessian metric of the form considered in Lemma \ref{lem:monodromy-hessian}.
\end{proposition}

\begin{proof}
%  We divide $B$ along
% the invariant line of the singularity, and use symplectic structures of
% the form considered in Lemma \ref{lem:monodromy-hessian} on either
% side of the invariant line, and we connect these pieces together in such a
% way that the critical point of the Lefschetz fibration occurs at the
% singular point of the focus-focus singularity. 

Suppose for convenience that the singularity occurs at $\eta =
0$. For the first step, divide the base $B$ into regions $B_{-\epsilon}=
\eta^{-1}(-\infty,-\epsilon]$ and $B_{+\epsilon} =
\eta^{-1}[\epsilon,\infty)$. On these affine manifolds we may take the
Hessian metrics and associated K\"{a}hler forms considered in Lemma
\ref{lem:monodromy-hessian}, using functions $K_-$ and $K_+$ derived from the example following that Lemma. Hence we get a fibration with symplectic
connection over the disjoint union of two annuli: $w: X(B_{-\epsilon}
\coprod B_{+\epsilon}) \to X(I_{-\epsilon} \coprod I_{+\epsilon})$

% The first step is to connect the two sides by going ``above'' and
% ``below'' the singularity. At the level of affine manifolds, we
% consider two bands connecting $B_{-\epsilon}$ to $B_{+\epsilon}$ near
% the two horizontal boundary faces. After adding appropriate affine linear functions to $K_-$ and $K_+$, we can glue $K_-$ and $K_+$ in each band to yield a convex fuction there. Since the boundary faces are
% straight in the affine structure, we can do this in
% such a way that the boundary faces are still orthogonal to the fibers
% of $\eta$, and so the portion of $X(B)$ lying over these faces is
% foliated by the symplectic connection. This extends the symplectic
% structure to the portion of $X(B)$ that lies over $B$ minus a disk
% around the singular point.

The second step is extend the Lefschetz fibration over a band connecting the two annuli. Consider the two annuli $X(I_{-\epsilon})\coprod
X(I_{+\epsilon}) \subset \C^*$. Choose a path connecting these two
annuli, along the positive real axis, say. By identifying the fibers
over the end points, the fibration extends over this path. By
thickening the path up to a band and filling in the fibers over the
band, we get a Lefschetz fibration over a surface which is
topologically a pair of pants. Since the fibration over the band is trivial, the gluing can be done in such a way that the horizontal boundary components of the total space are still foliated by multi-sections of the symplectic connection.

% If we also include the portions we
% filled in near the horizontal boundary, then we have a manifold with
% boundary, where one part of the boundary lies over the horizontal and
% vertical boundary of $B$, while the other is topologically an $S^3$. 

The third step is to fill in the rest of the region between the two annuli with the standard model of a
Lefschetz singularity. In order for this to make sense, we need the
symplectic monodromy around the loop in the base being filled in to be
a Dehn twist. This can be seen by comparing the symplectic monodromy
transformations around the loops in $X(I_{-\epsilon})$ and
$X(I_{+\epsilon})$. Let $z_{-}$ and $z_{+}$ denote complex coordinates
on $X(B_{-\epsilon})$ and $X(B_{+\epsilon})$ corresponding to a direction transverse to the fibration $\eta : B \to I$, so that $z_{-}$ and $z_{+}$ give coordinates on the fibers of $w: X(B)\to X(I)$. Assume these
coordinates match up in one of the bands connecting $B_{-\epsilon}$ to
$B_{+\epsilon}$. Due to the monodromy of the affine manifold they do
not coincide in the other band, where we put the branch cut. Let $F_1$
and $F_0$ denote the top and bottom faces of $\partial^h B$
respectively, and suppose that $F_0$ is split into two parts $F_{0+}$
and $F_{0-}$ by the branch cut. Associated to each of these we have a
slope $\sigma_F$.

Now we compute the monodromy obtained when we traverse a loop in $X(I_{-\epsilon})$ in the negative sense
followed by a loop in $X(I_{+\epsilon})$ in the positive sense,
connecting these paths through the band connecting the two annuli. We use Lemma \ref{lem:monodromy-hessian} to measure the difference between the amounts of phase
rotation in the $z_{-}$ and $z_{+}$ coordinates along the top and
bottom horizontal boundaries under symplectic parallel transport,
encoding this as an overall twisting. By Lemma \ref{lem:monodromy-hessian}, as we transport around the
negative loop in $X(I_{-\epsilon})$, the $z_{-}$ coordinate on
$\pi^{-1}(F_1)$ rotates through an angle $-2\pi\sigma_{F_1}$ and the $z_{-}$ coordinate on $\pi^{-1}(F_{0-})$ rotates through an angle $-2\pi\sigma_{F_{0-}}$. Thus the relative twist of the two ends of the fiber is $-(\sigma_{F_1} - \sigma_{F_{0-}})$ turns. By the same token, on the other
side the $z_+$ coordinate on $\pi^{-1}(F_1)$ rotates through an angle $2\pi\sigma_{F_1}$, while the $z_+$ coordinate on $\pi^{-1}(F_{0+})$ rotates through $2\pi\sigma_{F_{0+}}$, and the overall relative twisting is $(\sigma_{F_1} -
\sigma_{F_{0+}})$ turns. Composing these monodromy transformations, we find that the fiber undergoes a twisting of
$\sigma_{F_{0-}}-\sigma_{F_{0+}}$ turns. Because the monodromy of the affine
structure is a shear, this difference of slopes equals $-1$. Thus the fiber undergoes a right-handed Dehn twist (whose local model is obtained by taking the annulus $\{1\leq |z|\leq 2\}$ and rotating the outer boundary by one \emph{clockwise} turn).

This allows us to fill in the fibration with a standard fibration with
a single Lefschetz singularity whose vanishing cycle is the equatorial
circle on the cylinder fiber. Because the top and bottom boundaries are fixed under the monodromy transformation, the foliation of the horizontal boundary extends over the gluing. Since this local model is symmetric
under the $S^1$--action which rotates the fibers, choosing an
$S^1$--invariant gluing allows us to define a symplectic $S^1$--action
on $X(B)$ which rotates the fibers of $w: X(B)\to X(I)$.

Since the total space is $S^1$--symmetric, we can construct the Lagrangian
tori as in section \ref{sec:torus-fibrations}, by taking circles of
constant $|w|$ in the base and $S^1$--orbits in the fiber. These
actually coincide with the tori found in $X(B_{-\epsilon} \coprod
B_{+\epsilon})$ as fibers of the projection to $B$, so this
construction extends the torus fibrations on $X(B_{-\epsilon} \coprod
B_{+\epsilon})$ to all of $X(B)$.
\end{proof}

\begin{remark}
Since this construction is local on the base $X(I)$, the construction
extends in an obvious way to the situation where several focus-focus
singularities with parallel monodromy-invariant directions are
present. See Section \ref{sec:other-manifolds} for further discussion.
\end{remark}

\begin{remark}
  There is a more direct way to obtain the sort of symplectic structure we desire, based on a version of the mapping-torus construction. Start with a Lefschetz fibration over an rectangle, with a critical value corresponding to the singular point of $B$. Trivialize the fibration over the top and bottom sides of the base rectangle, and glue the top and bottom together, identifying the fibers using a symplectomorphism, yielding a fibration over an annulus that we identify with $X(I)$. Choosing the gluing symplectomorphism appropriately, we can ensure that the monodromies around the sub-annuli $X(I_{\pm \epsilon}) \subset X(I)$ agree with what is obtained from the Hessian metric construction. Our preference for the construction described above is that it proceeds more naturally from the affine geometry of $B$.
\end{remark}

\begin{remark}
\label{rem:cut-corner}
In the main example of this paper, the affine manifold does not have
vertical boundary, but rather corners at the extreme values of
$\eta$. To define the symplectic structure, we deal with this by
cutting off arbitrarily small pieces of $B$ at the corners so that the
result has vertical boundaries. However, in terms of the
correspondence to tropical geometry, we still consider the unmodified
manifold $B$.
  \end{remark}

  \begin{definition}
    Let $B$ be the affine manifold appearing in the mirror of $(\CP^2,D)$. We denote by $X(B)$ be the symplectic manifold obtained from Proposition \ref{prop:focus-focus-lefschetz}, after cutting off the corners of $B$ as in Remark \ref{rem:cut-corner}.
  \end{definition}
  
Recall that the horizontal boundary $\partial^h X(B)$ is the union
of two faces $(\partial^h X(B))_1$ and $(\partial^h X(B))_0$
corresponding to $F_1$ and $F_0$, the top and bottom faces of
$B$. Likewise we speak of the top and bottom boundary circles of the fibers of $w$. The following proposition summarizes key properties of $X(B)$. It follows from the analysis of the monodromies in the proof of Proposition \ref{prop:focus-focus-lefschetz} and Lemma \ref{lem:monodromy-hessian}.
\begin{proposition}
  \label{prop:key-properties}
  The Lefschetz fibration $w : X(B) \to X(I)$ has the following properties:
  \begin{enumerate}
  \item  The monodromy given by traversing the inner boundary loop \emph{clockwise} fixes the top boundary circle of the fiber, and rotates the bottom boundary circle through a half turn, in such a way that the square of this operation is isotopic to a \emph{right-handed} Dehn twist.
  \item The monodromy given by traversing the outer boundary loop \emph{clockwise} fixes the top boundary circle and acts on the bottom boundary circle by a half-turn, except that the square of this operation is isotopic to a \emph{left-handed} Dehn twist.
\item As the top face $F_1$ has integral slope (zero in our diagrams), the leaves of the foliation on $(\partial^h
X(B))_1$ are single-valued sections of the $w$-fibration. 
\item As the bottom face $F_2$ has half-integral slope ($\pm 1/2$ on either side of the cut in our diagrams), the leaves of
the foliation on $(\partial^h X(B))_0$ are two-valued sections of the
$w$-fibration.
  \end{enumerate}
\end{proposition}

\begin{remark}
  \label{rem:leaves-W}
  The foliation of the boundary faces by horizontal sections is related to the superpotential $W= u +e^{-\Lambda}v^2/(uv-1)$ studied in the previous Section \ref{sec:fiber-superpotential}. The general fiber $W = \text{constant}$ is a twice-punctured torus. On the other hand, we can consider the curves defined by the individual terms of $W$. The equation $u = \text{constant}$ defines a curve such that the projection $w = uv-1: X^\vee \to \C$ is one-to-one, so this curve is a section of $w$. The curves defined by the second term $v^2/(uv-1) = \text{constant}$ are likewise two-valued sections of $w: X^\vee \to \C$. Our perspective is that the leaves of the foliations on the boundary components of $X(B)$ correspond roughly to these curves defined by the individual terms of the superpotential. 
\end{remark}

\subsection{Lagrangians fibered over paths}
\label{sec:lags-construction}

The base of the Lefschetz fibration is the annulus $X(I) = \{R^{-1}
\leq |w| \leq R\}$ with a critical value at $w = -1$. The symplectic
structures constructed in \ref{sec:focus-focus-lefschetz} have the
property that the symplectic connection is flat throughout the annuli
$X(I_{-\epsilon}) = \{R^{-1} \leq |w| \leq e^{-\epsilon}\}$,
$X(I_{+\epsilon}) = \{e^\epsilon \leq |w| \leq R\}$, as well as
through a band along the positive real axis joining these annuli.

We call the Lagrangians $L(d)$ constructed here \emph{sections}, because they will turn out to be sections of the torus fibration on $X(B)$. At the same time, they are \emph{not} sections of the Lefschetz fibration, but rather fiber over paths in $X(I)$.

\subsubsection{The zero-section}
\label{sec:zero-section}

The first step is to construct the Lagrangian submanifold $L(0)
\subset X(B)$, which we will use as a zero-section of the torus fibration and reference point
through out the paper.

\begin{definition}
  Let  $\ell(0) \subset X(I)$ be the path that runs along
the positive real axis. Trivializing the fibration along that path, take a path in the fiber cylinder connecting the top and bottom boundary circles. The Lagrangian tori intersect the fibers of the Lefschetz fibration in circles, and we choose our fiber path so that it intersects each such circle once. We denote by $L(0) \subset X(B)$ the product Lagrangian submanifold. 
\end{definition}
If we want to be specific, we could take the
factor in the fiber to be the positive real locus of the coordinates
$z_-$ or $z_+$. By construction, $L(0)$ intersects each Lagrangian torus once.

Once $L(0)$ is chosen, it selects leaves of the foliations by horizontal sections on each
boundary face, namely those leaves containing the fiberwise boundary of $L(0)$. Call these
leaves $\Sigma_0$ and $\Sigma_1$ (bottom and top
respectively). Clearly we could have chosen these leaves first and
then constructed $L(0)$ accordingly.

\subsubsection{The degree $d$ section}
\label{sec:other-sections}

We can now use $L(0)$ as a reference to construct the other
Lagrangians $L(d)$.

\begin{definition}
 Let $\ell(d)$ be a base path, with the same end
points and midpoint as $\ell(0)$, and which winds $d$ times (relative
to $\ell(0)$) in $X(I_{-\epsilon})$ and also $d$ times in
$X(I_{+\epsilon})$. The winding of $\ell(d)$ is clockwise as we go
from smaller to larger radius. The Lagrangian $L(d)$ is defined to coincide with $L(0)$ in the fiber over the common endpoint of $\ell(0)$ and $\ell(d)$ on the inner boundary circle of $X(I)$, and to be defined elsewhere as the submanifold swept out by symplectic parallel transport along the path $\ell(d)$.
\end{definition}
Recall that we cut off the
corners of $B$ so that we actually have fibers at the endpoints of $\ell(0)$ and $\ell(d)$. Because the boundary leaves $\Sigma_0$ and
$\Sigma_1$ are preserved by symplectic parallel transport, the boundary of $L(d)$ is contained in these same leaves, which as in Remark \ref{rem:leaves-W} correspond to hypersurfaces defined by the individual terms of the superpotential.

The behavior of $L(d)$ in the fiber direction is determined with the aid of Proposition \ref{prop:key-properties}. 

\begin{proposition}
  \label{prop:fiber-behavior}
As we parallel transport clockwise along $\ell(d)$ in inner annulus $X(I_{-\epsilon})$, the fiber component of $L(d)$ is transformed by half-right-handed Dehn twists, while as we parallel transport clockwise along $\ell(d)$ in the outer annulus $X(I_{+\epsilon})$, we pick up half-left-handed Dehn twists. In the end, all of the twisting cancels out, so that $L(d)$ and $L(0)$ again coincide over the common endpoint of $\ell(0)$ and $\ell(d)$ on the outer boundary circle of $X(I)$.
\end{proposition}

The proposition is illustrated in Figure \ref{fig:lags-over-paths}, which  depicts $L(0)$, $L(1)$, and $L(2)$. The lower portion of the figure shows the
base: the straight line is $\ell(0)$, while the spirals are $\ell(1)$
and $\ell(2)$. The marked point is the Lefschetz critical value. The
upper portion of the figure shows the five fibers where $\ell(0)$ and
$\ell(2)$ intersect. The behavior of $L(i)$ for $i = 0,1,2$ in these fibers is depicted.
\begin{figure}
\includegraphics[width=3in]{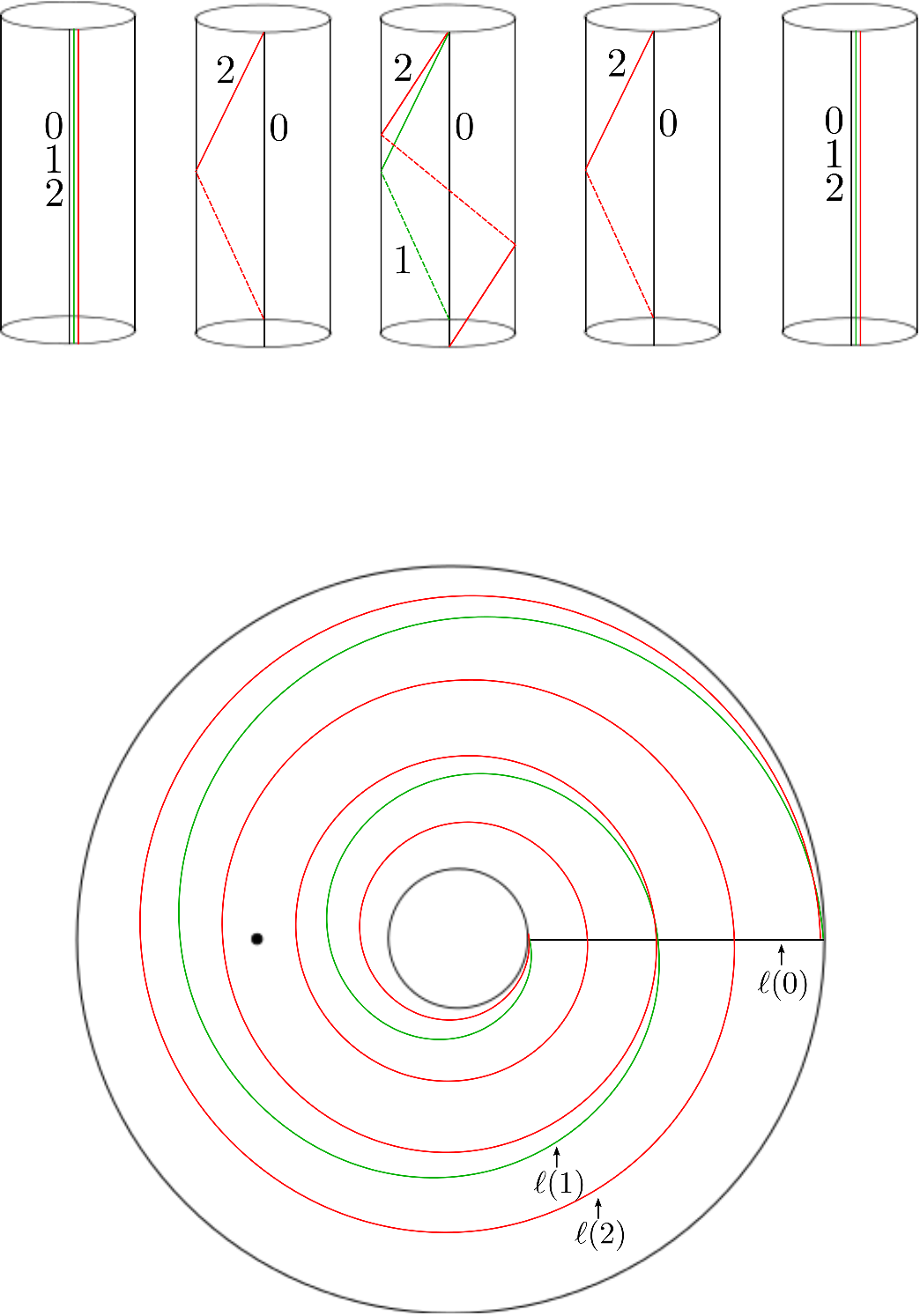}
\caption{The Lagrangians $L(0)$, $L(1)$, and $L(2)$. Parts of the figure corresponding to $L(i)$ are labeled by $i$.}
\label{fig:lags-over-paths}
\end{figure}

The Lagrangian submanifold $L(d)$ is indeed a section of the torus
fibration. If $T_{R,\lambda}$ is the torus over the circle $\{|w| = R\}$
at height $\lambda$, then since $\ell(d)$ intersects $\{|w| = R\}$ at
one point, there is exactly one fiber of $w:X(B)\to X(I)$, where
$L(d)$ and $T_{R,\lambda}$ intersect. Since $L(d)$ intersects each
$S^1$--orbit in that fiber once, we find that $L(d)$ and
$T_{R,\lambda}$ indeed intersect once.

\subsubsection{Admissibility}
\label{sec:admissibility}

We now explain in what sense these Lagrangians are admissible. The
relevant notion of admissibility is the one found in \cite[\S
7.2]{auroux07}, where admissibility with respect to a reducible
hypersurface whose components correspond to the terms of the
superpotential is discussed. In our case, we have two components
$\Sigma_0$ and $\Sigma_1$, and the admissibility condition is that,
near $\Sigma_i$, the holomorphic function $z_i$ such that $\Sigma_i =
\{z_i = 1\}$ satisfies $z_i|L \in \R$. The Lagrangian $L(d)$ will have
this property if the symplectic structure is chosen so that the
symplectic monodromy near $\Sigma_1$ is actually trivial, while the
symplectic monodromy near $\Sigma_0$ is a rigid rotation by
$\pi$. Otherwise, we can only say that the phase of $z_i$ varies
within a small range near $\Sigma_i$. In the computations in section
\ref{sec:degeneration} we actually perturb these Lagrangians in a
prescribed manner near the boundary.

The notion of admissibility also prescribes the behavior over the
endpoints of the base path $\ell(d)$. Recall that we cut off the
corners of $B$ in order to define the symplectic structure. Our notion
of admissibility remembers that there was a corner at this point, by
requiring the $L(d)$ to coincide with $L(0)$ over the endpoints of
$\ell(d)$. The justification for this is that, if we try to extend the
geometry all the way to the corner of $B$, we see that the two
components $\Sigma_0$ and $\Sigma_1$ would intersect, and we need the
Lagrangian to behave tamely at this intersection. A local model for
this sort of situation is the two complex curves $\Sigma_0 = \{x=0\}$,
$\Sigma_1 = \{y= 0\}$ intersecting at $(0,0)$ in $\C^2$, with the
Lagrangian being $L = (\R_{\geq 0})^2$, and with the fibration being
$w = x+y : \C^2 \to \C$.

% In the T-duality picture, one constructs a Lagrangian $L \subset
% X^\vee$ from a line bundle $\mathcal{E} \to X$ by taking a
% $U(1)$-connection on $\mathcal{E}$ which is flat on the SYZ fibers of
% $X$, plotting this as a submanifold of $X^\vee$, which is naively the
% moduli space of SYZ fibers with flat $U(1)$-connection. As the SYZ
% torus fibers degenerate along the anticanonical boundary divisor, one
% of the monodromies necessarily becomes unity, and as the torus fiber
% degenerates in the corner to a point, all the monodromies become
% unity. Hence, for each corner of $X$ there is a distinguished point in
% $X^\vee$ through which all the Lagrangians must pass. This corresponds
% to the requirement that all the Lagrangians coincide (or at least
% become close) over the endpoints of the base paths.

\subsubsection{Positive perturbations}
\label{sec:pert}
The Lagrangians $L(d)$ constructed above intersect each other on the
boundary of $X(B)$, and in particular it is not clear whether such
intersection points are supposed count toward the Floer
cohomology. There is a canonical way to perturb the Lagrangians so as
to push all intersection points which should count toward Floer
cohomology into the interior of $X(B)$, also used by Abouzaid
\cite{abouzaid06}. This construction is not
symmetric with respect to switching the Lagrangians. 

Let $K$ and $L$ be two of the Lagrangians $L(d)$. The perturbation
process we use in computing morphisms $CF^*(K,L)$ from $K$ to $L$ has two
steps. First we perturb the Lagrangians near the boundary, staying
within the class of Lagrangians fibering over paths, so that they
intersect at the boundary in such way that the tangent space to the base path of $L$ is a small \emph{counterclockwise} rotation of the tangent space to base path of $K$, and similarly in the fiber. If $K = L(d_1)$ and $L = L(d_2)$ with $d_1 < d_2$, this will create new intersection points in the interior, while if $d_1 > d_2$ the pair $K,L$ is already in the desired position.

The second step is that we perturb $K$ and $L$ at the boundary intersection so as to destroy the intersection there. Alternatively, we could just forget about the boundary intersection points in our computations, but actually removing them is more convenient in the technical part of the paper.
%  Since we do not care about the intersection points on the boundary, it
% does not affect our Floer cohomology computations if we further make a
% small perturbation near the boundary that destroys the boundary
% intersection points without creating new intersections.
 The pair of
Lagrangians obtained from this process is called \emph{positively
  perturbed}. This has the effect that one of Lagrangians in the pair
may not admissible in the sense above, since it does not have boundary
on $\Sigma_0$ and $\Sigma_1$. In section \ref{sec:degeneration}, the
Lagrangians we work with are positive perturbations of admissible
Lagrangians.

Please see figure \ref{fig:perturbation} for a local picture of the positive perturbation, and the end result in the case of $K = L(0)$ and $L = L(2)$.
\begin{figure}
\includegraphics[width=3in]{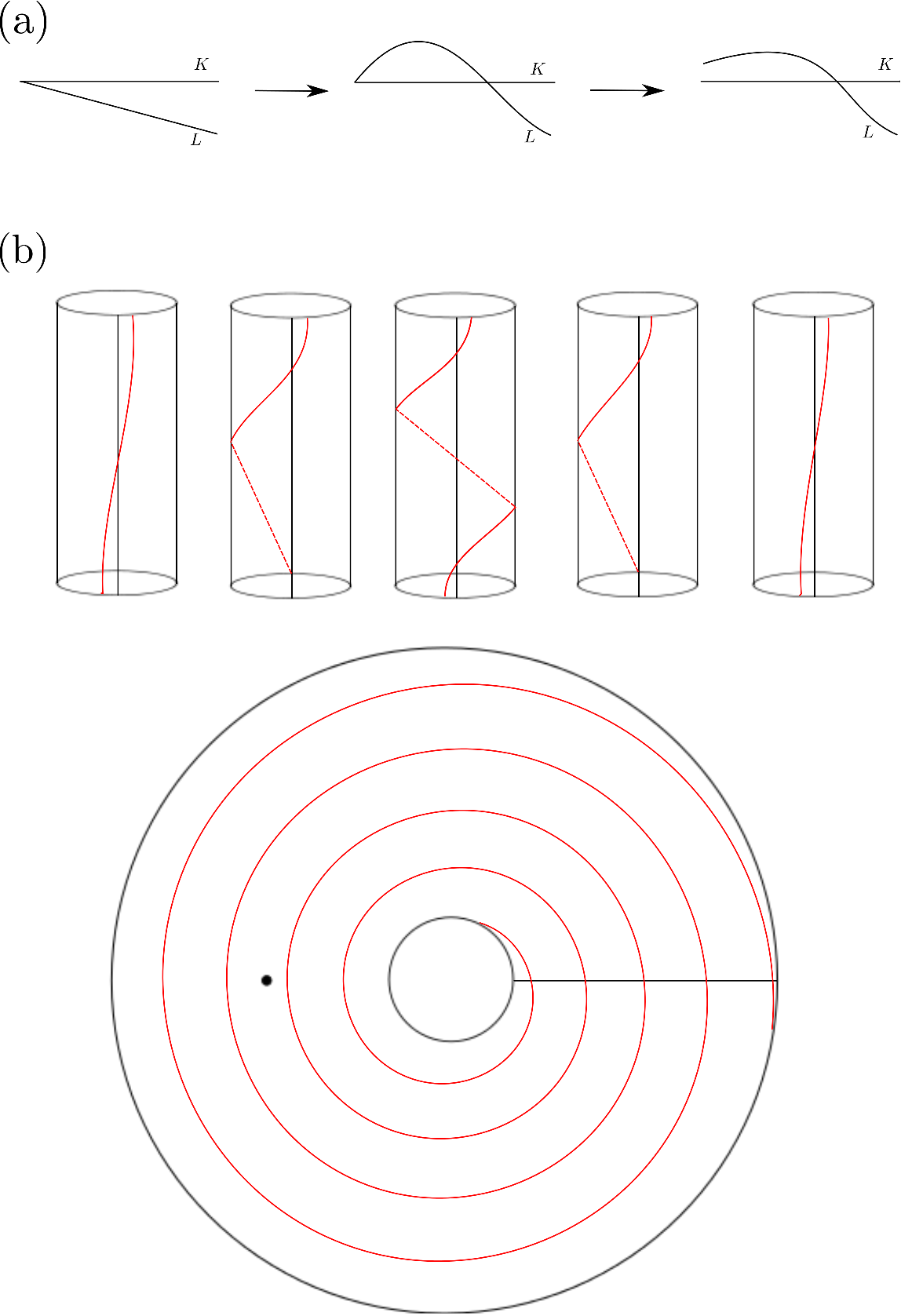}
\caption{(a) The local picture of positive perturbation. (b) The end result when $K = L(0)$ and $L= L(2)$.}
\label{fig:perturbation}
\end{figure}

\begin{remark}
  The asymmetry in the perturbations is essential for the mirror symmetry statement that we wish to prove. Mirror symmetry predicts that $HF^*(L(d_1),L(d_2))$ is isomorphic to $\Ext^*(\cO_{\CP^2}(d_1),\cO_{\CP^2}(d_2)) \cong H^*(\CP^2,\cO_{\CP^2}(d_2-d_1))$. (The main purpose of the paper is to verify that this is true for $d_1 < d_2$ in a way compatible with the product structures.) Even at the level of total dimension, this space is not symmetric with respect to swapping $d_1$ and $d_2$, so it makes sense that we must treat intersection points differently depending on whether we consider them to be morphisms from $L(d_1)$ to $L(d_2)$ or vice versa.

While we only discuss the product $\mu^2$ on Floer cohomology in this paper, it would be possible to use this sort of perturbation in the definition of higher $A_\infty$ operations, and to define a Fukaya category, by adapting the analogous constructions in Abouzaid's work \cite{abouzaid06,abouzaid09}.
\end{remark}

% The perturbation appropriate for computing Floer cohomology
% $HF^*(L(0),L(d))$ with $d> 0$ is the following: we perturb the base
% path $\ell(d)$ near the end points by creating a new intersection
% point in the interior, in addition to the one on the boundary. We also
% perturb the part of $L(d)$ over the fiber at $w = R^{-1}$, which was
% the initial condition for the parallel transport construction of
% $L(d)$, so that rather than coinciding $L(d)$ intersects $L(0)$ once
% in the interior as well as on the boundary, and at this intersection
% point, the tangent space of $L(d)$ is a small clockwise rotation of
% the tangent space of $L(0)$. After parallel transport this will ensure
% the intersections of $L(0)$ and $L(d)$ over other points of $\ell(0)
% \cap \ell(d)$ are transverse as well. With an appropriate choice of
% complex volume form $\Omega$ for the purpose of defining gradings on
% Floer complexes, all of the interior intersection points will have
% degree 0 when regarded as morphisms going from $L(0)$ to $L(d)$, while
% the intersection points on the boundary have degree 2. We then make a
% further perturbation at the boundary that destroys all of the boundary
% intersection points.

% On the other hand, when computing Floer cohomology $HF^*(L(d),L(0))$
% with $d > 0$, we perform the perturbation in the opposite
% direction. This does not create new intersections in the interior, and
% the boundary intersection points are forgotten, so there are actually
% fewer generators of $CF^*(L(d),L(0))$ than there are for
% $CF^*(L(0),L(d))$ when $d > 0$.

\subsubsection{Gradings and degrees}
\label{sec:gradings}

We now explain how the degrees of the intersection points between the positively perturbed Lagrangians $L(d)$ is determined. Our conventions follow Seidel \cite[\S 11]{seidel-book}. We choose a smooth trivialization of the canonical bundle of $X(B)$ given by a complex volume form $\Omega$. Our $X(B)$ is diffeomorphic to the manifold $X^\vee$ considered in section \ref{sec:fiber-superpotential}, which carries the volume form
\begin{equation}
  \Omega_{X^\vee} = \frac{du \wedge dv}{w} = \frac{du}{u} \wedge \frac{dw}{w}
\end{equation}
having the convenient property that, away from the singular fiber of the Lefschetz fibration, $\Omega_{X^\vee}$ decomposes as a wedge product of one-forms with respect to the local product structure determined by the coordinates $(u,w)$. We can transfer this form over to $X(B)$, and deform it to a complex volume form $\Omega$ that is compatible with the fibration structure, in the sense that, when $\Omega$ is restricted to the tangent space of a point on a smooth fiber, it decomposes into a wedge product of a one-form on the vertical tangent space and a one-form on the horizontal space.

 We can also describe the trivialization determined by $\Omega$ in terms of the foliation by the tori $T_{R,\lambda}$. The tangent space to $X(B)$ at any point is the complexification of the tangent space to the torus $T_{R,\lambda}$ passing through that point, so an orientation on $T_{R,\lambda}$ trivializes the canonical bundle of $X(B)$ along that torus. Since all of these tori can be oriented in a consistent manner, we obtain a trivialization of the canonical bundle over the complement of the critical point in $X(B)$. This trivialization is homotopic to the one determined by $\Omega$ for the following reason. The holomorphic volume form $\Omega_{X^\vee}$ on $X^\vee$ admits a family of special Lagrangian tori, which means precisely that $\Omega_{X^\vee}$ restricted to each torus is a real volume form times a constant phase factor. When we transfer $\Omega_{X^\vee}$ over to $X(B)$, this foliation by special Lagrangian tori goes over to a foliation by tori that is isotopic to the foliation by the tori $T_{R,\lambda}$. (The fact that the tori are Lagrangian is not important here, only that they are totally real.) The foliation by the tori $T_{R,\lambda}$ is also compatible with the fibration structure, as each torus fibers over a circle in the base $X(I)$ and intersects each fiber in a circle. Thus we obtain foliations by circles on the base and the fiber. 

The point of these choices is that we will be able to compute the degrees of our intersection points in terms of amounts of rotation with respect to the circle foliations on the base and the fiber. Because the Lagrangian $L(d)$ satisfies $H^1(L(d)) = 0$, it admits gradings. Moreover, because it fibers over a path in the base, and consists of a path in each fiber, it may be graded separately in the base and fiber directions. To say what we mean by this, recall the notion of a \emph{squared phase map} from \cite{seidel-book}. This is the map
\begin{equation}
  \alpha_{L(d)} : L(d) \to S^1, \quad \alpha_{L(d)}(p) = \frac{\Omega(v_1\wedge v_2)^2}{|\Omega(v_1\wedge v_2)|^2}
\end{equation}
where $\{v_1, v_2\}$ is a basis of $T_pL(d)$. A grading is then a real-valued lift $\tilde{\alpha}_{L(d)}: L(d) \to \R$ such that $\exp(2\pi i \tilde{\alpha}_{L(d)}) = \alpha_{L(d)}$. In our situation, we can decompose this squared phase map into base and fiber components
\begin{equation}
  \alpha_{L(d)}^b,\  \alpha_{L(d)}^f : L(d) \to S^1
\end{equation}
such that $\alpha_{L(d)} = \alpha_{L(d)}^b\cdot \alpha_{L(d)}^f$, such that $\alpha_{L(d)}^b$ is the squared phase with respect to the foliation by circles on the base, and $\alpha_{L(d)}^f$ is the squared phase with respect to the foliation by circles on the fiber. Normalizing and homotoping the volume form $\Omega$ if necessary, we may assume that the squared phase of the torus $T_{R,\lambda}$ is $1 \in S^1$ for all three versions, and so that the squared phase of $L(0)$ is $-1 \in S^1$ for both $\alpha^b$ and $\alpha^f$ (and hence its total phase is $\alpha = 1$). 

Now, if we set $d > 0$, and we consider $L(d)$, we find that the tangent space of $L(d)$, in either the base or fiber, is obtained from that of $T_{R,\lambda}$ by a small counterclockwise rotation. This means that $\alpha_{L(d)}^b$ and $\alpha_{L(d)}^f$ lie on the upper half circle between $1$ and $-1$ in $S^1$. 

Next, choose lifts $\tilde{\alpha}^b$ and $\tilde{\alpha}^f$ for the base and fiber squared phase maps of the manifolds $L(d)$ for $d \geq 0$, preferring values the interval $[0,1]$. Note that $L(0)$ has 
\begin{equation}
\tilde{\alpha}_{L(0)}^b = \tilde{\alpha}_{L(0)}^f = \frac{1}{2}  
\end{equation}
For the other $L(d)$ with $d > 0$ the value of $\tilde{\alpha}_{L(d)}^b$ will always lie in the interval $(0,\frac{1}{2})$, while the value of $\tilde{\alpha}^f$ lies in the interval $(0,1)$, since the intersection of $L(d)$ with any fiber is always a path that is transverse to the foliation by circles. 

Finally, with the gradings in place, we can compute the degrees of the intersection points between $L(0)$ and $L(d)$ for $d > 0$. 

\begin{proposition}
 Suppose $d_1 < d_2$. The the Floer complex $CF^*(L(d_1),L(d_2))$ of the positively perturbed pair $L(d_1), L(d_2)$ is concentrated in degree zero. The Floer complex $CF^*(L(d_2)), L(d_1))$ for the positively perturbed pair $L(d_2),L(d_1)$ is concentrated in degree two. 
\end{proposition}

\begin{proof}
Let us first consider the case $d_1 = 0$ and $d_2 = d > 0$.
Using the local splitting given by the fibration, this degree is a sum of contributions from the base and fiber. Since the base and fiber are complex one-dimensional, the base and fiber contributions are for $p \in L(0) \cap L(d)$ using \cite[Eq. 11.35]{seidel-book},
\begin{align}
  i^b(L(0),L(d),p) & = [\tilde{\alpha}^b_{L(d)}-\tilde{\alpha}^b_{L(0)}]+1\\
 i^f(L(0),L(d),p) & = [\tilde{\alpha}^f_{L(d)} - \tilde{\alpha}^f_{L(0)}]+1
\end{align}d
where $[\cdot]$ denotes the greatest integer function. Look at the diagram shows that, at any intersection point $p$, the quantity inside the $[\cdot]$ is negative, as the tangent space to $L(d)$ is a small \emph{clockwise} rotation of the tangent space to $L(0)$. Thus both $i^b(L(0),L(d),p)$ and $i^f(L(0),L(d),p)$ are zero, and the total degree is zero.

If we keep $d > 0$, but swap the roles of $L(0)$ and $L(d)$, we find that
\begin{align}
  i^b(L(d),L(0),p) & = [\tilde{\alpha}^b_{L(0)}-\tilde{\alpha}^b_{L(d)}]+1 = 1\\
  i^f(L(d),L(0),p) & = [\tilde{\alpha}^f_{L(0)} - \tilde{\alpha}^f_{L(d)}]+1 = 1
\end{align}
so that the intersection point $p$, regarded as a morphism from $L(d)$ to $L(0)$, has degree 2.

A similar analysis shows that the degrees of intersections $L(d_1)\cap L(d_2)$, regarded as a morphism from $L(d_1)$ to $L(d_2)$, depends on $d_2 - d_1$. If $d_2 - d_1$ is positive, they have degree $0$, while if $d_2 - d_1$ is negative, they have degree $2$.
\end{proof}

Since, in each case, the complex is concentrated in a single degree, the differential vanishes, and we may identify the complex $CF^*(L(d_1),L(d_2))$ with its cohomology $HF^*(L(d_1),L(d_2))$.

\subsubsection{Intersection points and integral points}
\label{sec:intersections}

With the above prescription for perturbing the Lagrangians $L(d)$, we
now work out the bijection between the intersection points of $L(0)$
and $L(d)$ for $d>0$, regarded as morphisms from $L(0)$ to $L(d)$, and
the $(1/d)$-integral points of $B$.

\begin{proposition}
  Let $B$ be scaled so that the top face has length $2$. Take $d > 0$. Then the intersection points $L(0) \cap L(d)$ giving a basis of $CF^0(L(0),L(d))$ correspond to $(1/d)$-integral points of $B$, including those on the boundary of $B$. For $d < 0$, the intersection points giving a basis of $CF^2(L(0),L(d))$ correspond the same set but with points on the boundary of $B$ excluded. For the pair $L(d_1),L(d_2)$, the same conclusion holds with $d = d_2 - d_1$. 
\end{proposition}

This proposition is proved by explicitly indexing all of the various points. We start at the intersection point of $\ell(0)$ and $\ell(d)$ near the
inner radius of the annulus $X(I)$. This intersection point survives
the perturbation. In the fiber over this point there is one
intersection point that survives after perturbation. As we transport
around the inner part of the annulus, we pick up half-twists in the
fiber (Proposition \ref{prop:fiber-behavior}), which increases the number of intersection points by one after
every two turns in the base. This pattern continues until $\ell(d)$
reaches the middle radius and starts winding around the other side of
the Lefschetz singularity, where the pattern reverses. (See Figure
\ref{fig:lags-over-paths}.)

Assign the rational numbers \[[-1,1] \cap (1/d)\Z = \{-1,
-(d-1)/d,\dots, -1/d, 0, 1/d, \dots, 1\}\] to the intersection points
of $\ell(0)$ and $\ell(d)$, then over the point indexed by $a/d$ the
number of intersection points in the fiber is $1+\left\lfloor\frac{d-|a|}{2}\right\rfloor$. We index the points in a given fiber using the index $i \in \{0,1,\dots, \left\lfloor\frac{d-|a|}{2}\right\rfloor\}$, starting at the top of the fiber.

If we scale $B$ so that the top face has affine length $2$, the $1/d$
integral points of $B$ are also organized by the projection $\eta: B \to I$
into columns indexed by $[-1,1] \cap (1/d)\Z$. By inspection, the column over
$a/d$ has $1+ \left\lfloor\frac{d-|a|}{2}\right\rfloor$ points. These are also indexed by $i \in \{0,1,\dots, \left\lfloor\frac{d-|a|}{2}\right\rfloor\}$, starting at the top of the fiber.

Under this bijection the $(1/d)$-integral points on the boundary of $B$ correspond to intersections of the unperturbed $L(0)$ and $L(d)$ that lie on the boundary of $X(B)$. According to our definition of positive perturbation, these points do (do not) contribute when $d > 0$ ($d < 0$).

% A convenient way to index the intersection points in each column is by
% their distance from the top of the fiber. So in the column over $a/d$,
% we have intersections indexed by $i \in \{0,1,\dots,
% \left\lfloor\frac{d-|a|}{2}\right\rfloor\}$ which lie at distances
% $i/\left(\frac{d-|a|}{2}\right)$ from the top of the fiber.

\begin{definition}
\label{def:indexing} Take $d > 0$.
For $a \in \{-d,\dots,d\}$, and $i \in
\{0,1,\dots,\left\lfloor\frac{d-|a|}{2}\right\rfloor\}$, let $q_{a,i}(n)
\in L(n)\cap L(n+d)$ be the intersection which lies in the column indexed by $a/d$, and
which is the $i$th from the top of the fiber.
\end{definition}

We can already verify mirror symmetry at the level of the Hilbert polynomial:
\begin{equation}
  \label{eq:hilbert-poly}
  |L(0)\cap L(d)| = \left|B\left(\frac{1}{d}\Z\right)\right| =
  \frac{(d+2)(d+1)}{2} = \dim H^0(\CP^2, \cO_{\CP^2}(d))
\end{equation}

Figure \ref{fig:intersections} shows the points of $B(\frac{1}{4}\Z)$
representing the basis of morphisms $L(d) \to L(d+4)$.
\begin{figure}
\includegraphics[width=3in]{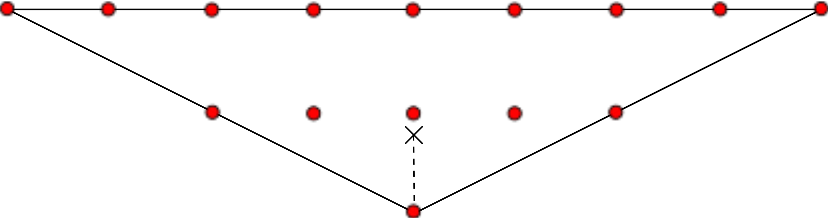} 
\caption{The $1/4$--integral points of $B$.}
\label{fig:intersections}
\end{figure}

\subsubsection{Hamiltonian isotopies}
\label{sec:ham-iso}

There is an alternative way to express the relationship between $L(d)$
and $L(0)$, which is by a Hamiltonian isotopy. There is a Hamiltonian
function $H$ on $X(B)$ such that the time-$d$ flow of $H$ takes $L(0)$
to $L(d)$.  During the intermediate times of this isotopy the
Lagrangian will not be admissible (or even close to it), but at
the end of the isotopy admissibility is restored. This observation is
used when we consider wrapped Floer cohomology in section
\ref{sec:complements}.

This isotopy also allows us to identify the intersection points $L(n)\cap L(n+d)$ for different values of $n$. For this reason, we will write $q_{a,i}(n)$ as simply $q_{a,i}$.

\section{A degeneration of holomorphic triangles}
\label{sec:degeneration}

Since we have set up our Lagrangians as fibered over paths, a
holomorphic triangle with boundary on the Lagrangians composed with
the projection is a holomorphic triangle in the base, which is an
annulus, with boundary along the corresponding paths. The triangles
that are most interesting are those that pass over the critical value
$w = -1$ (possibly several times). In general, such triangles are
immersed in the annulus, and, after passing the the universal cover of
the annulus, are embedded. Hence, we can regard such triangles as
sections over a triangle in the base of a Lefschetz fibration having
as base a strip with a $\Z$--family of critical values. Once this is
done, we can apply a TQFT for counting sections of Lefschetz
fibrations developed by Seidel \cite{seidel-book,seidel-les}.

We consider the deformation of the Lefschetz fibration over the
triangle where the critical values bubble out along one of the
sides. At the end of this degeneration, we count sections of a trivial
fibration over a $(k+3)$--gon, along with sections of $k$ identical
fibrations, each having a disk with one critical value and one
boundary marked point. Each of these fibrations is equipped with a
Lagrangian boundary condition given by following the degeneration of
the original Lagrangian submanifolds. The sections of the trivial
fibration over the $(k+3)$--gon can be reduced to counts in the
fiber, while the counts of the $k$ other parts are identical, and can
be computed directly using the techniques of \cite{seidel-les}.

\subsection{Triangles as sections}
\label{sec:triangles-sections}

Let $q_1 \in HF^0(L(0),L(n))$ and $q_2 \in HF^0(L(n),L(n+m))$ be two
degree zero morphisms whose Floer product $\mu^2(q_2,q_1)$ we wish to
compute. Suppose that $p \in HF^0(L(0),L(n+m))$ contributes to this
product. Let $S$ denote a disk with three boundary punctures, and a
complex structure that is allowed to vary. To find the coefficient of
$p$ in $\mu^2(q_2,q_1)$, we count pseudo-holomorphic triangles, that
is, pseudo-holomorphic maps $u: S \to X(B)$ that send the punctures to
$q_1,q_2,p$ and the boundary components to $L(0),L(n),L(n+m)$. Various authors have studied the construction of such invariants; we follow the theory as developed by Seidel in \cite{seidel-book,seidel-les}.

Because we want to make contact with the theory of pseudo-holomorphic sections of Lefschetz fibrations, we need to import the setup from \cite[\S 2.1]{seidel-les}. Let $\pi : E \to B$ be a exact symplectic Lefschetz fibration. We choose primitive $\theta$ for the symplectic form $\omega$ such that the Liouville vector field $Z$ defined on the fiber $M$ by $\iota_Z(\omega|_M) = \theta|_M$ points outward along the boundary of $M$. The flow of $Z$ defines a collar $[-\epsilon,0]\times M \to M$, and we let $\sigma$ be the function on a neighborhood of $\partial M$ given by projection to the $[-\epsilon,0]$ factor. We pick an almost complex structure $j$ on the base of the fibration, we always consider an almost complex structure $J$ on the total space which is \emph{compatible relative to $j$}, meaning
\begin{enumerate}
\item $J$ is integrable in a neighborhood of each critical point,
\item $\pi$ is a $(J,j)$-holomorphic map,
\item For each fiber $E_z$, the form $\omega(\cdot, J\cdot)|TE_z$ is symmetric and positive definite
\item on a neighborhood of the horizontal boundary of $E$, $J$ satisfies $\theta \circ J = d(e^\sigma)$.
\end{enumerate}
Later on, in Section \ref{sec:horiz-sections}, we will also need to consider $J$ which are \emph{horizontal}, meaning that $J$ preserves the horizontal subspaces of the symplectic connection. Once we extend the Lagrangian boundary condition Section \ref{sec:extend-fiber} so that our Lagrangians are compact in each fiber, these conditions imply that the curves we wish to count lie in a compact subset disjoint from the horizontal boundary \cite[Lemma 2.2]{seidel-les}. The second condition is particularly useful because it allows us to study pseudo-holomorphic curves by projecting them to the base of the fibration.

Consider the projection of such a triangle to the base by $w : X(B)\to
X(I)$. This yields a 2-chain on the base with boundary on the
corresponding base paths $\ell(0),\ell(n),\ell(n+m)$ with corners at
the points $w(q_1),w(q_2),w(p)$. This projection is not necessarily embedded, so the next step is to pass to the universal cover of the
base. Let $\tilde{X}(I)$ denote infinite strip which is the universal
cover of the annulus, and let $\tilde{w}: \tilde{X}(B) \to
\tilde{X}(I)$ denote the induced fibration. We denote by $\pi:
\tilde{X}(B) \to X(B)$ the covering map. When drawing pictures in the
base $\tilde{X}(I)$, we can represent it as $[-1,1] \times \R$, with
the infinite direction drawn vertically. With this convention, the
path $\ell(d)$ lifts to a $\Z$-family of paths which have slope $-d$.

Figure \ref{fig:universal-cover} shows the universal cover of $X(I)$,
with the base paths for $L(0)$, $L(1)$, and $L(2)$.
\begin{figure}
\includegraphics[width=2in]{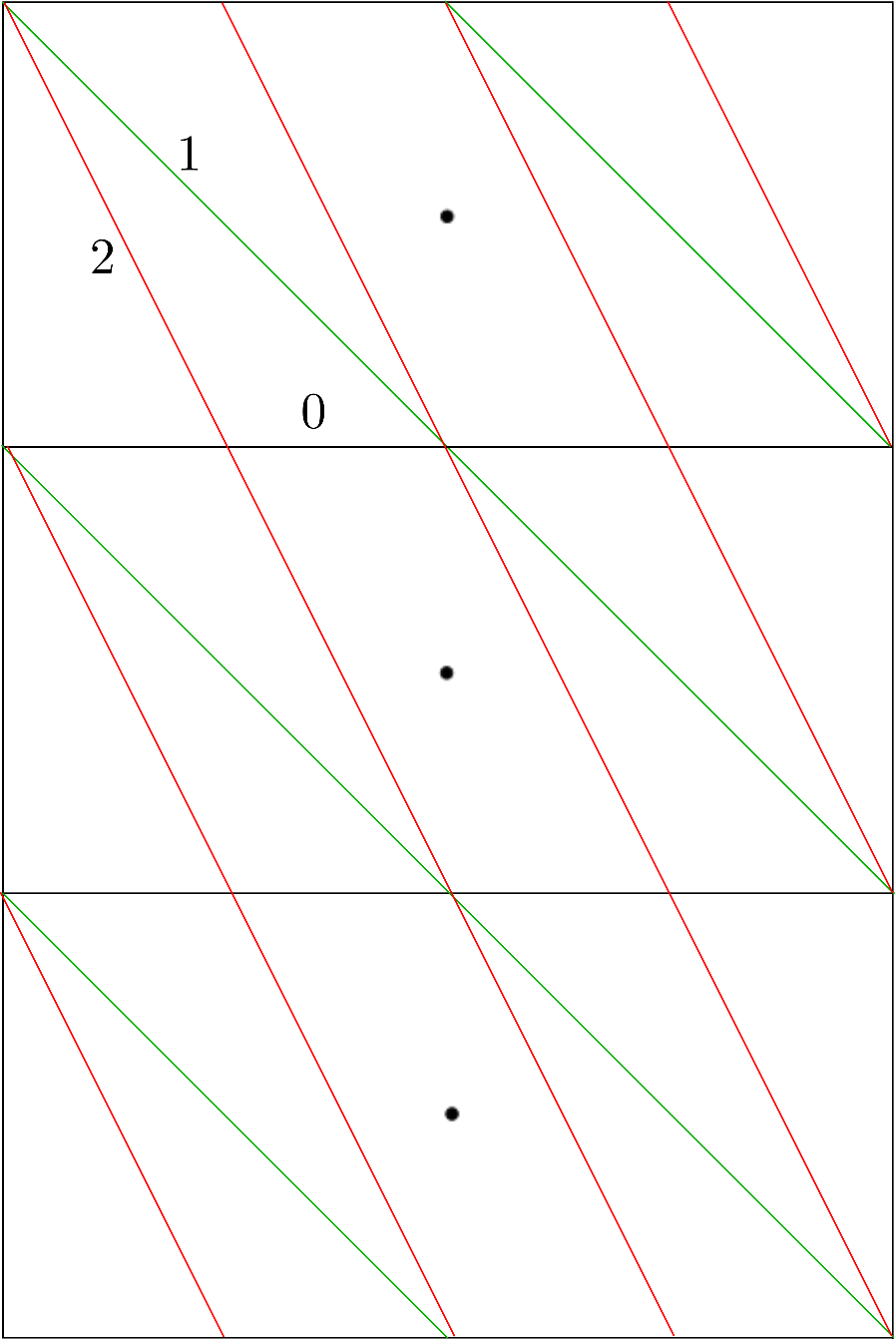}
\caption{The universal cover of $X(I)$.}
\label{fig:universal-cover}
\end{figure}

The choice of lift of $q_1$ determines a lift of $L(0)$ and $L(n)$,
which then determines a lift of $q_2$ and of $L(n+m)$, which in turn
determines where the lift of any $p$ must lie. By looking at the
slopes of the base paths $\ell(0),\ell(n),\ell(n+m)$ involved, we
obtain the following proposition:

\begin{proposition}
  In the terminology of Definition \ref{def:indexing}, Suppose that $q_1 =
  q_{a,i}$ lies in the fiber indexed by $a/n$, and that $q_2 =
  q_{b,j}$ lies in the fiber indexed by $b/m$. Then any $p$
  contributing to the product is $q_{a+b,h}$ for some $h$, that is, it
  lies in the fiber indexed by $(a+b)/(n+m)$
\end{proposition}

We can rephrase this proposition as saying that we can introduce a second
grading on $HF^*(L(d),L(d+n))$ where $HF^{*,a}(L(d),L(d+n))$ is
generated by $q_{a,i}$ for $i \in \{0,1,\dots\left\lfloor\frac{n-|a|}{2}\right\rfloor\}$, and that $\mu^2$ respects this grading.

We must now consider two cases: either the input generators $q_1$ and $q_2$ lie in different fibers of $w$, or they lie in the same fiber. The rest of this section is devoted to case where $q_1$ and $q_2$ lie in different fibers, which is the more difficulty one. In the case where $q_1$ and $q_2$ lie in the same fiber, the preceding proposition says that the output $p$ must also lie in the same fiber. This means that we can compute the contribution of $p$ to $\mu^2(q_2,q_1)$ by counting triangles that lie entirely in this fiber, since the projection of such a triangle to the base must be constant. We will employ an alternative strategy where we perturb one of the base paths, say $\ell(0)$, by a small amount so that there are no points where all three base paths intersect. This forces the triangle in the fiber to spread out over a small triangle in the base, so that in particular it cannot be constant. With this proviso, the arguments in this section apply without exception.

Now we show that any triangles contributing to the product of interest
are sections of the Lefschetz fibration $\tilde{w}: \tilde{X}(B)\to
\tilde{X}(I)$: 

\begin{proposition}
\label{prop:triangles-sections}
Let $u: S \to X(B)$ be a pseudo-holomorphic triangle contributing to
the component of $p$ in $\mu^2(q_2,q_1)$, such that $q_1$ and $q_2$ do not lie in the same fiber of $w$. Then there is a triangle $T$
in $\tilde{X}(I)$ bounded by appropriate lifts of $\ell(0),\ell(n)$
and $\ell(n+m)$, a holomorphic isomorphism $\tau: S\to T$, and a
pseudo-holomorphic section $s :T \to \tilde{X}(B)$, such that
$u = \pi\circ s \circ \tau$.

  Conversely, any pseudo-holomorphic section $s: T \to \tilde{X}(B)$
  with boundary on $L(0),L(n),L(n+m)$ which maps the corners to $q_2,
  q_1, p$ contributes to the coefficient of $p$ in $\mu^2(q_2,q_1)$.
\end{proposition}

\begin{proof}
  Since $u : S \to X(B)$ is a map from a simply-connected domain,
  there is a lift $\tilde{u}: S \to \tilde{X}(B)$. Since $q_1$ and
  $q_2$ are not in the same fiber, the image of $u$ cannot be
  contained entirely within a fiber. The triangle $T$ and the lifts of
  the $\ell(d)$ are determined by the considerations from the previous
  proposition. Clearly $\tilde{w}\circ \tilde{u}$ defines a 2-chain in
  $\tilde{X}(I)$, which by maximum principle is supported on $T$. By
  positivity of intersection with the fibers of $\tilde{w}$, all
  components of this 2-chain are positive, and the map $\tilde{w}\circ
  \tilde{u} : S\to T$ is a ramified covering. However, if the degree
  were greater than one, then $\partial S$ would have to wind around
  $\ell(0),\ell(n),\ell(n+m)$ more than once, contradicting the
  boundary conditions we placed on the map $u$.

  Since the projection $\tilde{w}: \tilde{X}(B) \to \tilde{X}(I)$ is
  holomorphic, the composition $\tilde{w}\circ \tilde{u}: S \to T$ is
  a holomorphic map which sends the boundary to the boundary and the
  punctures to the punctures, so it is a holomorphic isomorphism, and
  we let $\tau$ be its inverse.

  For the converse, uniformization for the disk with three boundary
  punctures yields a complex structure on $S$ and a map $\tau: S \to
  T$, such the composition $\tilde{u} = s \circ \tau$ is the desired
  triangle in $\tilde{X}(B)$. Composing this with $\pi :
  \tilde{X}(B)\to X(B)$ yields the triangle in $X(B)$.
\end{proof}

\begin{proposition}
  \label{prop:k-criticalpoints}
  Suppose $q_1= q_{a,i}$ lies in the fiber indexed by $a/n$, and $q_2
  = q_{b,j}$ in the fiber
  indexed by $b/m$. Then the sections in Proposition
  \ref{prop:triangles-sections} cover the critical values of the
  Lefschetz fibration $k$ times, where
  \begin{itemize}
  \item $k = 0$ if $a$ and $b$ are both non-negative or both non-positive.
  \item $k = \min(|a|,|b|)$ if $a$ and $b$ have different signs.
  \end{itemize}
\end{proposition}

\begin{proof}
  We identify $\tilde{X}(I)$ with $[-1,1]\times \R$. The critical
  values lie at the points $\{0\} \times (\Z + \frac{1}{2})$.

  If $a$ and $b$ are both non-negative or both non-positive, then the
  triangle $T$ is entirely to one side of the vertical line $\{0\}
  \times \R$ where all the critical values lie.

  Suppose $a$ and $b$ have opposite signs and $|a| \leq |b|$. Then the
  output point lies at $(a+b)/(n+m)$, which has the same sign as
  $b$. The side of $T$ corresponding to $\ell(n)$ crosses the line
  $\{0\}\times \R$ at $(0,a)$, while the side corresponding to
  $\ell(0)$ crosses at $(0,0)$, so the distance is $|a|$, and in fact
  the set $\{0\} \times (\Z+\frac{1}{2})$ contains $|a|$ points in
  this interval.

  If $|b| \leq |a|$, the output point at $(a+b)/(n+m)$ has the same
  sign as $a$, and so we need to look at where $\ell(n+m)$ intersects
  the line $\{0\}\times \R$. This happens at $(0,a+b)$, and the
  distance to $(0,a)$ is $|b|$.
\end{proof}

With the notation introduced so far, we can state the main result of
our computation for $\mu^2(q_{b,j},q_{a,i})$. 

\begin{proposition}
  \label{prop:counts-summary}
  Suppose that $q_{a,i} \in HF^0(L(0),L(n))$ and $q_{b,j} \in
  HF^0(L(n),L(n+m))$ as in Definition \ref{def:indexing}, and let $k$ be as in
  Proposition \ref{prop:k-criticalpoints}. Then
  \begin{equation}
    \mu^2(q_{b,j},q_{a,i}) = \sum_{s=0}^k \binom{k}{s}q_{a+b,i+j+s}
  \end{equation}
\end{proposition}

\begin{proof}
  This proposition is the combination of Propositions
  \ref{prop:TQFT-gluing}, \ref{prop:horiz-sections},
  \ref{prop:which-output}, \ref{prop:polygon-count}, and \ref{prop:signs}.
\end{proof}

The proof of Proposition \ref{prop:counts-summary} is where most of the technical work of this paper is spent, and it will occupy Sections \ref{sec:extend-fiber}--\ref{sec:signs}. Sections \ref{sec:extend-fiber} and \ref{sec:degenerated-spaces} present a technique for degenerating the total space of the Lefschetz fibration to break these counts into simpler pieces. Sections \ref{sec:horiz-sections} and \ref{sec:sections-polygon} compute these pieces, and Section \ref{sec:signs} determines the signs.

We shall now reformulate Proposition \ref{prop:counts-summary} in
algebro-geometric terms. Let
\begin{equation}
  \label{eq:homog-coord-ring}
  A = \bigoplus_{d=0}^\infty A_d = \bigoplus_{d=0}^\infty H^0(\PP^2, \cO_{\PP^2}(d)) \cong \K[x,y,z]
\end{equation}
be the homogeneous coordinate ring of $\PP^2$. We write elements of $A_d
= H^0(\PP^2,\cO_{\PP^2}(d))$ as degree $d$ homogeneous polynomials in
the variables $x,y,z$. Define
\begin{equation}
  \label{eq:P}
  p = xz - y^2,
\end{equation}
and set, for $a \in \{-d,\dots,d\}$, $i \in \{0,\dots,\left\lfloor\frac{d-|a|}{2}\right\rfloor\}$,
\begin{equation}
  \label{eq:big-Q}
  Q_{a,i} = \left\{
  \begin{aligned}
    x^{-a}p^iy^{d+a-2i}& \quad \text{if } a \leq 0\\
    z^ap^iy^{d-a-2i} & \quad \text{if } a > 0\\
  \end{aligned} \right\} \in A_d.
\end{equation}
\begin{proposition}
  \label{prop:polynomial-algebra}
  Take $Q_{a,i} \in A_n$ and $Q_{b,j} \in A_m$, then in $A$,
  \begin{equation}
    \label{eq:polynomial-algebra}
    Q_{a,i}Q_{b,j} = \sum_{s=0}^{k}\binom{k}{s} Q_{a+b,i+j+s}
  \end{equation}
  where $k = \min(|a|,|b|)$ if $a$ and $b$ have different signs, and
  $k = 0$ otherwise.
\end{proposition}
\begin{proof}
  The case where $a$ and $b$ have the same sign is obvious. 

  Suppose that $a \leq 0$ and $b \geq 0$, and suppose that $|a| \leq
  |b|$. Then we have $a+b \geq 0$, and $k = -a$.
  \begin{equation}
    Q_{a,i}Q_{b,j} = x^{-a}p^iy^{n+a-2i}z^bp^jy^{m-b-2j} =
    z^{a+b}(xz)^kp^{i+j}y^{n+m+a-b-2(i+j)}
  \end{equation}
  Since $xz = p + y^2$, we have 
  \begin{equation}
    (xz)^k = \sum_{s=0}^k \binom{k}{s}p^sy^{2(k-s)}
  \end{equation}
  \begin{equation}
    Q_{a,i}Q_{b,j} = \sum_{s=0}^k
    \binom{k}{s}z^{a+b}p^{i+j+s}y^{(n+m)-(a+b)-2(i+j+s)}
  \end{equation}
  and the monomial on the right is just $Q_{a+b,i+j+s}$. The other cases are similar.
\end{proof}

The operation $\mu^2$ determines a product, $q_2 \cdot q_1 =
(-1)^{|q_1|} \mu^2(q_2,q_1)$. The sign is present to connect the
conventions for an $A_\infty$-category at those of a dg-category; in
the case all morphisms have degree zero the sign is trivial. The
following proposition states how our Lagrangian intersections give
rise to a distinguished basis of the homogeneous coordinate ring $A$. From the preceding propositions, we can deduce Theorem \ref{thm:triangle-products}.

\begin{proposition}
  \label{prop:coord-ring-iso}
  The map $\psi_{d,n}: HF^0(L(d),L(d+n)) \to A_n$ defined by
  \begin{equation}
    \psi_{d,n}: q_{a,i} \mapsto Q_{a,i}
  \end{equation}
  is an isomorphism. We have
  \begin{equation}
    \psi_{d,n+m} (q_1 \cdot q_2) = \psi_{d,n}(q_1) \cdot \psi_{d+n,m}(q_2)
  \end{equation}
\end{proposition}

\begin{proof}
  That $\psi_{d,n}$ is an isomorphism is because it maps a basis to a
  basis. The other statement is the combination of 
Propositions \ref{prop:counts-summary} and
  \ref{prop:polynomial-algebra}.
\end{proof}

Thus the proof of Theorem \ref{thm:triangle-products} will be complete once we prove Proposition \ref{prop:counts-summary}, which in turn depends on  Propositions \ref{prop:TQFT-gluing}, \ref{prop:horiz-sections},
  \ref{prop:which-output}, \ref{prop:polygon-count}, and \ref{prop:signs}. The rest of this section is devoted to the proofs of these propositions.

\subsection{Extending the fiber}
\label{sec:extend-fiber}

One problem with our Lagrangian boundary conditions $L(d)$ is that
they intersect the horizontal boundary $\partial^h X(B)$. This raises
the possibility that as we degenerate the fibration, part of a
pseudo-holomorphic section can escape through $\partial^h X(B)$.

We now describe a technical trick that, by attaching bands to the
fiber, allows us to close up the Lagrangians for the purpose of a
particular computation, and thereby use only the results in the
literature on sections with Lagrangian boundary conditions disjoint
from the boundaries of the fibers. Proposition \ref{prop:extended-sections} states that this attachment does not change the spaces of holomorphic curves that we wish to count, because they do not enter the bands. On the other hand, when we degenerate the fibration, we find that some of these curves degenerate into curves that do enter the bands. 

% It appears that the way we do this makes no real difference, and in
% fact almost all of our arguments will concern curves which necessarily
% remain inside our original cylinder fibers. However, when we try to
% find the element $c \in HF^*(L,\tau(L))$ appearing in the Floer
% cohomology exact sequence, we will find a situation where sections
% actually can escape our original fiber, depending on the particular
% choice of perturbations.

The starting point for this construction is, given inputs $a_1$ and
$a_2$, to consider the portion of the fibration $\tilde{X}(B)|T \to T$
lying over the triangle $T$ in the base. We recall the assumption from
section \ref{sec:focus-focus-lefschetz} that the symplectic connection
is actually flat near the horizontal boundary. After passing to the
universal cover of the base, the symplectic monodromies around the
boundaries of $X(I)$ can be trivialized, and the fibration is actually
symplectically trivial near the horizontal boundary. We also assume
that the boundary intersections of our Lagrangians have been
positively perturbed as in Section \ref{sec:pert}, so that they do not
intersect at the boundary. After trivializing the fibration
near the horizontal boundary, we find that in each fiber $F_z$ of
$\tilde{X}(B)|T \to T$, there are six points on $\partial F_z$ (three
on either component), arising as the symplectic parallel transport of
the boundary points of $L(0)$, $L(n)$, and $L(n+m)$. These are the
points where the Lagrangians $L(0),L(n),L(n+m)$ are allowed to
intersect $\partial F_z$ (though  $L(d) \cap \partial F_z$ is
only nonempty if $z \in \ell(d)$ is on the appropriate boundary
component of the triangle $T$). The two sets of three points on each
component of $\partial F_z$ are matched according to which Lagrangian
they come from, and we extend the fiber $F_z$ to $\hat{F}_z$ by
attaching three bands running between the two components of $F_z$
according to this matching. We call the resulting fibration $\hat{X}_T
\to T$. We have an embedding $\iota: \tilde{X}(B)|T \to \hat{X}_T$.

Over the component of $\partial T$ where the Lagrangian boundary
condition $L(0)$ lies, we extend $L(0)$ to $\hat{L}(0)$, closing it up
fiberwise to a circle by letting it run through the corresponding band
in $\hat{F}_z$. Similarly we close up $L(n)$ and $L(n+m)$ to their
hat-versions, each passing through a different band.

Figure \ref{fig:extended-fiber} shows the cylinder with a band
attached. The actual extended fiber has three such bands.
\begin{figure}
\includegraphics[width=2in]{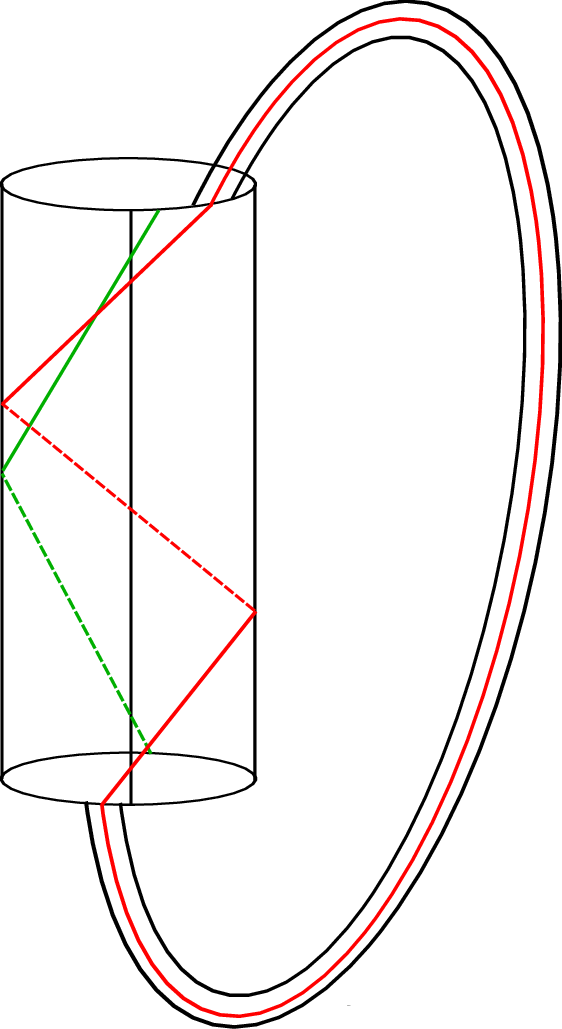} 
\caption{Attaching a band to close up one of the Lagrangians in the fiber.}
\label{fig:extended-fiber}
\end{figure}

It is immediate from the construction that $\hat{X}_T \setminus \image
(\iota)$ is symplectically a product. Let us use a complex structure
which is also a product in this region. Transversality can be achieved
using such structures because all intersection points lie in $\image
(\iota)$, and hence any pseudo-holomorphic section must also pass
through $\image (\iota)$, where we are free to perturb $J$ as usual. We
have the following proposition:

\begin{proposition}
\label{prop:extended-sections}
  Any pseudo-holomorphic section $s:T \to \hat{X}_T$ with boundary
  conditions $\hat{L}(d)$, $d = 0, n, n+m$ lies within $\image
  (\iota)$.
\end{proposition}

\begin{proof}
  Let $p: \hat{X}_T \setminus \image(\iota) \to \hat{F}$ denote the
  projection to the fiber, whose image consist of the bands. Let us
  consider one band containing the Lagrangian $L$. Introduce
  coordinates $(s,t) \in [0,a]\times [-1,1]$ on the band such that the
  part of the Lagrangian within the band is $L = [0,a]\times \{0\}$.
  
  Let $V \subset T$ be the preimage of the open band $(0,a)\times
  [-1,1]$ under $p \circ s$. The set $V$ is relatively open. Let
  $V^\circ = V \setminus \partial T$. As the map $p\circ s :V^\circ
  \to [0,a]\times[-1,1]$ is actually a holomorphic map between Riemann
  surfaces, the open mapping theorem implies that the image $W =
  (p\circ s)(V^\circ)$ is an open subset of $(0,a)\times
  (-1,1)$. Also, the boundary of $W$ is contained in the union of the
  Lagrangian $[0,a]\times \{0\}$ (where the boundary of the domain is
  sent) and the ends of the band $\{0,a\}\times (-1,1)$ (where
  image of $V$ connects to the rest of the holomorphic
  curve). The only such $W$ is the empty set.
\end{proof}

\subsection{Degenerating the fibration}
\label{sec:degenerated-spaces}

By propositions \ref{prop:triangles-sections} and
\ref{prop:extended-sections}, in order to compute the Floer product
between two morphisms $q_1 \in HF^0(L(0),L(n))$ and $q_2 \in
HF^0(L(n),L(n+m))$, it is just as good to count sections of the
fibration $\hat{X}_T \to T$ with Lagrangian boundary conditions
$\hat{L}(0)$, $\hat{L}(n)$, $\hat{L}(n+m)$.

In order to obtain these counts, we will consider a degeneration of
the Lefschetz fibration. We consider a one-parameter family of
Lefschetz fibrations $\hat{X}^r \to T^r$, which for $r = 1$ is
simply the one we started with. As $r$ goes to zero, the fibration
deforms so that all of the $k$ critical values contained within $T$
move toward the side of $T$ corresponding to $\ell(n)$. In the limit,
a disc bubble appears around each critical value, and at $r = 0$, the
base $T$ has broken into a $(k+3)$-gon, with $k$ ``new'' vertices
along the side corresponding to $\ell(n)$, each of which joins to a
disk with a single critical value. We can equip each fibration with
Lagrangian boundary conditions varying continuously with $r$, and
degenerating to a collection of Lagrangian boundary conditions for
each of the component fibrations at $r = 0$.

The base of the Lefschetz fibration undergoes the degeneration shown
in figure \ref{fig:degeneration}. This figure shows specifically the
case for the product of $x \in CF^*(L(0),L(1))$ and $z \in
CF^*(L(1),L(2))$. The marked point on the upper portion is the Lefschetz
critical value, while the marked points on the lower portion are the Lefschetz
critical value and a node.
\begin{figure}
\includegraphics[width=1.5in]{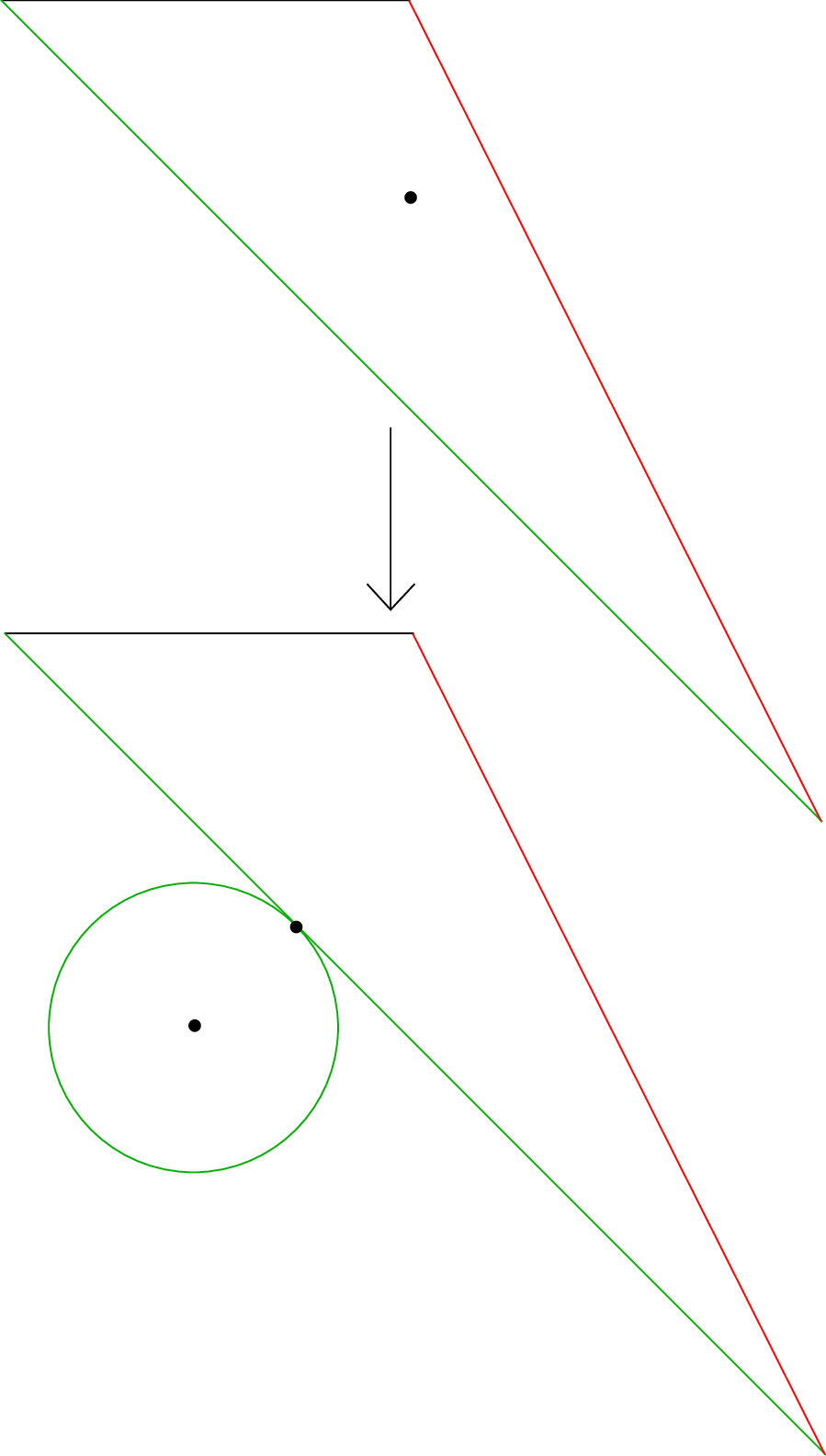} 
\caption{Degenerating the fibration.}
\label{fig:degeneration}
\end{figure}

When performing this construction carefully, it is better to describe
this family as a smoothing of the degenerate $r = 0$ end. Let $T^0$ be
a disk with $(k+3)$ boundary punctures. We emphasize that the
conformal structure of $T^0$ is fixed. Let $\hat{X}^0 \to T^0$ be a
symplectically trivial Lefschetz fibration. Let $D_1,\dots,D_k$ denote
the $k$ disks with one boundary puncture that are our ``bubbles''. For
each $i$ let $E_i \to D_i$ denote a Lefschetz fibration with a single
critical value, and which is trivial near the puncture. Let the
symplectic monodromy around $\partial D_i$ be denoted $\tau_i$.

We will equip each of the components of this fibration with a
Lagrangian boundary condition as follows.

\begin{itemize}
\item The base $T^0$ has one boundary component corresponding to
  $\ell(0)$, one boundary component corresponding to $\ell(n+m)$, and
  $(k+1)$ boundary components corresponding to $\ell(n)$. Since the
  fibration $\hat{X}^0 \to T^0$ is trivial, it suffices to describe
  each boundary condition in the fiber. Over the point where $\ell(0)$
  and $\ell(n+m)$ come together, we identify the fiber with that of
  $\hat{X}^1 \to T^1$, and take $\hat{L}(0)^0$ and $\hat{L}(n+m)^0$ to
  be the corresponding Lagrangians.

\item Over the $k+1$ boundary components of $T^0$ corresponding to
  $\ell(n)$, we construct a sequence $\hat{L}(n)^0_i$ of
  Lagrangians. At the puncture where $\ell(0)$ and $\ell(n)$ come
  together, we take $\hat{L}(n)^0_0$ to have the same position
  relative to $\hat{L}(0)^0$ (already constructed) that $\hat{L}(n)$
  has relative to $\hat{L}(0)$ in the original fibration. As we pass
  each of the new punctures where the disks are attached, the
  monodromies $\tau_i$ must be applied. So we let 
  \begin{equation}
  \label{eq:Ln-twists}
    \hat{L}(n)^0_i = \tau_i(\hat{L}(n)^0_{i-1})
  \end{equation}
  This can be done so that, over the puncture where $\ell(n)$ and
  $\ell(n+m)$ come together, $\hat{L}(n)^0_k$ and $\hat{L}(n+m)^0$
  intersect as the original $\hat{L}(n)$ and $\hat{L}(n+m)$ do.

\item Over the $k$ disks $D_i$, we take a Lagrangian boundary
  condition which interpolates between $\hat{L}(n)^0_{i-1}$ and
  $\hat{L}(n)^0_i$. This is possible because each $D_i$ contains a
  single Lefschetz critical value.
% by equation \eqref{eq:Ln-twists}.
\end{itemize}

Note that at this stage we only care about the twists $\tau_i$ up to
Hamiltonian isotopy, but we will make a particular choice in
\S\ref{sec:horiz-sections}, as required by the technical
considerations there. This gives us $\tau_i$ such that
$\hat{L}(n)^0_{i-1}$ and $\hat{L}(n)^0_i$ have one intersection point
over $i$-th new puncture of $T^0$.

Since the boundary conditions agree over the corresponding punctures,
we can glue the components over $T^0$ and $D_1,\dots,D_k$ with large
gluing length in order to obtain $\hat{X}^\epsilon \to T^\epsilon$ for
small $\epsilon > 0$, which has three boundary components and $k$
critical values all near the $\ell(n)$ boundary. There is a
family of Lefschetz fibrations interpolating between $\hat{X}^\epsilon
\to T^\epsilon$ back to our original $\hat{X}^1 \to T^1$, along which
the critical values move back to their original points. 

We equip $T^0$ and $D_1,\dots,D_k$ with complex structures, and equip $T^\epsilon$ with a family of complex structures $j^\epsilon$ converging to those in the limit. All total spaces are equipped with relatively compatible almost complex structures. 

Now we appeal to the gluing theorem \cite[Proposition 2.2]{seidel-les} telling us how the curve counts behave under the degeneration. 

\begin{proposition}
\label{prop:TQFT-gluing}  
  The count of pseudo-holomorphic sections of $\hat{X}^1 \to T^1$,
  with Lagrangian boundary conditions $\hat{L}(d)$, $d = 0 ,n ,n+m$ is
  obtained from the counts of sections of $\hat{X}^0 \to T^0$ and $E_i
  \to D_i$ by gluing together sections whose values over the punctures match.
\end{proposition}

\begin{proof}
Considering pseudo-holomorphic sections of the fibrations
$\hat{X}^\epsilon \to T^\epsilon$, as long as the gluing length is
large, the sections for $r = \epsilon$ will be obtained from sections
over $T^0$ and $D_1,\dots,D_k$ by gluing sections with matching values
at the punctures \cite[Proposition 2.2]{seidel-les}. During the deformation from $r = \epsilon$ to $r =
1$, no Floer strip breaking can occur because our Lagrangians
$\hat{L}(d)$, $d = 0, n, n+m$ do not bound any strips, even
topologically, so the count remains the same at $r = 1$.
\end{proof}

\subsection{Horizontal sections over a disk with one critical value}
\label{sec:horiz-sections}

We now determine the counts of pseudo-holomorphic sections of the
Lefschetz fibrations $E_i \to D_i$ using the theory of horizontal
sections developed by Seidel in \cite{seidel-les}. In particular we
will apply results from section 2.5 of that paper, so we adopt its
notation. This theory involves analysis of the symplectic connection, whose horizontal subspaces are the symplectic orthogonal subspaces to the fibers, and the symplectic parallel transport defined by horizontal lifting of paths in the base of the fibration.

In order for this technique to work, we need to set up carefully a
model Lefschetz fibration over a base $S$, the disk with one end, in
order to ensure that all the sections we need to count are in fact
horizontal. The key properties are:

\begin{itemize}
\item Transversality of the boundary conditions over the strip-like
  end. This means that we cannot use a standard model Dehn twist
  fibration, but need to introduce a perturbation somewhere.

\item Non-negative curvature. The curvature of the symplectic
  connection is a two-form on the base with values in functions on the
  fibers (the Lie algebra of the group of Hamiltonian
  diffeomorphisms), and we require that this two-form evaluated on a
  positive basis of the tangent space to the base yields a non-negative
  function on the fiber. The standard model Dehn twist fibration has
  non-negative curvature, but requiring the perturbation to have
  non-negative curvature imposes a constraint.

\item Vanishing action of horizontal sections. This imposes a further
  constraint on the perturbation.
\end{itemize}

As we progress through the construction, the definitions of all of
these terms will be recalled.

The first step is to construct a fibration which is flat away from the
critical point, following section 3.3 of \cite{seidel-les}. Let $S$ be
the Riemann surface $\{\re z \leq 0, |\im z| \leq 1\} \cup \{|z| \leq 1\}$ (a negative half-strip which has been rounded off with a
half-disk). 

Let $M$ denote the fiber, which is a cylinder with three bands
attached. The vanishing cycle $V$ is the equator of the cylinder. The
circle running through the core of one of these bands is $L$. Equip
$M$ with a symplectic form $\omega = d\theta$ such that $L$ is exact. Over $S^- =
(-\infty,-1]\times [-1,1]$, we let $\pi: E^- \to S^-$ be a trivial
fibration and equip $E$ with 1-form $\Theta$ and 2-form $\Omega =
d\Theta$ which are pulled back from the fiber (to get a symplectic
form we add the pullback a positive 2-form $\nu \in \Omega^2(S)$, but
this does not affect the symplectic connection). Following the pasting
procedure described in \cite[\S 3.3]{seidel-les} (though with the cut
on the right rather than the left), we can complete this fibration
$\pi: E \to S$ to one where the monodromy around the boundary is
$\tau_V$, a standard model Dehn twist supported near the vanishing
cycle. The result has non-negative curvature, is actually trivial on
the part of the fiber away from the support of the Dehn twist, and is
flat over the part of the base away from the critical value.

For this fibration, there is a Lagrangian boundary condition which
over the end corresponds to the pair $(L,\tau_V(L))$. Of course, these
are not transverse since $\tau_V$ is identity in the band. Hence we
will perturb the symplectic form by a term which depends on a
Hamiltonian. We will introduce the perturbation in a neighborhood of
the edge $(-\infty,-1]\times \{1\}$. Let $\beta$ be a cutoff function
which

\begin{itemize}
\item is supported in $U = \{s+it\mid -2\leq s \leq -1, 1-\epsilon \leq t
  \leq 1 \}$, 
\item vanishes along the bottom, left and right sides of $U$, and
\item has $\partial\beta/\partial t \geq 0$.
\end{itemize}

Figure \ref{fig:horizontal-sections} shows the base of the fibration; the region where $\beta$ has support is shaded.
\begin{figure}
\includegraphics[width=2.5in]{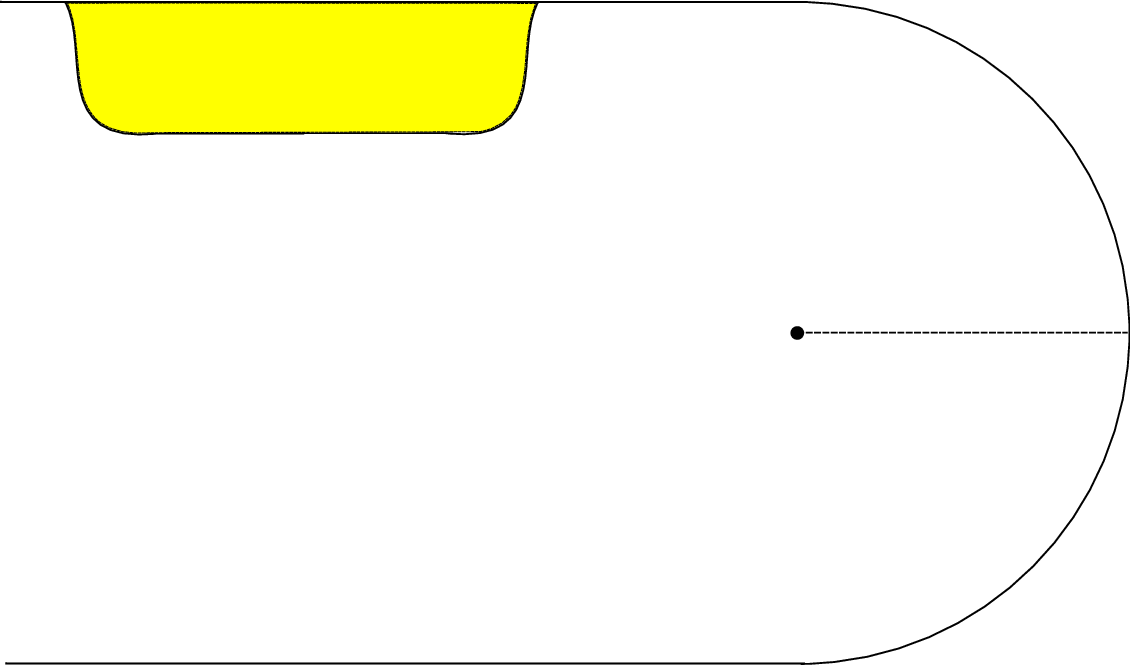} 
\caption{The base of the fibration with the region of perturbation shaded.}
\label{fig:horizontal-sections}
\end{figure}

Let $H$ be a Hamiltonian function on the fiber $M$, and let $X_H$ be
its vector field defined by $\omega(\cdot, X_H) = dH$. Then over
$S^-$, where our fibration was originally trivial with $\Omega =
d\Theta$ pulled back from the fiber, consider 
\begin{equation}
  \Theta' = \Theta + H \beta ds
\end{equation}
\begin{equation}
  \Omega' = d\Theta' = \Omega + \beta dH \wedge ds - H(\partial
  \beta/\partial t) ds\wedge dt
\end{equation}

This modifies the symplectic connection as follows: Let $Y \in TE^v$
denote the general vertical vector. Then if $Z \in TS$, and $Z^h$ is
its horizontal lift
\begin{equation}
  0= \Omega'(Y,Z^h) = \Omega(Y,Z^h) + \beta dH(Y)ds(Z) 
\end{equation}

If we denote by $Z$ again the horizontal lift with respect to the
trivial connection, we have $Z^h = Z - \beta ds(Z) X_H$. Since $\beta
\geq 0$, this means that as we parallel transport in the \emph{negative}
$s$-direction through the region $U$, we pick up a bit of the
Hamiltonian flow of $H$, compared with the trivial connection. By
adjusting the function $\beta$, we can ensure that the parallel
transport along the boundary in the positive sense picks up $\phi_H$, the time 1
flow of $H$.

We must compute the curvature of this connection. This is the 2-form
on the base with values in functions on the fibers given by
$\Omega'(Z_1^h,Z_2^h)$. It will suffice to compute for $Z_1
= \partial/\partial s$ and $Z_2 = \partial/ \partial t$, a positive
basis. We have $Z_1^h = \partial/\partial s - \beta X_H$ and $Z_2
= \partial/\partial t$.
\begin{equation}
\begin{split}
  \Omega'(Z_1^h,Z_2^h) &=\\
 \Omega(Z_1^h,Z_2^h) &+ \beta
  dH(Z_1^h)ds(Z_2^h) - \beta dH(Z_2^h)ds(Z_1^h) -
  H(\partial\beta/\partial t) ds\wedge dt(Z_1^h,Z_2^h)
\end{split}
\end{equation}
The first term vanishes because $Z_2^h$ is horizontal for the trivial
connection, the second term because $ds(Z_2^h) = 0$, and the third
because $dH(Z_2^h) = 0$. This leaves $-H(\partial\beta/ \partial
t)ds\wedge dt(Z_1^h,Z_2^h) = -H(\partial \beta/\partial t)$. By our
assumptions on $\beta$, this is non-negative as long as $H \leq 0$.

We equip the deformed fibration with a Lagrangian boundary condition $Q$
given by parallel transport of $L$ around the boundary. This picks up
a Dehn twist and the time 1 flow of $H$, so over the end we have the
pair $(L,\phi_H(\tau_V(L)))$.

A \emph{horizontal section} is a map $u: S \to E$ such that $Du(TS) =
TE^h$. The importance of such sections is that, while they are
determined by the symplectic connection, they are pseudo-holomorphic
for any \emph{horizontal complex structure} $J$, meaning an almost
complex structure that preserves $TE^h$ and for which the projection $E \to S$ is holomorphic.

The \emph{action} $A(u)$ of a section $u$ is defined to be $\int_S u^*
\Omega'$. The symplectic area of $u$ is then $A(u) + \int_S \nu$. The
identity relating action to energy is (\cite{seidel-les}, eq. 2.31)
\begin{equation}
  \label{eq:energy-action}
\frac{1}{2}  \int_S ||(Du)^v||^2 + \int_S f(u) \nu = A(u) + \int_S
||\bar{\partial}_J u||^2
\end{equation}
for any horizontal complex structure $J$. Here $Du = (Du)^h + (Du)^v$
is the splitting induced by the connection, and $f$ is the function
determined by the curvature as $f(\pi^*\nu|TE^h) = \Omega'|TE^h$. In
our example $f$ is supported near the critical point and in $\Supp
\beta$, where $f = -H(\partial\beta/\partial t)$. 

A direct consequence of \eqref{eq:energy-action} is the following:
\begin{proposition}
\label{prop:action-zero}
  If the curvature of $\pi: E \to S$ is non-negative, and if $u$ is a
  $J$-holomorphic section with $A(u) = 0$, then $u$ is horizontal
  and the curvature function $f$ vanishes on the image of $u$.
\end{proposition}

With all this in mind, we choose our Hamiltonian $H: M \to \R$ as
follows:
\begin{itemize}
\item $H \leq 0$ and $H = 0$ near $\partial M$. This ensures that the
  fibration is still trivial near the horizontal boundary and that the
  curvature is non-negative within $\Supp \beta$.
\item $H$ achieves its global maximum of $0$ near $\partial M$ and on
  the ``cocore'' of the band in $M$. It achieves its minimum along the
  vanishing cycle, and has no other critical points in the cylinder or
  lying on $L$ (which intersects the vanishing cycle and the cocore
  once). The first condition implies that horizontal sections passing
  through the cocore do not pick up any curvature, while the second is
  there in order to ensure that $(L,\phi_H(\tau_V(L)))$ is a
  transverse pair.
\item The flow $\phi_H$ has the property that $L \cap
  \phi_H(\tau_V(L))$ consists of one point $x$ lying on the cocore of
  the band. At this point, the tangent space to $\phi_H(\tau_V(L))$ is a small clockwise rotation of the tangent space to $L$. 

This implies that this intersection point has degree zero. To see this, recall from Section \ref{sec:gradings} that the fiber has foliation by circles. We extend this foliation into the attached bands, foliating each band by intervals. The Dehn twist $\tau_V$ preserves this foliation, so it admits a grading in the sense of \cite[(12i)]{seidel-book}, namely an $\R$-valued lift of the $S^1$-valued function $\alpha(D\tau_V(\Lambda))/\alpha(\Lambda)$ defined on the Lagrangian Grassmannian bundle. We can choose the lift to be zero in the bands, where $\tau_V$ is the identity. The Lagrangians $L$ and $\phi_H(\tau_V(L))$ do not rotate with this foliation, so they admit gradings. Thus, if we equip $L$ with some grading and $\phi_H(\tau_V(L))$ with the induced grading, the $\R$-valued grading functions for $L$ and $\phi_H(\tau_V(L))$ differ by a small amount. Since tangent space to $\phi_H(\tau_V(L))$ is a small clockwise rotation of the tangent space to $L$, the intersection point $x$ has degree zero.
\end{itemize}

After this setup, we come to the problem of determining the set
$\cM_J(x)$ of $J$-holomorphic sections $u: S \to E$ satisfying $u(\partial S) \subset Q$ which are
asymptotic to $x \in L \cap \phi_H(\tau_V(L))$ over the end.

\begin{proposition}
  \label{prop:horiz-sections}
  Let $J$ be a horizontal complex structure. Then $\cM_J(x)$ consists
  of precisely one section. It is horizontal, has $A(u) = 0$, and is regular.
\end{proposition}

\begin{proof}
  The first step is to construct a horizontal section. Any horizontal
  section, if it exists, is determined by parallel transport of the
  point $x \in L\cap \phi_H(\tau_V(L))$ throughout $E$. Consider the
  section over $S^-$ given by $u^-: (s,t) \mapsto (s,t,x)$. This is
  horizontal outside $\Supp \beta$, since the fibration is trivial
  over $S^- \setminus \Supp \beta$. In $\Supp \beta$, the fact that
  $x$ lies at a critical point of $H$ means that $(TE^h)_{(s,t,x)}$ is
  the same as for the trivial connection, so the section is horizontal
  there as well. Near the singularity, the fibration is trivial in the
  band where $x$ lies, so $u^-$ extends to a horizontal section $u: S
  \to E$.

  For this section, we compute $A(u) = \int_S u^*\Omega'$. Since $u$
  lies in the region where $\tau_V$ is trivial, $\int_{S \setminus
    \Supp \beta} u^* \Omega' = 0$. Since $u$ passes through the point
  $x$ where $H(x) = 0$, the contribution $\int_{\Supp \beta}
  u^*\Omega' = \int_{\Supp \beta} -H(x)(\partial \beta/\partial t)\nu$
  vanishes.

  Since $A(u) = 0$, any $u' \in \cM_J(X)$ must also have action
  $0$. To see this, note that $\Theta'$ is exact on each fiber of the boundary Lagrangian $Q$. This implies that $\Theta'|_Q = \pi^*\kappa_Q + dK_Q$, where $\kappa_Q$ is a one-form on $\partial S$, and $K_Q$ is a function on $Q$. Thus
  \begin{equation}
    A(u) = \int_S u^*\Omega' = \int_{\partial S} u^*\Theta' = \int_{\partial S} \kappa_Q + \int_{\partial S} dK_Q
  \end{equation}
  The first term on the right-hand side does not depend on $u$, and the second term, which is the net change in $K_Q$ along $u(\partial S)$, depends on the asymptotic data for $u$, which is the same for all $u \in \cM_J(X)$.

  Because the curvature of $\pi: E\to S$ is non-negative,
  proposition \ref{prop:action-zero} implies that $u'$ is
  horizontal. Since $u'$ and $u$ are both asymptotic to $x$, they are
  equal. Hence $\cM_J(x) = \{u\}$.

  It remains to show that $u$ is regular. The linearization of
  parallel transport along $u$ trivializes $u^*(TE^v)$ such that the
  boundary condition $u^*(TQ)$ is mapped to a family of Lagrangian
  subspaces which, as we traverse $\partial S$ in the positive sense,
  tilt clockwise by a small amount. That $\ind D_{u,J} = 0$ follows from Proposition 11.13 of \cite{seidel-book}.  On the
  other hand, Lemma 2.27 of \cite{seidel-les} applies to the section
  $u$, implying $\ker D_{u,J} = 0$. We could also appeal to Lemma 11.5 of \cite{seidel-book} (with $\mu(\rho_1) = 0$ and $|\Sigma^-| = 1$) to see that $\ker D_{u,J} = 0$. Hence $\coker D_{u,J} = 0$ as well.
\end{proof}

\subsection{Polygons with fixed conformal structure}
\label{sec:sections-polygon}

Recall from section \ref{sec:degenerated-spaces} the trivial fibration
$\pi: \hat{X}^0 \to T^0$, with fiber $M$, where $T^0$ is a a disk with
$(k+3)$ boundary punctures. We have $\hat{X}^0 = M \times T^0$
symplectically. We equip $\hat{X}^0$ with a product almost complex
structure $J = J_M \times j$, where $j$ is the complex structure on $T^0$.

\begin{proposition}
  \label{prop:sections-polygon}
  The $(j,J)$-holomorphic sections $u: T^0 \to \hat{X}^0$ are in
  one-to-one correspondence with $(j,J_M)$-holomorphic maps $u_M: T^0
  \to M$.
\end{proposition}

\begin{proof}
  Given $u: T^0 \to \hat{X}^0$, write $u =(u_M, u_{T^0})$ with respect
  to the product splitting. Since $J$ is a product each component is
  pseudo-holomorphic in the appropriate sense. But $u_{T^0}$ is the
  identity map because $u$ is a section. This correspondence is
  clearly invertible.
\end{proof}

This reduces the problem of counting sections $u:T^0 \to \hat{X}^0$ to
the problem of counting maps $u_M: T^0 \to M$. We emphasize that we
are counting maps from a Riemann surface with a fixed conformal
structure.

The maps $u_M:T^0 \to M$ are holomorphic maps between Riemann surfaces,
and hence their classification is mostly combinatorial. However,
because we are in a situation where the conformal structure on the
domain is fixed, we are not quite in the situation described, for
example, in \cite[\S 13]{seidel-book}. The holomorphic curves we are
looking for have non-convex corners and hence boundary branch points
or ``slits,'' and if the situation is complicated enough they may also
have branch points in the interior. However, the condition that the
conformal structure of the domain is fixed makes this an index zero
problem, which is to say it prevents these slits and branch points
from deforming continuously. The question is then, given a
combinatorial type of such a curve, what conformal structures (with
multiplicity) can be realized by holomorphic representatives?

 The $(k+3)$ boundary components of $T^0$ are equipped
with Lagrangian boundary conditions $\hat{L}(0)^0, \hat{L}(n)^0_0,
\dots, \hat{L}(n)^0_k, \hat{L}(n+m)^0$. Recall that over the ends
where disk bubbles are attached we have $\hat{L}(n)^0_i = \tau_i(
\hat{L}(n)^0_{i-1})$, where $\tau_i$ is the monodromy around the
$i$-critical point inside the $i$-th disk bubble.
In section \ref{sec:horiz-sections}, we refined the construction and
made a particular choice for this monodromy:
\begin{equation}
  \label{eq:monodromies-equal}
  \tau_i = \phi_H \circ \tau_V
\end{equation}
Since all of these symplectomorphisms are isotopic, we will denote
them all by $\tau$ for most of this section.  

In order to simplify notation, for the rest of this section
\ref{sec:sections-polygon} we will drop the hats and superscript zeros
and denote by
\begin{equation}
  L(0),L(n),\tau L(n), \tau^2 L(n),\dots \tau^k L(n), L(n+m) \subset M
\end{equation}
the Lagrangians in the fiber $M$ which give rise to the boundary
conditions over the $(k+3)$-punctured disk $T$. Though the monodromies
are all denoted by $\tau$, we actually choose the perturbations
slightly differently so as to ensure that this collection of
Lagrangians is in general position in $M$; this is necessary for the
argument in Lemma \ref{lem:existence}.

Recall that $q_1 = q_{a,i}$ and $q_2 = q_{b,j}$ are the morphisms
whose product we wish to compute. We now regard $q_{a,i} \in L(0)\cap
L(n)$ and $q_{b,j}\in \tau^kL(n)\cap L(n+m)$. Let $x_i \in
\tau^{i-1}L(n) \cap \tau^i L(n)$ denote the unique intersection point.

Recall that the possible output points $q_{a+b,h} \in L(0)\cap L(n+m)$
are indexed by $h \in \{0,1,\dots,
\left\lfloor\frac{(n+m)-|a+b|}{2}\right\rfloor\}$.

We can now state the main results of this section.

\begin{proposition}
\label{prop:which-output}
If $h$ is such that $0 \leq h-(i+j) \leq k$, then there are
$\binom{k}{h-(i+j)}$ homotopy classes of maps $u:T\to M$ satisfying
the boundary conditions and asymptotic to
$q_{a,i},x_1,\dots,x_k,q_{b,j},q_{a+b,h}$ at the punctures. For $h$
outside this range there are no holomorphic maps satisfying these
conditions.
\end{proposition}

\begin{proposition}
\label{prop:polygon-count}
For each feasible homotopy class from Proposition
\ref{prop:which-output}, and for each complex structure $j$ on $T$,
there is exactly one holomorphic representative $u: T \to M$.
\end{proposition}

The strategy of proof is first to prove Proposition
\ref{prop:which-output}, which is done in \S\ref{sec:homotopy-classes}. Then we show the existence of holomorphic
representatives for some conformal structure (\S\ref{sec:existence}), and show that the
number of representatives does not depend on the conformal
structure. By degenerating the domain we are able to show uniqueness (\S\ref{sec:moduli-space}).

\subsubsection{Homotopy classes}
\label{sec:homotopy-classes}

The analysis of homotopy classes requires some explicit combinatorics,
which we shall now set up. Recall that $M$ is a cylinder with three
bands attached, one for each of $L(0),L(n),L(n+m)$. We will classify
homotopy classes of maps $u: T\to M$ by their boundary loop $\partial
u$, which must be contractible, traverse
$L(0),L(n),\dots,\tau^kL(n),L(n+m)$ in order, and hit the intersection
points $q_{a,i},x_1,\dots,x_k,q_{b,j},q_{a+b,h}$ in order. That they
arise as the boundary of a holomorphic map controls their behavior
within the bands.

Let us use $L(n)$ to frame the cylinder, so that winding around the
cylinder is computed with respect to $L(n)$. Let $x \in M$ be a
basepoint which is located in the band near the intersection points
$x_r$. Let $\alpha \in \pi_1(M,x)$ denote a loop that enters the
cylinder from the bottom, veers right, winds around once, and goes
back downward to $x$. We also have a class $\beta \in \pi_1(M)$ that is
represented by the loop $L(n)$, oriented upward through the
cylinder. Note that $\alpha$ and $\beta$ generate a free group in
$\pi_1(M)$.

\begin{lemma}
  $L(0)$ winds $-(n-|a|)/2$ times relative to $L(n)$. $\tau^r L(n)$ winds $r$ times relative to $L(n)$. $L(n+m)$ winds $(m-|b|)/2 + k$ times relative to $L(n)$.
\end{lemma}

\begin{proof}
We are using the symbols $L(0), \tau^rL(n), L(n+m)$ to represent certain Lagrangians in a single fiber $M$, but we can compute the windings by comparing our situation to certain fibers of the original fibration. In the original fibration, $L(0)$ and $L(n)$ intersect in the fiber containing $q_{a,i}$. Thus we have that $L(0)$ winds $-(n-|a|)/2$ times relative to $L(n)$. Now $\tau^r L(n)$ winds $r$ times relative to $L(n)$ by construction. Finally, in the original fibration $L(n+m)$ and $L(n)$ intersect in the fiber containing $q_{b,j}$, and the relative winding in that fiber is $(m-|b|)/2$. In the current context, this becomes the winding of $L(n+m)$ relative to $\tau^k L(n)$. Thus the winding of $L(n+m)$ relative to $L(n)$ in current context is $(m-|b|)/2 + k$.
\end{proof}

The only unknown is how many times the boundary path traverses $L(0)$,
$L(n)$, $\tau L(n)$, etc., as we traverse the boundary in the positive
sense. The argument from Proposition \ref{prop:extended-sections}
shows that a holomorphic map cannot enter the bands corresponding to
$L(0)$ and $L(n+m)$. Hence the portion of our loop along $L(0)$ and
$L(n+m)$ lies within the cylinder, and therefore it is determined by
the choices of $q_{a,i}$, $q_{b,j}$ and $q_{a+b,h}$.

As for the portion of the loop along $\tau^r L(n)$, this can be
represented by a sequence of integers $(\delta_r)_{r=0}^k$, where
$\delta_r$ represents the number of times we wind around $\tau^r
L(n)$.

With these conventions in place, we can compute the winding of a
choice of paths satisfying the boundary and asymptotic conditions. We
record the parts of the computation:

\begin{itemize}
\item Passing from $x_k$ to $q_{b,j}$ by a short upward path on $\tau^k L(n)$
contributes
\begin{equation}
  \left(1-\frac{j}{(m-|b|)/2}\right)\left(k\right)
\end{equation}

\item The winding around the cylinder along $L(n+m)$ as we pass from
  $q_{b,j}$ to $q_{a+b,h}$ is
\begin{equation}
  \left[\frac{h}{(n+m-|a+b|)/2} - \frac{j}{(m-|b|)/2}\right]
  (-1)((m-|b|)/2 + k)
\end{equation}

\item The winding around the cylinder along $L(0)$ as we pass from
  $q_{a+b,h}$ to $q_{a,i}$ is 
  \begin{equation}
    \left[\frac{i}{(n-|a|)/2} - \frac{h}{(n+m-|a+b|)/2}\right] (n-|a|)/2
  \end{equation}

\item Passing from $q_{a,i}$ to $x_1$ by a short downward path on $L(n)$ contributes $0$
to the winding around the cylinder. 
\end{itemize}

Adding up these contributions and using the fact that $(m-|b|)/2 +
(n-|a|)/2 + k = (n+m-|a+b|)/2$, we get a total of $i+j - h+ k$. Thus,
if we go up on $\tau^k L(n)$ and down on $L(n)$, we pick up the class
$\alpha^{i+j-h+k} \in \pi_1(M,x)$ for the loop from $x_k$ to $x_1$.

Thus, the class $\alpha^{i+j-h+k}$ corresponds to the choice
$\delta_r = 0$ for $0 \leq r \leq k$. The homotopy class of any
other path can be computed from this as follows:

\begin{itemize}

\item Taking another path on $L(n)$ contributes a factor
  $\beta^{\delta_0}$ on the right.

\item Passing from $x_{r}$ to $x_{r+1}$ along $\tau^r L(n)$
  contributes the word
  \begin{equation}
    (\alpha^r\beta)^{\delta_r}
  \end{equation}
  where $\delta_r \in \Z$, and this class is added on the right.

\item Going down on $\tau^kL(n)$ rather than up contributes
  the class $(\alpha^k\beta)^{\delta_k}$ on the left, which up to
  conjugation is the same as adding the class $(\alpha^k\beta)^{\delta_k}$ on the right.
\end{itemize}
% Putting this together, and splitting into cases on the orientations of
% $L(n)$ and $\tau^k L(n)$, we obtain the
% following boundary classes in $\pi_1(M,x)$:

% \begin{itemize}
% \item Up on $\tau^kL(n)$, down on $L(n)$:
% \begin{equation}
%   \alpha^{i+j-h + k} \left\{\prod_{r=1}^{k-1}
%   (\alpha^r\beta)^{\delta_r}\right\}
% \end{equation}
% \item Up on $\tau^kL(n)$, up on $L(n)$:
% \begin{equation}  
%   \alpha^{i+j-h+k} \beta \left\{\prod_{r=1}^{k-1}
%   (\alpha^r\beta)^{\delta_r}\right\}
% \end{equation}
% \item Down on $\tau^kL(n)$, down on $L(n)$:
%   \begin{equation}
%     \alpha^{i+j-h+k}  \left\{\prod_{r=1}^{k-1}
%   (\alpha^r\beta)^{\delta_r}\right\}(\alpha^k\beta)^{-1}
%   \end{equation}
% \item Down on $\tau^kL(n)$, up on $L(n)$:
% \begin{equation}
%     \alpha^{i+j-h+k} \beta \left\{\prod_{r=1}^{k-1}
%   (\alpha^r\beta)^{\delta_r}\right\}(\alpha^k\beta)^{-1}
% \end{equation}
% \end{itemize}
% These four cases can be unified setting $\delta_0 = 0$ and $\delta_k =
% 0$ in the first case, $\delta_0 = 1$ and $\delta_k = 0$ in the second,
% $\delta_0 = 0$ and $\delta_k = -1$ in the third, and $\delta_0 = 1$ and
% $\delta_k = -1$ in the fourth.

Thus the class in question is
\begin{equation}
  \label{eq:homotopy-class}
  \alpha^{i+j-h+k} \prod_{r = 0}^k (\alpha^r \beta)^{\delta_r}
\end{equation}
 The key condition is that this class is trivial in
$\pi_1(M)$. This means in particular that all of the $\beta$'s must
cancel out. Because $\alpha$ and $\beta$ generate a free group we have
the following:

\begin{lemma}
\label{lem:beta-cancel}
  In the word \eqref{eq:homotopy-class}, the $\beta$'s cancel out if
  and only if $\delta_r \in \{-1,0,1\}$ for $0 \leq r \leq k$, the
  first nonzero $\delta$ is $1$, the last nonzero $\delta$ is $-1$, and
  the nonzero $\delta$'s alternate in sign.
\end{lemma}

\begin{proof}
  The first thing to see is that $|\delta_r| \leq 1$ for $1 \leq r
  \leq k$. This is because $(\alpha^r\beta)^2 =
  \alpha^r\beta\alpha^r\beta$ contains an isolated $\beta$, while
  $(\alpha^r \beta)^{-2}$ contains an isolated $\beta^{-1}$. Then we
  can see that the nonzero $\delta$'s must alternate sign, since
  having two consecutive $\delta$'s equal to $1$ yields
  $\alpha^{r_1}\beta\alpha^{r_2}\beta$, which has an isolated beta,
  while having two consecutive $\delta$'s equal to $-1$ would yield an
  isolated $\beta^{-1}$.

  Since the nonzero $\delta$'s in the range $1 \leq r \leq k$
  alternate in sign, all the $\beta$'s coming from the
  range $1 \leq r \leq k$ cancel, except for possibly the first or the
  last. This implies that $|\delta_0| \leq 1$ as well, since the only
  thing that can cancel $\beta^{\delta_0}$ is the first non-identity factor or
  the last non-identity factor. 

  Now the $\beta$ from the first non-identity factor can only cancel
  the $\beta$ from the last non-identity factor if all the $\alpha$'s
  as well as $\beta$'s in between cancel. This means that
    \begin{equation}
    \sum_{r = a}^b r\delta_r = 0
  \end{equation}
  for the appropriate range of $r$: $a \leq r \leq b$. Since the
  $\delta$'s are in $\{-1,0,1\}$ and they alternate in sign, the only
  solution to this equation is when all $\delta = 0$. In this case,
  the first and the last non-identity factors are in fact
  consecutive, the first has $\delta = 1$, while the last has $\delta
  = -1$. This shows that it is impossible to have $\beta^{-1}$ from
  the first factor cancel a $\beta$ from the last factor.

  In general, we find that each $\beta$ is canceled by a $\beta^{-1}$
  from the next non-identity factor, so that the nonzero $\delta$'s
  alternate sign for $0 \leq r \leq k$, the first nonzero $\delta$ is
  $1$, and the last nonzero $\delta$ is $-1$.
\end{proof}

By passing to $H_1(M;\Z)$, we obtain
the relations
\begin{equation}
\label{eq:sum-zero}
  \sum_{r=0}^k \delta_r = 0
\end{equation}
\begin{equation}
\label{eq:get-h}
  i+j-h+k + \sum_{r=0}^k r\delta_r =0
\end{equation}
Equation \eqref{eq:sum-zero} is implied by Lemma \ref{lem:beta-cancel},
while \eqref{eq:get-h} determines which $h$ the homotopy class
contributes to.

The sequences $(\delta_r)_{r=0}^k$ which solve the constraints are in
one-to-one correspondence with sequences $(s_r)_{r=0}^{k-1}$ such that
$s_r \in \{0,1\}$. In one direction, we extend the sequence by $s_{-1}
= 0 = s_k$, and set
\begin{equation}
  \delta_r = s_r - s_{r-1}
\end{equation}
In the other direction, given $\delta_r$, we can solve for $s_r$ using
this equation the initial condition $s_{-1}=0$. Since
the signs of the nonzero $\delta_r$ alternate, and the first nonzero
term is $1$, we will have $s_r \in \{0,1\}$, and since the final
nonzero term is $-1$, we will have $s_k = 0$, thus inverting the
correspondence. This yields $2^k$ solutions.

Plugging this into the summation in \eqref{eq:get-h}, we have
\begin{equation}
  \sum_{r=0}^k r\delta_r = \sum_{r=0}^k r(s_r-s_{r-1}) =
  \sum_{r=0}^{k-1} (-1)s_r
\end{equation}
because the summation telescopes. This is simply minus the number of
$1$'s in the sequence $s_r$. Thus we obtain
\begin{equation}
  \label{eq:get-h-2}
  h-(i+j) = k - \sum_{r=0}^{k-1}s_r
\end{equation}
The right hand side is always an integer between $0$ and $k$, and it
takes the value $s$ for $\binom{k}{s}$ choices of the sequence
$(s_r)_{r=0}^{k-1}$. Thus homotopy classes of maps exist for $h$ such
that $0 \leq h-(i+j) \leq k$, and there are $\binom{k}{h-(i+j)}$ such
classes. This proves Proposition \ref{prop:which-output}.

\subsubsection{Existence of holomorphic representatives for some
  conformal structure}
\label{sec:existence}

The first step in characterizing the holomorphic representatives of
these homotopy classes is to prove the existence of holomorphic
sections for some conformal structure. This is also essentially combinatorial.

We begin with some general concepts that will be useful in the proof.

\begin{definition}
  Let $\gamma: S^1 \to \C$ be a piecewise smooth loop. A
  \emph{subloop} $\gamma'$ of $\gamma$ is the restriction $\gamma' =
  \gamma|\bigcup_\alpha I_{\alpha}$ to a collection of intervals
  $\bigcup_\alpha I_{\alpha}$. The indexing set inherits a cyclic
  order from $S^1$, and we require that for each $\alpha$,
  $\gamma(\max I_{\alpha}) = \gamma(\min I_{\alpha+1})$ is a
  self-intersection of $\gamma$. Thus $\gamma'$ simply ``skips'' the
  portion of $\gamma$ between $\max I_\alpha$ and $\min
  I_{\alpha+1}$. A subloop is called \emph{simple} if it is non-self-intersecting.
\end{definition}

Note that a subloop is not the same as a loop formed by segments of
$\gamma$ joining self-intersections. Such an object is only a subloop
if the segments appear in a cyclic order compatible with $\gamma$.

\begin{definition}
  A piecewise smooth loop $\gamma: S^1 \to \C$ is said to have the
  \emph{(weak) positive winding property} (PWP) if the winding number of
  $\gamma$ around any point in $\C \setminus \image(\gamma)$ is
  non-negative. The loop $\gamma$ is said to have the \emph{strong
    positive winding property} (SPWP) if every subloop $\gamma' \subset
  \gamma$ has the positive winding property.
\end{definition}

\begin{lemma}
  \label{lem:simple-spwp}
  A loop $\gamma: S^1 \to \C$ has SPWP if and only if every
  \emph{simple} subloop has PWP.
\end{lemma}

\begin{proof}
  The ``only if'' direction is contained in the definition. Suppose
  that every simple subloop of $\gamma$ has PWP. If $\gamma'$ is a
  subloop that is not simple, then by splitting $\gamma'$ at a
  self-intersection, we can write $\gamma'$ as the composition of two
  proper subloops. Repeating this inductively, we can write $\gamma'$
  as the composition of simple subloops. By hypothesis, each of these
  subloops winds positively, and the winding of $\gamma'$ about a
  point is the sum of the contributions from each of the simple subloops.
\end{proof}

The next lemma shows that the strong positive winding property is
stable under branched covers. Suppose $\gamma$ is a loop and $y \in \C
\setminus \image(\gamma)$ is a point where the winding number of
$\gamma$ is $m > 1$. Taking the $m:1$ cover branched at $y$, we can
lift $\gamma$ to a path $\tilde{\gamma}$. The path $\tilde{\gamma}$ is
a loop that covers $\gamma$.
\begin{lemma}
  \label{lem:spwp}
  Suppose $\gamma$ has SPWP. Then $\tilde{\gamma}$ has SPWP.
\end{lemma}

\begin{proof}
  Suppose that $\tilde{\gamma}$ does not have SPWP. Then some subloop
  $\tilde{\gamma}'$ does not have PWP. By Lemma \ref{lem:simple-spwp},
  we may take $\tilde{\gamma}'$ to be a \emph{simple} subloop. Thus
  $\tilde{\gamma}'$ winds around some region once clockwise, and we
  have $\tilde{\gamma}' = \partial \tilde{C}$, where $\tilde{C}$ is
  the chain consisting of this region with coefficient $-1$. Pushing
  $\tilde{\gamma}'$ and $\tilde{C}$ forward under the branched cover,
  we obtain a subloop $\gamma' \subset \gamma$, and chain $C$ such
  that $\gamma' = \partial C$. Since $\tilde{C}$ is purely negative,
  no cancellation can occur when we push forward, and $C$ is purely
  negative as well. Thus $\gamma'$ winds negatively about a point in
  the support of $C$, which contradicts SPWP for $\gamma$.
\end{proof}

\begin{lemma}
  \label{lem:spwp-mapping}
  Suppose $\gamma: S^1 \to \C$ is a piecewise smooth loop with
  SPWP. Then there is a map $u:D^2 \to \C$ holomorphic on the interior of $D^2$ such that
  $\partial[u] = \gamma$.
\end{lemma}
\begin{proof}
  The idea of the proof is to iteratively take branched covers of the
  plane and lift $\gamma$ so as to reduce the density of winding. So
  let $y \in \C \setminus \image(\gamma)$ be a point where the winding
  number is $m > 1$. Taking the $m:1$ branched cover at $y$, we obtain
  as in Lemma \ref{lem:spwp} a lift $\tilde{\gamma}$ that covers
  $\gamma$ once and has SPWP. Repeating this process and using Lemma
  \ref{lem:spwp} to guarantee that the lift always has SPWP,
  eventually we obtain a \emph{simple} piecewise smooth loop with
  positive winding. This loop bounds a simply connected region in $M$, where $M$ is the Riemann surface resulting
  from the branched covering construction. By the Riemann mapping theorem, there is a map from this region to the unit disk $D^2$, biholomorphic on the interior and continuous on the closure, that maps each boundary segment to an arc in $\partial D^2$ and each corner to a point on $\partial D^2$. Let $\tilde{u} : D^2 \to M$ be the inverse of that map. Composing $\tilde{u}$ with
  the covering $M \to \C$ yields the desired map $u$.
\end{proof}

Fix a choice of homotopy class $\phi$ of polygons $u: T\to M$, which
essentially means fixing a choice for the sequence
$(\delta_r)_{r=0}^k$. Passing to the universal cover $\tilde{M}$ of
the fiber, fix a choice of lift $\tilde{u}:T \to \tilde{M}$. Let $y
\in \tilde{M}$ be any point. Because the boundary loop $\partial[u]$
is contractible in $M$, it lifts to a closed loop
$\partial[\tilde{u}]$ in $\tilde{M}$. In fact $\partial[\tilde{u}]$ is
contained in a domain which is isomorphic to a domain in $\C$, and with
this identification, we have the following.
% The image of $\partial[\tilde{u}]$ in
% $\pi_1(\tilde{M}\setminus\{y\}) \cong \Z$ (generated by a small
% counterclockwise loop around $y$) yields an element $w(y,\partial[\tilde{u}]) \in \Z$, which is the winding
% number of $\partial[\tilde{u}]$ around $y$.

\begin{lemma}
  \label{lem:positive-winding}
  The boundary loop $\partial[\tilde{u}]$ has SPWP.
\end{lemma}

\begin{proof}
  The key observation is that the slopes of the paths $L(0),
  L(n),\dots,\tau^k L(n), L(n+m)$ through the cylinder are positive
  and monotonically decreasing. If we frame the cylinder using $L(0)$,
  then
  \begin{itemize}
  \item $L(0)$ has slope $\infty$
  \item $L(n)$ has slope $[(n-|a|)/2]^{-1}$
  \item $\tau^r L(n)$ has slope $[r+(n-|a|)/2]^{-1}$
  \item $L(n+m)$ has slope $[(n+m-|a+b|)/2]^{-1}$.
  \end{itemize}
  Only at the intersection between $L(n+m)$ and $L(0)$ does the slope
  increase.

  Hence as we traverse $\partial[\tilde{u}]$, or any subloop thereof,
  the slope can only increase at one point. This point is either where
  the subloop either uses or skips over $L(0)$.

  Now we use some elementary plane geometry. Suppose that $P$ is an
  oriented polygonal path in the plane (possibly self-intersecting),
  all of whose sides $(S_i)_{i=1}^N$ have positive slope. Suppose that
  $P$ winds negatively around some point $y$. 
We claim the slope has to increase at no fewer than two vertices. We prove this by induction on the number of sides $N$.

Checking cases proves the claim when $N=3$ and $P$ is a
  triangle. See figure~\ref{fig:spwp}(a).

  Suppose
  for induction that the claim is true for $N < N_0$. Now we prove the claim if $P$ is a \emph{simple} $N_0$-gon, so assume for a contradiction that $P$ is a simple $N_0$-gon that winds negatively and has at most one vertex where the slope increases. We assume that $N_0 > 3$, so there must be a side $S_i$ such that at both ends there is a decrease in slope. We remove the
  side $S_i$ from the $N_0$-gon $P$ and extend the two incident sides
  $S_{i-1}$ and $S_{i+1}$ in order to obtain an $(N_0-1)$-gon $P'$. Since $P$ has at most one vertex where the slope increases, so does $P'$: Since the slope decreased at both $S_{i-1}S_i$ and at
  $S_iS_{i+1}$, it will decrease at the new vertex
  $S_{i-1}S_{i+1}$. See Figure~\ref{fig:spwp}(b). Since the original simple loop $P$ winds
  negatively, the new loop $P'$ must also wind negatively around some point. This
  contradicts the induction hypothesis. Now we have established the claim for simple $N_0$-gons.

  \begin{figure}
    \includegraphics[width=3.75in]{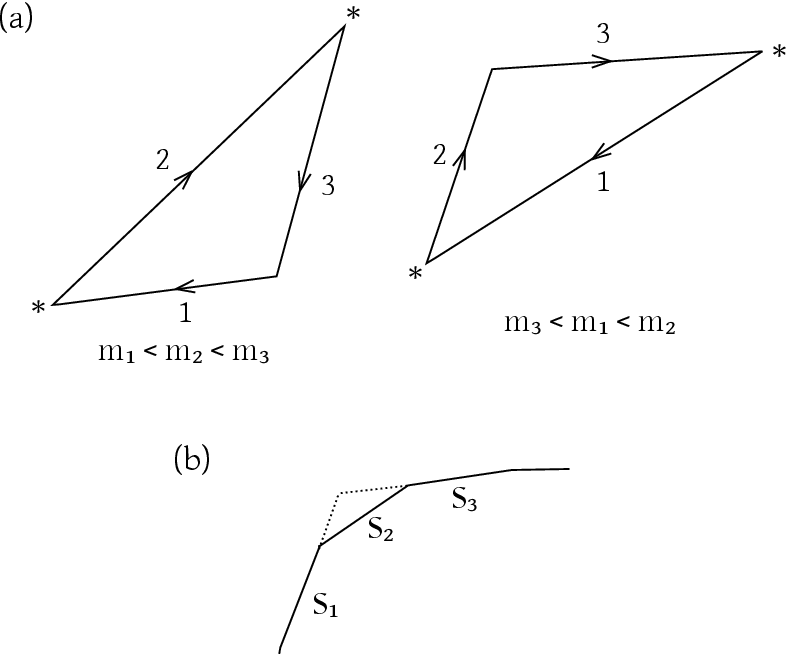} 
    \caption{(a) Negatively winding triangles with positive slopes
      have two slope increases ($m$ denotes slope). (b) Removing the side $S_2$.}
    \label{fig:spwp}
  \end{figure}

  Now suppose $P$ is a self-intersecting $N_0$-gon, that winds negatively around
  $y$. Using Lemma \ref{lem:simple-spwp}, we can find a simple subloop
  $P'$ winding negatively around $y$. Since $P'$ is a simple $N$-gon for some $N \leq N_0$, it must have at least two vertices where the slope increases. For each vertex $v$ of $P'$ where the slope
  of $P'$ increases, there is a vertex in the original polygon $P$
  where the slope increases, either at $v$ itself, or at some
  point in the interval of $P$ that was deleted at $v$. Thus, since
  $P'$ has at least two vertices where the slope increases, so does $P$.

 % Then $P= P' * P''$, where $P'$ is a simple polygon winding
 %  negatively around $y$, and $P''$ is some other polygon, not
 %  necessarily simple, and $*$ denotes composition of loops. Hence by
 %  the claim for simple polygons, the slope increases at least two
 %  vertices of $P'$. If these two vertices are also vertices of $P$, we
 %  are done. Otherwise, one is the ``new'' vertex where the composition
 %  of paths is done. From the point of view of $P''$, however, the
 %  slope actually \emph{decreases} at this vertex, because of the
 %  orientations. Because $P''$ is a closed polygonal path, the slope
 %  must increase at some other vertex, which is also a vertex of
 %  $P$. Hence the slope increases at no fewer than two vertices of $P$.
\end{proof}

Having chosen a lift $\partial[\tilde{u}]$ of the boundary loop,
define a $2$-chain $C$ on $\tilde{M}$ whose multiplicity at $y$ is the
winding number of $\partial[\tilde{u}]$ around $y$. This has $\partial C
= \partial[\tilde{u}]$.

\begin{lemma}
\label{lem:existence}
There is a complex structure $j$ on $T$ and a holomorphic map
$\tilde{u}: T \to \tilde{M}$ such that $\tilde{u}_*[T] = C$.
\end{lemma}

\begin{proof}
  Lemma \ref{lem:positive-winding} allows us to apply Lemma
  \ref{lem:spwp-mapping}, which yields map $\tilde{u} : T \to \tilde{M}$. More
  precisely, the complex structure on $T$ is the one obtained from
  uniformization of the region bounded by the simple lift of $\gamma =
  \partial[u]$ at the end of the construction in
  \ref{lem:spwp-mapping}, as a Riemann surface with boundary and
  punctures (at the non-smooth points of the loop).
\end{proof}

% \begin{proof}
%   If $C$ has multiplicity $1$ everywhere within its support, then
%   $\partial C$ is a simple closed contractible loop in a contractible
%   Riemann surface $\tilde{M}$. Hence $\Supp C$, with the complex
%   structure induced by $M$, is a simply-connected Riemann surface with
%   boundary and boundary punctures. Since the Lagrangians are in
%   general position, the number of boundary punctures is correct so
%   that uniformization of $\Supp C$ as a Riemann surface with boundary
%   and boundary punctures yields the complex structure $j$ on $T$ and
%   the map $\tilde{u}$.

%   If $C$ has multiplicity greater than $1$ somewhere, we can pass to a
%   branched cover of $\tilde{M}$ in order to reduce its
%   complexity. That is to say, let $\tilde{M}'$ be a $2$-fold branched
%   covering of $\tilde{M}$ over a point $y$ where $w(y,\partial C) \geq
%   2$. Choosing a lift $C'$ of $C$, we find that $C'$ has lower
%   multiplicities at the preimages of $y$. Repeating this process at
%   different points, we obtain a branched covering $\tilde{M}^{(k)}$
%   and a $2$-chain $C^{(k)}$ which has multiplicity $1$ within its
%   support. Then as above we obtain a complex structure $j$ on $T$ and
%   a map $\tilde{u}^{(k)}: T \to \tilde{M}^{(k)}$, which after
%   composing with the covering $\tilde{M}^{(k)} \to \tilde{M}$, gives
%   the desired result.
% \end{proof}

Pushing the map $\tilde{u}$ from Lemma \ref{lem:existence} down to
$M$, we obtain the existence of a holomorphic representative in the
homotopy class $\phi$, for a particular complex structure on the domain.

\subsubsection{The moduli space of holomorphic representatives with
  varying conformal structure}
\label{sec:moduli-space}

Let $\cM(\phi,j)$ denote the moduli space of $(j,J_M)$-holomorphic
maps $u: T \to M$ in the homotopy class $\phi$. Let $\cM(\phi) =
\bigcup_{j} \cM(\phi,j)$ denote 
the moduli space of such maps with varying conformal structure on the
domain. Let $\cR^{k+3}$ denote the moduli space of conformal
structures on the disk with $(k+3)$ boundary punctures. There is a
natural map $\pi: \cM(\phi) \to \cR^{k+3}$ which forgets the map.

\begin{lemma}
  \label{lem:no-kernel}
  For any $u \in \cM(\phi,j)$, we have $\ind D_{u,(j,J_M)} = 0$ and $\ker D_{u,(j,J_M)} = 0$.
\end{lemma}

\begin{proof}
  Let the intersection points $q_{a,i},x_1,\dots,x_k,q_{b,j}$ be
  regarded as positive punctures and let $q_{a+b,h}$ be regarded as a
  negative puncture. Then by the conventions for Maslov index, we have
  that all of these intersection points have index $0$, for as we go
  $L(0)\to L(n)$, $L(n) \to \tau L(n)$, \dots, $\tau^k L(n) \to
  L(n+m)$, and $L(0) \to L(n+m)$, the Lagrangian tangent space tilts
  clockwise by a small amount. Then by Proposition 11.13 of
  \cite{seidel-book} (with $|\Sigma^-| = 1$ for the negative
  puncture), we have $\ind D_{u,(j,J_M)} = 0$.

  Furthermore, the operator $D_{u,(j,J_M)}$ is a Cauchy--Riemann operator acting on
  the line bundle $u^*TM$, so the results of Section (11d) of
  \cite{seidel-book} apply. The hypotheses of Lemma 11.5 are satisfied
  with $\mu(\rho_1) = 0$ and $|\Sigma^-| = 1$, so $\ker D_{u,(j,J_M)}
  = 0$.
\end{proof}

\begin{lemma}
  \label{lem:covering}
  The map $\pi: \cM(\phi) \to \cR^{k+3}$ is a proper submersion of
  relative dimension zero (that is, a finite covering).
\end{lemma}

\begin{proof}
  The index $\ind D_{u,(j,J_M)} = 0$ is the expected dimension of the space of curves with fixed conformal structure. When we allow the conformal structure to vary, this adds $k = \dim \cR^{k+3}$ dimensions, so $\cM(\phi)$ has expected dimension $k$. Lemma \ref{lem:no-kernel} implies $\coker D_{u,(j,J_M)} = 0$, so $\cM(\phi)$ is a manifold of dimension $k$.

  % The relative dimension is the dimension of a generic fiber, which is
  % $\ind D_{u,(j,J_M)}$ for some fixed $j$. By the previous Lemma this
  % is zero. Thus the dimension of $\cM(\phi)$ is equal to that of
  % $\cR^{k+3}$, which is $k$.

  If $u \in \cM(\phi,j) \subset \cM(\phi)$ is a point where the map
  $\pi: \cM(\phi) \to \cR^{k+3}$ is not a submersion, we must have
  $\ker D\pi \neq 0$. On the other hand $\ker D\pi$ consists of
  infinitesimal deformations of the map which do not change the
  conformal structure on the domain, and so is equal to $\ker
  D_{u,(j,J_M)}$, which is zero by the previous lemma. Hence $\pi$ is
  a submersion.

  The properness of $\pi$ is an instance of Gromov--Floer
  compactness. The only thing to check is whether, as we vary $j \in
  \cR^{k+3}$, any strips can break off. This is impossible because our
  boundary conditions do not bound any bigons in $M$.
\end{proof}

\begin{lemma}
  \label{lem:degree-1}
  The map $\pi: \cM(\phi) \to \cR^{k+3}$ has degree one.
\end{lemma}

\begin{proof}
  For this Lemma we will pass to the Gromov--Floer compactification
  $\bar{\pi}: \bar{\cM}(\phi) \to \bar{\cR}^{k+3}$. Because no bigons can
  break off, this compactification consists entirely of stable disks,
  and $\bar{\pi}$ is also a proper submersion. Hence to count the
  degree of $\pi$, it will suffice to count the points in the fiber of
  $\bar{\pi}$ over a corner of $\bar{\cR}^{k+3}$, which is to say when
  the domain is a maximally degenerate stable disk.

  A maximally degenerate stable disk $S = (G, (S_\alpha))$ consists of a
  trivalent graph $G = (V,E^{fin} \cup E^\infty)$ without cycles, with
  $(k+3)$ infinite edges $E^\infty$, and a disk $S_\alpha$ with three
  boundary punctures for each $\alpha \in V$. The boundary punctures
  of $S^\alpha$ are labeled by elements of $E^{fin}\cup E^\infty$. The
  elements of $E^{fin}$ correspond to nodes of the stable disk, while
  the elements of $E^\infty$ correspond to boundary punctures of the
  smooth domains in $\cR^{k+3}$. The homotopy class $\phi$ determines
  the Lagrangian boundary conditions on each component $S_\alpha$, and
  the asymptotic values at the boundary punctures labeled by
  $E^\infty$. The position of the nodes labeled by $E^{fin}$ is not
  determined \emph{a priori}.

  Looking at the Lagrangians $L(0)$, $L(n)$, \dots, $L(n+m)$ shows
  that any three of them bound triangles, and that such triangles are
  determined by two of the corners. Hence by tree-induction starting
  at the leaves of stable disk (those $S_\alpha$ for which two
  punctures are labeled by $E^\infty$), the positions of all the nodes
  are determined by $\phi$, or we run into a contradiction because no
  triangles consistent with the labeling exist.

  Furthermore, in each homotopy class of triangles consistent with the
  labeling of $S_\alpha$, there is exactly one holomorphic
  representative. 

  Hence there is at most one stable map from the stable domain $S =
  (G,(S_\alpha))$ to $M$ consistent with the homotopy class $\phi$.

  Thus we have shown that the degree of $\pi: \cM(\phi) \to \cR^{k+3}$ is either
  zero or one. On the other hand, Lemma \ref{lem:existence} shows that
  $\cM(\phi)$ is not empty, so the degree must be one.
\end{proof}

Proposition \ref{prop:polygon-count} follows immediately from Lemmas \ref{lem:covering} and \ref{lem:degree-1}.

\subsection{Signs}
\label{sec:signs}

In order to determine the signs appearing in the counts of triangles,
we need to specify the brane structures on the Lagrangians
$L(d)$. Since $L(d)$ fibers over a curve in the base, and its
intersection with each fiber is a curve, the tangent bundle of $L(d)$ is
trivial, and we can define a framing of $TL(d)$ using the vertical and
horizontal tangent vectors at each point. Using this framing, we can
give $L(d)$ a trivial $\Spin(2)$ structure, which is induced by
product of the trivial $\Spin(1)$ structures on the horizontal an
vertical tangent bundles.

Although we have not said much about it up until now, strictly
speaking the generators of $CF^*(L(d_1),L(d_2))$ are not canonically
identified with intersection points $q \in L(d_1)\cap L(d_2)$. Rather,
each intersection point $q$ gives rise to an abstract $1$-dimensional
$\R$--vector space, the orientation line $o(q)$. Following
\cite[Ch. 11]{seidel-book}, this line is canonically isomorphic up to
multiplication by a positive number with the determinant line of a
certain Fredholm operator $D_q$. The moduli space of
pseudo-holomorphic curves is canonically oriented relative to these lines.

In the case $d_1 \leq d_2$, where all the intersections have degree
$0$, there is a preferred choice of trivialization for $o(q)$. Let $q
\in L(d_1)\cap L(d_2)$. Then we have horizontal-vertical splittings
\begin{equation}
  T_qL(d_i) = (T_qL(d_i))^h \oplus (T_qL(d_i))^v
\end{equation}
The intersection point $q$ has degree $0$ as a morphism from $L(d_1)$
to $L(d_2)$, and moreover both $(T_qL(d_i))^h$ and $(T_qL(d_i))^v$
tilt clockwise by a small amount as we pass from $L(d_1)$ to
$L(d_2)$. 

Let $H$ denote the half-plane with a negative puncture. We define the
orientation operator $D_q = D_q^h \oplus D_q^v$, acting on the product
bundle $\C\times \C \to H$, where the boundary condition in the first
factor is the short path $(T_qL(d_1))^h\to (T_qL(d_2))^h$, while that
in the second factor is the short path $(T_qL(d_1))^v \to
(T_qL(d_2))^v$. By \cite[Eq. (11.39)]{seidel-book}, there is a
canonical isomorphism
\begin{equation}
  \det(D_q) \cong o(q)
\end{equation}
which for our purposes we take as the definition of $o(q)$.  On the
other hand, $D_q$ is the direct sum of the operators $D_q^h$ and
$D_q^v$, which have vanishing kernel and cokernel. Hence
\begin{equation}
  \det(D_q) \cong \det(D_q^h)\otimes \det(D_q^v) \cong \R \otimes \R
  \cong \R
\end{equation}
where all isomorphisms are canonical. This gives us a preferred choice
of isomorphism $o(q) \cong \R$.

\begin{proposition}
  \label{prop:signs}
  Taking the preferred isomorphisms $o(q) \cong \det(D_q)\cong \R$ for all generators
  $q \in CF^0(L(d_1),L(d_2))$ for $d_1 \leq d_2$, all of the
  holomorphic triangles found above have the same sign. Hence, after possibly reversing all off the preferred isomorphisms $o(q) \cong \R$, all of them have a positive sign.
\end{proposition}
\begin{proof}
  Let $u: S \to X(B)$ be a triangle with positive punctures at $q_1
  \in CF^0(L(d_1),L(d_2))$ and $q_2 \in CF^0(L(d_2),L(d_3))$ and
  negative puncture at $q_0 \in CF^0(L(d_1),L(d_3))$. Let $D_u$ denote
  the linearized operator at $u$. Gluing onto $D_u$ the chosen
  orientation operators $D_{q_2}$ and $D_{q_1}$ in that order gives
  another orientation operator $D_{q_0}'$ for the point $q_0$, and we want to compare the two orientations for the line $o(q_0)$. 
  
  The linearized operator $D_u$ is an operator in the pulled back tangent bundle $u^*TX(B)$. The bundle $TX(B)$ has a splitting into horizontal and vertical subspaces $TX(B) = TX(B)^h \oplus TX(B)^v$ given by the symplectic connection. Over the boundaries of $S$, the Lagrangian boundary condition also splits. Recall from the discussion on gradings in Section \ref{sec:gradings} that we have phase maps on horizontal and vertical subspaces, which measure rotation with respect to the foliations by circles on the base and fiber respectively. These serve to give us preferred trivializations of $u^*TX(B)^h$ and $u^*TX(B)^v$. With respect to these trivializations, the Lagrangian boundary conditions move only by a small amount: in the horizontal space, $TL^h$ has essentially constant phase, while in the vertical space, $TL^v$ moves in such a way that it never crosses the the line corresponding to the circle foliation on the fibers. 

The next step is to observe that the boundary conditions for the operator $D_u$ are essentially the same independent of the map $u$ and of the chosen degree zero generators $q_0, q_1, q_2$, and that hence all of the maps $u$ must contribute with the same sign. In all cases we have an operator in a trivial rank two vector bundle, with split Lagrangian boundary conditions, each factor of which is trivial. Additionally, each pair of Lagrangians at a puncture has the same local form, namely they are related by a small clockwise rotation in both the horizontal and vertical directions. The orientation operators $D_{q_i}$ used to define the orientation lines are also essentially the same for every point $q_i$. Hence the relative sign between the determinant of $D_{q_0}'$ and $o(q_0)$ must be the same for each map $u$.

If the signs are all positive, we are done. Otherwise by reversing our choice of generator for every line $o(q)$ we obtain bases for the groups $CF^0(L(d_1),L(d_2))$ in which all curves contribute positively.

% Since
%   all Spin structures involved are trivial, they introduce no
%   complication in this gluing. Since $D_u$, $D_{q_2}$, and $D_{q_1}$
%   are index zero operators with vanishing kernel and cokernel,
%   $D_{q_0}'$ is as well, and has $\det(D_{q_0}') \cong \R$
%   canonically. Hence the isomorphism $\det(D_{q_0}') \cong o(q_0)$
%   induces the same orientation as $\det(D_{q_0}) \cong o(q_0)$.
\end{proof}

\begin{remark}
  The preceding proposition is another example of the phenomenon observed by Abouzaid \cite[Lemma 3.21]{abouzaid06}. A similar argument is found in \cite[(13c)]{seidel-book}.
\end{remark}

\section{A tropical count of triangles}
\label{sec:trop-fuk}
Abouzaid, Gross and Siebert have proposed a definition of a category
defined from the tropical geometry of an integral affine manifold,
which is meant to describe some part of the Fukaya category of the
corresponding symplectic manifold. The starting point for this
definition is an integral affine manifold $B$. The objects are then
the non-negative integers, with $\hom^0(d_1,d_2) = \Span
B(\frac{1}{d_2-d_1}\Z)$ when $d_1 < d_2$, and chains on $B$ when $d_1
= d_2$. The composition is defined by counting a certain type of
tropical curve that is balanced after addition of \emph{tropical
disks}.

The motivation is that the non-negative integers correspond to certain
Lagrangian sections $L(n)$ of a special Lagrangian torus fibration
over $B$, such that the intersection points between $L(d_1)$ and $L(d_2)$ lie precisely
over the points in $B(\frac{1}{d_2-d_1}\Z)$. The tropical curves then
correspond to the pseudo-holomorphic polygons counted in the A$_\infty$ operations.

The Lagrangians considered above are essentially an example of this
symplectic setup, so it is encouraging that our computation agrees
with the expectation of Abouzaid-Gross-Siebert. The tropical triangles
counted in their definition correspond closely to the
pseudo-holomorphic triangles found in Section \ref{sec:degeneration}.

\subsection{Tropical polygons}
\label{sec:tropical-polygons}

Let $\nabla$ denote the canonical torsion-free flat connection on $B$
associated to the affine structure. The following definitions are due to
Abouzaid \cite{abouzaid-msri-dec09}.

\begin{definition}
  \label{def:tropical-disk}
  Let $B$ be a two-dimensional singular affine manifold with focus-focus singularities, and let $x \in B$ be a non-singular point. Let $\Gamma$ be a tree all of whose edges are finite. Label one univalent vertex $p_{\text{out}}$, and denote the other univalent vertices by $p_1,\dots,p_r$. A \emph{tropical disk} in $B$ ending at $x$ is a balanced tropical embedding $v : \Gamma \to B$, satisfying the following conditions.
  \begin{enumerate}
  \item $v(p_{\text{out}}) = x$,
  \item For $i = 1,\dots,r$, $v(p_i)$ is a singular point of the affine structure, and $v$ maps the (unique) edge incident to $p_i$ into a monodromy-invariant line of the singularity.
  \end{enumerate}
\end{definition}

In the case at hand, there is only one singularity, so the only such tropical disks are line segments emanating from the singular point in the monodromy-invariant direction. If there were more singularities with intersecting monodromy-invariant lines, we could obtain more complicated tropical disks.

\begin{definition}
  \label{def:tropical-polygon}
  Let $q_0, q_1, \dots, q_k$ be points of $B$, with $q_j \in
  B(\frac{1}{d_j}\Z)$. Let $\Gamma$ be a metric ribbon tree with $k+1$
  infinite edges. One infinite edge, the \emph{root}, is labeled with
  $q_0$ and it is the output. The other $k$ infinite edges, the
  \emph{leaves}, are labeled with $q_1,\dots,q_k$ in counterclockwise
  order and these are the inputs. Assign to the region between $q_j$
  and $q_{j+1}$ the weight $\sum_{i=1}^jd_i$, and give the region
  between $q_0$ and $q_1$ weight $0$. Orient the tree upward from the
  root, so that each edge has a ``left'' and a ``right'' side coming
  from the ribbon structure. To each edge $e$, assign a weight $w'_e$
  given by the weight on the left side of $e$ minus the weight on the
  right side of $e$. Define a corrected weight $w_e$ by
  \begin{equation}
    \begin{split}
      w_e = 0 & \text{ if $w'_e < 0$ and $e$ contains a leaf}\\
      w_e = 0 & \text{ if $w'_e > 0$ and $e$ contains the root}\\
      w_e = w'_e &\text{ otherwise}
    \end{split}
  \end{equation}
  Then a \emph{tropical polygon} modeled on $\Gamma$ is a map $u:
  \Gamma \to B$ such that:
  \begin{enumerate}
  \item $u$ converges on the root edge to $q_0$ and on the $j$-th leaf edge to $q_j$;
  \item on the edge $e$, the tangent vector $\dot{u}_e$ to the
    component $u_e$ satisfies
    \begin{equation}
      \nabla_{\dot{u}_e} \dot{u}_e = w_e \dot{u}_e
    \end{equation}
    and $\dot{u}_e$ converges to $0$ on the infinite ends of the root and leaf edges; this differential
    equation only holds outside a finite set of points where tropical
    disks are attached;
  \item there exists a finite collection of tropical disks
    $v_1,\dots,v_N$ such that the union $u \cup \{v_1,\dots,v_N\}$ is
    balanced; the balancing condition at a vertex $x$ is the vanishing
    of the sum of the derivative vectors $\dot{u}_e$ of the various
    components of $u$ incident at $x$ and the integral tangent vectors
    to $v_i$ at $x$, oriented toward the vertex.
  \end{enumerate}
\end{definition}

The heuristic motivation for these definitions is that, if we consider a deformation of complex structures in which the torus fibers collapse, we expect that a family of honest holomorphic polygons will likewise collapse onto a piecewise linear complex in $B$. We further expect that two types of local limiting behavior can occur. Some parts of the curve may limit to surfaces that fiber over straight lines in $B$ intersecting the torus fibers in paths joining two Lagrangian sections, while other parts may limit to surfaces that fiber over straight lines in $B$ intersecting the torus fibers in circles. The latter are what become the tropical disks, while the former become the rest of the tropical polygon. The balancing condition comes from the necessity that these parts can be connected up to form a topological polygon with boundary on the Lagrangian sections. See also \cite[Ch. 8]{d-branes-book}.

To unpack this definition, let us restate it in the simplest case,
which is that of tropical triangles (this mainly simplifies the issues
regarding weights): 

\begin{proposition}
  \label{prop:tropical-triangles}
  Let $q_1 \in B(\frac{1}{n}\Z)$ and $q_2 \in B(\frac{1}{m}\Z)$, with
  $n > 0$ and $m > 0$. Let $q_0 \in B(\frac{1}{n+m}\Z)$. Let $\Gamma$
  be the ribbon tree with one vertex and three infinite edges. Then a
  tropical triangle modeled on $\Gamma$ consists of three
  maps $u_0: [0,\infty) \to B$, $u_1: (-\infty,0] \to B$, $u_2:
  (-\infty,0] \to B$ such that
  \begin{enumerate}
  \item $u_0 \equiv q_0$ is a constant map;
  \item $u_1(0) = u_2(0) = q_0$;
  \item $u_1(-\infty) = q_1$, $u_2(-\infty) = q_2$;
  \item We have $\nabla_{\dot{u}_1} \dot{u}_1 =
    n\dot{u}_1$ and $\nabla_{\dot{u}_2} \dot{u}_2 = m\dot{u}_2$, outside
    of a finite set of points where tropical disks are attached;
  \item there exists a finite collection of tropical disks
    $v_1,\dots,v_N$ such that the balancing condition holds.
  \end{enumerate}
\end{proposition}

We can also unpack the equation $\nabla_{\dot{u}} \dot{u} =
n\dot{u}$. The key outcome of the following Lemmas is the insight that
the tangent vector $\dot{u}$ increases by $n$ times the distance the
path $u$ travels. Recall that a path $\gamma$ is called a geodesic for a connection $\nabla$ if $\nabla_{\dot{\gamma}} \dot{\gamma} = 0$.

\begin{lemma}
  Let $\gamma: [0,1] \to B$ be a geodesic for $\nabla$. Define $u :
  (-\infty,0] \to B$ by $u(t) = \gamma(\exp(nt))$. Then at the point
  $x = \gamma(s)$, we have $\dot{u} = ns\dot{\gamma}$, and
  $\nabla_{\dot{u}} \dot{u} = n\dot{u}$.
\end{lemma}
\begin{proof}
  We set $s = \exp(nt)$. The equation $\dot{u} = ns\dot{\gamma}$ is just
  the chain rule (note that $\dot{u} = du/dt$ while $\dot{\gamma} =
  d\gamma/ds$). We have
  \begin{equation}
    \nabla_{\dot{\gamma}} \dot{u} = \nabla_{\dot{\gamma}} (ns\dot{\gamma}) = n\dot{\gamma}
    + ns\nabla_{\dot{\gamma}}\dot{\gamma} = n\dot{\gamma}
  \end{equation}
  since $\gamma$ is a geodesic. Multiplying this equation by $ns$, and
  using the fact that $\nabla$ is tensorial in its subscript, we
  obtain $\nabla_{\dot{u}} \dot{u} = n\dot{u}$.
\end{proof}

\begin{lemma}
\label{lem:differences}
  Let $u: [a,b] \to B$ solve $\nabla_{\dot{u}}\dot{u} =
  n\dot{u}$. Then if we patch together affine charts along $u$ to embed a
  neighborhood of $u$ into $\R^n$, we have
  \begin{equation}
    \dot{u}(b) - \dot{u}(a) = n(u(b)-u(a))
  \end{equation}
  where we use the affine structure of $\R^n$ to take differences of
  points and vectors at different points.
\end{lemma}
\begin{proof}
  Define $\gamma: [\exp(na),\exp(nb)] \to B$ by $\gamma(s) = u((\log
  s)/n)$, so that $\gamma$ is a geodesic. We can embed a neighborhood
  of $\gamma$ into $\R^n$ by patching together affine coordinate
  charts along $\gamma$. Using addition in this embedding, we can
  write $\gamma(s) = \gamma(\exp(na)) + (s-\exp(na))\dot{\gamma}$. We
  have $\dot{u}(\log(s)/n) = ns\dot{\gamma}(s)$, and
  \begin{equation}
    \dot{u}(b) - \dot{u}(a) =n(\exp(nb)-\exp(na))\dot{\gamma} =
    n(\gamma(\exp(nb)) - \gamma(\exp(na))) = n(u(b)-u(a))
  \end{equation}
\end{proof}

One more thing to note regarding these tangent vectors $\dot{u}$ is
how they represent homology classes on the torus fibers of the
fibration over $B$.

When considering a tropical curve $C^\trop$ corresponding to a closed
holomorphic curve $C$, each edge of the tropical curve carries an
integral tangent vector, which morally represents the class $[C\cap
T^2_b] \in H_1(T^2_b;\Z)$ that measures how the
holomorphic curve intersects the torus fiber $T^2_b = \pi^{-1}(b)$. Because the curve is closed, this class is locally constant along
each edge of $C^\trop$.

When considering a tropical curve $C^\trop$ representing a holomorphic
curve $C$ with boundary on Lagrangian sections $L(i),L(j)$, the
intersection of $C$ with $T^2_b$ would morally be a path on $T^2_b$
from $L(i)_b$ to $L(j)_b$, where $L(i)_b$ is intersection of $T^2_b$
and $L(i)$. Let $A(L(i)_b,L(j)_b) \subset
H_1(T^2_b,\{L(i)_b,L(j)_b\};\Z)$ be the subset consisting of such
cycles, which could also be described as the preimage of $L(j)_b -
L(i)_b \in H_0(\{L(i)_b,L(j)_b\};\Z)$ under the boundary
homomorphism. Hence $A(L(i)_b,L(j)_b)$ is a torsor for the kernel of
that homomorphism, which is $H_1(T^2_b;\Z)$. Using the group structure
on $T^2_b$, we can identify $A(L(i)_b,L(j)_b)$ with the coset
\begin{equation}
  [L(j)_b-L(i)_b] + H_1(T^2_b;\Z) \subset H_1(T^2_b;\R)
\end{equation}
On the other hand, there is an isomorphism $(T_bB)_\R \cong
H_1(T^2_b;\R)$. Hence the class of $[C\cap T^2_b]$ can be regarded as
a tangent vector to the base, which is in general real and varies
along the tropical curve as $L(i)$ and $L(j)$ move relative to one
another. The tangent vector $\dot{u}$ to the tropical curve is
this class.

The balancing condition at a vertex $b$ of the tropical polygon
amounts to requiring that the three paths $L(i)_b\to L(j)_b$, $L(j)_b
\to L(k)_b$ and $L(k)_b\to L(i)_b$ form a contractible loop. At a
point where a tropical disk is attached, the path $L(i)_b \to L(j)_b$
changes discontinuously by a loop in the homology class in
$H_1(T^2_b;\Z)$ corresponding to the integer tangent vector to the
tropical disk. 

The last thing to describe for tropical polygons is their
multiplicities. There is a multiplicity coming from the different ways to attach
a disk $v$ to the tropical polygon. If one incoming edge of the
polygon has tangent vector $\dot{u}_e$ corresponding to a path
$\gamma_1: L(i)_b \to L(j)_b$ on $T^2_b$, and the disk has tangent
vector $w$ corresponding to a loop $\gamma_2$ on $T^2_b$, there are
$|\gamma_1.\gamma_2| = |\det(\dot{u}_e, w)|$ points where the disk can
be attached, assuming this determinant is an integer (as it is in the
special case below). Otherwise, one must look carefully at exactly
where on the torus the paths $\gamma_1$ and $\gamma_2$ are located.

For a tropical disk $v$, we also expect the Gross--Siebert theory
to associate a multiplicity $m(v)$. Morally speaking, this number should be a virtual count of
holomorphic disks corresponding to $v$, and it should possible to extract this from the ``structure'' on the affine base in the sense of \cite{gs-affine-complex}. In the case at hand, where the base has just a single focus-focus singularity, all of these factors are expect to be one for simple disks, and zero for multiply covered disks \cite{abouzaid-comm}. In the rest of this section, we adopt this as an ansatz. Modulo this ansatz, we show in Proposition \ref{prop:tropical-our-case} that the tropical curve counts agree with the holomorphic curve counts computed in Section \ref{sec:degeneration}. From a certain point of view, this may be regarded as evidence for the ansatz itself, since tropical curve invariants are intentionally designed to correspond to holomorphic curve invariants, and we have computed the latter.

In summary, the multiplicity of a tropical polygon will have both the
Gross--Siebert factors $m(v)$ counting how many holomorphic disks are
in each class, as well as simpler factors counting how many ways these
classes of disks can be attached. 

\begin{remark}
  After this paper was originally written, an exposition of the closely related idea of \emph{jagged paths} appeared \cite[Definition 3.2]{gs-theta}. In that paper, the various multiplicities are packaged in a different way, but it turns out the that tropical curve counts defined there also agree, in the case of $(\CP^2,D)$, with the holomorphic curve counts computed in Section \ref{sec:degeneration}. The multiplicities we seek are packaged into certain series $f_{\mathfrak{d}} = 1 + \sum_p c_pz^p$ where $\mathfrak{d}$ runs over certain distinguished rays and lines in $B$. See \cite[\S 3]{gs-theta} for more explanation. In our case, there are two such rays emanating from the singularity of the affine structure in the monodromy invariant directions. Our ansatz is equivalent to the fact that the corresponding series reduce, in appropriate choices of coordinates, to $1 + z$ and $1+w$ respectively. The salient feature is that these series have a linear term with coefficient $1$ and no higher degree terms.
\end{remark}

\subsection{Tropical triangles for $(\CP^2, D)$}
\label{sec:tropical-our-case}

We now write out explicitly the tropical curves contributing to the
triangle products in the case of $(\CP^2, D)$. Let us use coordinates
$(\eta,\xi)$ where the point $q_{a,i} \in B(\frac{1}{n}\Z)$ has
coordinates $(a/n,-i/n)$

There is one family of simple tropical disks that emanate from the
singularity on the $\eta = 0$ line in the vertical direction. Their
primitive tangent vectors are $\pm (0,1)$.

Figure \ref{fig:tropical-triangle} shows the the tropical triangle
representing the contribution of $y^2p$ to the product of $x^2$ and
$z^2$. This triangle has multiplicity $2$. The singularity of the
affine structure is placed so as to emphasize the tropical disk ending
at the singularity.
\begin{figure}
\includegraphics[width=3in]{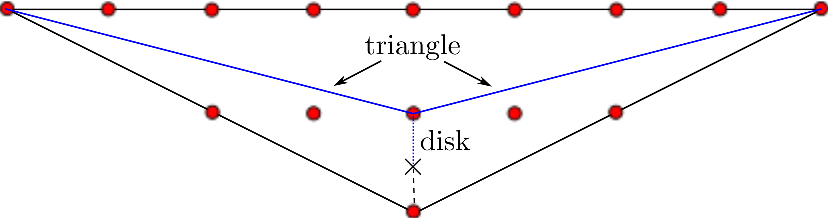} 
\caption{A tropical triangle. The dotted line extending up from the singular point is a tropical disk.}
\label{fig:tropical-triangle}
\end{figure}

\begin{proposition}
  \label{prop:tropical-our-case}
  Let $n > 0$ and $m > 0$, and take $q_{a,i}\in B(\frac{1}{n}\Z)$,
  $q_{b,j} \in B(\frac{1}{m}\Z)$, and $q_{a+b,h} \in
  B(\frac{1}{n+m}\Z)$.

  Suppose $a$ and $b$ have different signs, and let $k =
  \min(|a|,|b|)$, and $h = i+j+s$. If $0 \leq s \leq k$, there is one tropical triangle
  connecting these three points. It is balanced after the addition of
  either $s$ or $k-s$ tropical disks, depending on the position of the
  singularity. The multiplicity of this curve is $\binom{k}{s}$. If
  $s$ does not lie in this range, there is no triangle.

  Suppose that $a$ and $b$ have the same sign. Then unless $h = i+j$
  there is no triangle, and when $h = i+j$ there is exactly one, which
  is represented geometrically by the line segment joining $q_{a,i}$
  and $q_{b,j}$.
  
\end{proposition}

\begin{proof}
When $a$ and $b$ have the same sign the tropical triangle can have no
tropical disks attached to it. Therefore the tropical triangle is
simply a line segment which passes through the three points, and this
only exists if $q_{a+b,h}$ lies on the line between $q_{a,i}$ and $q_{b,j}$.

As for when $a$ and $b$ have different signs, let us consider the case
$a \leq 0$, $b \geq 0$, $a+b \geq 0$, and hence $k = -a$. The other
cases are related to this by obvious reflections and re-namings.

By Proposition \ref{prop:tropical-triangles}, a tropical triangle
consists essentially of two maps $u_1, u_2: (-\infty,0]\to B$,
together with some copies of the tropical disk and their multiples.

\begin{itemize}
\item The leg $u_2$ of the tree connecting $q_{b,j}$ to $q_{a+b,h}$
  cannot have any tropical disks attached, since both endpoints lie on
  the same side of the line $\eta = 0$. We can apply Lemma
  \ref{lem:differences} to obtain the tangent vector $\dot{u}_2$ at
  $q_{a+b,h}$ as $m$ times the difference between the
  endpoints, or
  \begin{equation}
    \dot{u}_2(q_{a+b,h}) = m(q_{a+b,h}-q_{b,j}) = \left(\frac{ma-nb}{n+m}, \frac{-(mi-nj+ms)}{n+m}\right)
  \end{equation}

\item The leg $u_1$ of the tree connecting $q_{a,i}$ to $q_{a+b,h}$
  crosses the line $\eta = 0$ at some point $x$, where tropical disks
  can be attached. It may bend there, and continue on to
  $q_{a+b,h}$, where the balancing condition $\dot{u}_1 + \dot{u}_2 =
  0$ must hold. This shows that the portion of $u_1$ connecting $x$ to
  $q_{a+b,h}$ must be parallel to $u_2$. Hence $x$ must be on the line
  joining $q_{b,j}$ and $q_{a+b,h}$. We obtain the position
  of $x$:
\begin{equation}
  x:   (\eta,\xi) = (0,(-aj+bi+bs)/(ma-nb))
\end{equation}

\item By Lemma \ref{lem:differences} at $x$ the tangent vector
  $\dot{u}_1$ is given by $n$ times
  the difference of the endpoints $x$ and $q_{a,i}$:
\begin{equation}
  \dot{u}_1(x)_L = n(x - q_{a,i}) = (-a,[a(mi-nj)+nbs]/(ma-nb))
\end{equation}
we use the subscript $L$ to denote this is the tangent vector coming
from the left.

\item If the singularity of the affine structure occurs below the
  point $x$, then tropical disks are line segments going up from the singularity to $x$ in the direction
  $(0,1)$.  Adding the vector $(0,s)$ to $\dot{u}_1(x)_L$, we obtain
  $\dot{u}_1(x)_R$, the tangent vector from the right,
  \begin{equation}
    \dot{u}_1(x)_R = \dot{u}_1(x)_L + (0,s) = (-a, a(mi-nj+ms)/(ma-nb))
  \end{equation}
  which is parallel to $\dot{u}_2(q_{a+b,h})$, as it must be. This
  shows that we must attach a collection of tropical disks whose total
  weight is $s$.
  
\item If the singularity of the affine structure occurs above $x$,
  then tropical disks are line segments going down from the singularity to $x$ in the direction $(0,-1)$, but there
  is also the monodromy to be taken into account. First we must act on
  $\dot{u}_1(x)_L$ by the monodromy
  $M=\begin{pmatrix}1&0\\1&1\end{pmatrix}$ to get
  \begin{equation}
    M\dot{u}_1(x)_L = (-a,[a(mi-nj)+nbs]/(ma-nb) - a) 
  \end{equation}
  Adding the vector $(0,-(k-s))$ to this, with $k = -a$, gives the
  same result as before for $\dot{u}_1(x)_R$. However, in this case we
  are attaching a collection of disks with total weight $(k-s)$. 

\item The leg $u_1$ propagates in the direction $\dot{u}_1(x)_R$ from
  $x$ to $q_{a+b,h}$. As it does so, the tangent vector $\dot{u}_1$
  increases by $\Delta = n(q_{a+b,h} - x)$, which is parallel to
  $\dot{u}_1(x)_R$. By comparing affine lengths of the segment of $u_1$ from $q_{a,i}$ to $x$ and the segment of $u_1$ from $x$ to $q_{a+b,h}$, we have the proportion
  \begin{equation}
    \dot{u}_1(x)_R:\Delta = [-a/n] : [(a+b)/(n+m)]
  \end{equation}
  and 
  \begin{equation}
    (\dot{u}_1(x)_R + \Delta):\dot{u}_1(x)_R = [(a+b)/(n+m) - a/n]:[-a/n] = [-(ma-nb)/(n+m)]:[-a]
  \end{equation}
  Thus 
  \begin{equation}
    \dot{u}_1(q_{a+b,h}) =
    \left(-\frac{ma-nb}{n+m},\frac{mi-nj+ms}{n+m}\right) = -\dot{u}_2(q_{a+b,h})
  \end{equation}
  Verifying the balancing condition at $q_{a+b,h}$.
\end{itemize}

Now that we know which tropical curves contribute, we must compute
their multiplicities. The tropical curve constructed above uses the
tropical disk $s$ or $(k-s)$ times. This means that either we attach a single
simple disk $s$ times, or we attach some multiple covers of the disk
in some fashion as to achieve a total multiplicity of $s$. By our ansatz, only the simple disks contribute, and each with multiplicity one.

Attaching simple disks does introduce a multiplicity, since there are
multiple places to attach this disk. In fact, we have
$\det(\dot{u}_1(x)_L,(0,1)) = -a = k$, so there are a total of $k$
places for disks to be attached. Thus we get the multiplicity
$\binom{k}{s}$ or $\binom{k}{k-s}$, which are equal and give the
desired result.

\end{proof} 

\begin{remark}
The fact that the multiplicities $\binom{k}{s}$ and
$\binom{k}{k-s}$ are equal is an illustration of the general
phenomenon that the exact position of the singularity along its
invariant line does not matter for tropical curve counts.  
\end{remark}

\section{Parallel monodromy--invariant directions}
\label{sec:other-manifolds}

In this section we describe a class of affine manifolds to which the
results obtained for $(\CP^2,D)$ naturally generalize. Consider an
integral affine manifold $B$, compact with boundary. We require that
$B$ has the following properties. Recall that a focus-focus
singularity has a monodromy invariant direction, corresponding to the
tangent vector field that is fixed by the monodromy.
\begin{enumerate}
\item The boundary faces of $B$ are straight with respect to the
  affine structure.
\item The singularities of $B$ are focus-focus singularities, and the monodromy--invariant directions of all singularities are parallel.
\item There is an affine linear function $\eta: B \to \R$ such that any corners of $B$ occur at extreme values of $\eta$.
\end{enumerate}

\begin{remark}
  The condition that the monodromy--invariant directions are parallel
  is equivalent to the existence of a global integral vector field, or
  a vector invariant under parallel transport along an arbitrary path
  in $B$ minus the singular points. This vector field spans $\ker
  d\eta$.
\end{remark}

\subsection{Symplectic forms}
\label{sec:other-symplectic}

Let $x_1,\dots,x_n$ denote the singularities of $B$. Let
$[a_0,a_{n+1}]$ be the image of $B$ under $\eta$, and let $a_i =
\eta(x_i)$. We split up $B$ along the monodromy invariant lines of
each singularity, and obtain intervals $I_i =
[a_{i}+\epsilon,a_{i+1}-\epsilon]$, with fibrations $X(B_i) \to
X(I_i)$. Choosing an affine function $\xi$ on $B_i$ such that
$(\eta,\xi)$ are coordinates, $X(B_i)$ has corresponding complex
coordinates $(w,z_i)$, while the coordinate on $X(I_i)$ is $w$.

Each piece $B_i$ is an affine manifold whose horizontal boundary
consists of two straight lines (since $B$ has no corners but at the
extreme values of $\eta$), and hence we are in the situation of
section \ref{sec:monodromy-hessian}. We obtain a symplectic form for which the symplectic connection of $X(B_i)
\to X(I_i)$ foliates the horizontal boundary facets of $X(B_i)$.

Corresponding to each focus--focus singularity, we glue in a Lefschetz
singularity. The discussion in section
\ref{sec:focus-focus-lefschetz} applies directly. The result is a
manifold $X(B)$ with a Lefschetz fibration $w: X(B)\to X(I)$, and such
that the horizontal boundary faces of $X(B)$ are foliated by the
symplectic connection. Let $w_1,\dots,w_n$ denote the critical values
of $w: X(B)\to X(I)$. 

If $B$ does not have vertical boundary facets, but rather corners, we simply cut off the corners to introduce horizontal boundary facets. This modification only applies to the symplectic geometry constructions, and the relevant tropical geometry is still that of the original $B$.

We recall that $X(I)$ is an annulus $\{a_0 \leq \log|w| \leq a_{n+1}\}$. Let $M_0$ denote the fiber over $e^{a_0}$, and $M_1$ the fiber over $e^{a_{n+1}}$.

\subsection{Lagrangian submanifolds}
\label{sec:other-lagrangian}

As before, the construction of Lagrangian sections proceeds by taking
paths in the base and a Lagrangian in the fiber, and sweeping out a
Lagrangian in the total space by symplectic parallel transport. Potentially, we
have more freedom than in the mirror to $\CP^2$.

The admissible Lagrangian submanifolds we consider have boundary
conditions given by a complex curve in $\partial X(B)$ along each
boundary face. This gives two complex curves $\Sigma_0$ and $\Sigma_1$
for the bottom and top horizontal boundary faces. At the vertical
boundary faces, we have corresponding complex curves $M_0$ and $M_1$. These are fibers of the fibration $X(B)\to X(I)$, and to be
admissible requires the Lagrangian to end on this fiber. If $B$ has a
corner rather than a vertical boundary, then the fiber $M_i$ corresponds
to a vertical boundary created by cutting off the corner, and the
admissibility condition requires the Lagrangians to intersect this
fiber in a prescribed real curve.

Thus, a Lagrangian submanifold may be constructed by taking a
Lagrangian $L_0$ in $M_0$, and a path $\ell$ in the base joining $M_0$
to $M_1$, and taking the parallel transport. If $M_0$ corresponds to a
corner, we have only one choice for $L_0$, and if $M_1$ corresponds to
a corner, this imposes a constraint on $L_0$ and $\ell$. In order to
obtain sections of the torus fibration, we choose $L_0$ to be a curve
which is a section of the fibration by circles of
constant $\xi$ on each fiber, and $\ell$ to be a section of the fibration of the base
by circles of constant $\eta$. Thus when drawing $X(I)$ as an annulus,
$\ell$ appears as a spiral. 

As an example, consider the degree six del Pezzo surface $X_6$, which may be obtained from $\PP^1\times \PP^1$ by blowing up two points that do not lie on the same ruling line. The moment polytope of $\PP^1\times \PP^1$ is a rectangle. By doing toric blowups at opposite corners, we obtain $X_6$ with a toric boundary divisor consisting of six $(-1)$-curves, and the moment polytope is a hexagon. Alternatively, by doing almost toric blowups \cite{symington} on opposite sides of the hexagon, we obtain $X_6$ with an anticanonical divisor of four components with self-intersections $-1,0,-1,0$. The base affine manifold $B$ has four sides and two singularities with parallel monodromy invariant directions.

Passing over to the mirror side, we consider Lagrangians in the symplectic manifold $X(B)$. Then we can choose independently
\begin{enumerate}
\item the number of times the initial Lagrangian $L_0$ winds around
  the fiber $M_0$,
\item the number of times $\ell$ winds around the base between the
  $M_0$ and the first singularity,
\item the number of times $\ell$ winds around the base between the two
  singularities, and
\item the number of times $\ell$ winds around the base between the
  second singularity and $M_1$.
\end{enumerate}
This gives rise to a $4$-parameter family of Lagrangians. Under mirror
symmetry, all of them correspond to line bundles. We expect that
$X(B)$ is a mirror to the degree $6$ del Pezzo surface $X_6$ with a
$4$-component anticanonical divisor. The $4$ parameters correspond to
$\Pic(X_6) \cong \Z^4$. 

Though we can construct many Lagrangians this way, in order to compute
Floer cohomology and identify the basis with $B(\frac{1}{d}\Z)$, we
must choose a family of Lagrangians $\{L(d)\}_{d\in\Z}$ corresponding
to the tensor powers of an ample line bundle. First we choose $\ell(0)$ as a
reference path in the base, over which lies $L(0)$, a Lagrangian
satisfying the boundary conditions. We take $\ell(1)$ to be a certain
path in the base: the number of times that $\ell(1)$ must wind between
the singularities and the vertical boundaries/corners is determined by
$B$: it is essential that the number of turns $\ell(1)$ makes between
two consecutive singularities (or between a boundary and the
neighboring singularity) is the affine width of the corresponding
portion of $B$. An equivalent condition is that the $\ell(1)$ winds at
unit speed. If $B$ has a vertical boundary face rather than a corner,
the intersection of $L(1)$ with the fiber at that boundary must be a
curve that, relative to $L(0)$, makes a number of turns equal to the
affine length of the corresponding vertical boundary face. Then we
choose $\ell(d)$ to be a path in the base whose slope is $d$ times the
slope of $\ell(1)$, relative to $\ell(0)$. We also make sure that in
the fiber, the slope $L(d)$ is $d$ times the slope of $L(1)$ (relative
to $L(0)$ in each fiber).

\subsection{Holomorphic and tropical triangles}
\label{sec:other-triangles}

Having chosen a family of Lagrangians $\{L(d)\}_{d\in\Z}$
corresponding to the powers of an ample line bundle, we find that
after positively perturbation, $HF^*(L(d_1),L(d_2))$ is concentrated
in degree $0$ when $d_1 \leq d_2$. The techniques of section
\ref{sec:degeneration} allow us to compute the holomorphic triangles
contributing to the multiplication
\begin{equation}
  HF^*(L(d_2),L(d_3))\otimes HF^*(L(d_1),L(d_2))\to HF^*(L(d_1),L(d_3))
\end{equation}
Looking at the winding numbers of the Lagrangians in the base once
again yields an auxiliary $\Z$-grading. For fixed values of this
$\Z$-grading on the input, one can determine the number of times that
triangles contributing to the product cover the critical values $w_1,
\dots, w_n$; call these numbers $k_1,\dots,k_n$. Then the degeneration
process breaks the triangle into $k = \sum_{i=1}^n k_i$ copies of the
fibration over a disk with single critical value, as well a trivial
fibration over a $(k+3)$-gon. Over the disks the count of sections is
$1$, while the analysis of sections over the $(k+3)$-gon still goes
through because, in the fiber, the Lagrangian boundary condition is
still a sequence of curves on the cylinder whose slope changes
monotonically. Hence the matrix coefficients of this product are
binomial coefficients of the form $\binom{k}{s}$. 

In this degeneration argument, the Lefschetz singularities that come
from different focus-focus singularities are not distinguished, while
in the case of tropical triangles, different singularities of the
affine structure contribute differently to the tropical curves. We
find that this family of triangles, with total count $\binom{k}{s}$,
corresponds to several tropical triangles $T_{s_1,\dots,s_n}$ with
$s_1+\cdots+s_n = s$, indexed by ordered partitions of $s$, with $0$
allowed as a part (of which there are $\binom{s+n}{s}$). The triangle
$T_{s_1,\dots,s_n}$ uses the tropical disk emanating from the $i$-th
singularity either $s_i$ or $k_i-s_i$ times, and the multiplicity of
$T_{s_1,\dots,s_n}$ is
$\binom{k_1}{s_1}\binom{k_2}{s_2}\cdots\binom{k_n}{s_n}$. The equality
of the total counts
\begin{equation}
  \binom{k}{s} = \sum_{\{s_1,\dots,s_n \mid s_i \geq 0,
    \sum_{i=1}^n s_i = s\}} \prod_{i=1}^n \binom{k_i}{s_i}
\end{equation}
follows from comparing the coefficients of $x^s$ in the equation
\begin{equation}
  (1+x)^k = \prod_{i=1}^n (1+x)^{k_i}
\end{equation}

\section{Mirrors to divisor complements}
\label{sec:complements}

In this section we examine the relationship between the wrapped Floer
cohomology of our Lagrangians $L(d)$ and the cohomology of coherent
sheaves on complements of components of the anticanonical divisor in
$\CP^2$. Let $D = C \cup L$ denote the anticanonical divisor which is
the union of a conic $C$ and a line $L$. We can consider the
divisor complements $U_D =\CP^2 \setminus D$, $U_C = \CP^2 \setminus
C$, and $U_L = \CP^2 \setminus L$. The torus fibration $\CP^2
\setminus D$ can be restricted to such a complement, and T-duality
gives the same space $X^\vee$ as before, but with a different
superpotential, reflecting the counts of holomorphic disks
intersecting the remaining components of the anticanonical
divisor.

\begin{table}
  \centering
  \begin{tabular}{ll|ll}
    Variety & Anticanonical divisor & Mirror space & Superpotential\\
    \hline
    $\CP^2$ & $D = C \cup L$ & $X^\vee = \{(u,v) \mid uv\neq 1\}$ & $W
    = u + \frac{e^{-\Lambda}v^2}{uv-1}$\\
    $U_L = \CP^2 \setminus L$ & $C \setminus (C\cap L)$ & $X^\vee$ & $W_L =
    u$\\
    $U_C = \CP^2 \setminus C$ & $L \setminus (L\cap C)$ & $X^\vee$ & $W_C =
    \frac{e^{-\Lambda}v^2}{uv-1}$\\
    $U_D = \CP^2 \setminus D$ & $\emptyset$ & $X^\vee$ & $W_D = 0$
  \end{tabular}
  \caption{Mirrors to divisor complements.}
  \label{tab:divisor-complements}
\end{table}

Once again, the cohomology of coherent sheaves $\cO(d)$ corresponds to
Floer cohomology of the Lagrangian submanifolds $L(d)$.

Removing a divisor $D$ from a compact variety $X$ changes the cohomology of a
coherent sheaf $\cF$, since, for example, sections of $\cF$ with poles
along $D$ are regular on the complement $U=X\setminus D$, so that
$H^0(U,\cF)$ is not finitely generated in general.

On the symplectic side, changing the superpotential by dropping a term
modifies the boundary condition for our Lagrangian submanifolds
$L(d)$. Some parts of $L(d)$ that were required to lie on the fiber of
$W$ are no longer so constrained, and it is appropriate to \emph{wrap}
these parts of $L(d)$. The algebraic structure associated to $L(d)$ is
then \emph{wrapped Floer cohomology} $HW^*(L(d_1),L(d_2))$, which is the
limit $\lim_{w\to \infty} HF^*(\phi_{wH}(L(d_1)),L(d_2))$, where
$H$ is an appropriate Hamiltonian function (a more precise definition
is given below). The limit
$HW^*(L(d_1),L(d_2))$ will not be finitely generated in general, since
it potentially contains trajectories of $H$ joining $L(d_1)$ to
$L(d_2)$ of \emph{any} length. The general theory of wrapped Floer
cohomology is developed in \cite{as-open-string}.

\subsection{Algebraic motivation}
\label{sec:algebraic-motivation}

In order to motivate the symplectic constructions of wrapped Floer
cohomology, it is useful to understand the algebraic side first. The
starting point is the following proposition (\cite{seidel-a-infinity-subalgebras}, statement
(1.10)).
\begin{proposition}
  Let $X$ be a smooth quasiprojective variety over $\C$, $Y \subset
  X$ a hypersurface, and $U = X \setminus Y$ the complement. Write $Y
  = s^{-1}(0)$, where $s$ is the canonical section of the line bundle
  $\cL = \cO_X(Y)$. Let $\cF$ be a coherent sheaf on
  $X$. Multiplication by $s$ defines an inductive system
  \begin{equation}
    \xymatrix{
      H^*(X, \cF\otimes \cL^{r-1}) \ar[r] & H^*(X,\cF\otimes
      \cL^r) \ar[r]&
      H^*(X,\cF \otimes \cL^{r+1}) \ar[r] & \cdots
    }
  \end{equation}
  and the limit is
  \begin{equation}
    \lim_{r\to \infty} H^*(X,\cF\otimes \cL^r) \cong H^*(U , \cF|U)
  \end{equation}
\end{proposition}

We spell out the application of this proposition to each of the cases
we consider

\begin{itemize}
\item $U_L$: Since $L = \{y=0\}$ is a line, we identify $\cL \cong \cO(1)$ and
  take $s = y$. Thus
  \begin{equation}
    H^*(U_L,\cO(d)) \cong \lim_{r\to \infty} H^*(\CP^2,\cO(d+r))
  \end{equation}
  where the limit is formed with respect to multiplication by $y$. An
  element of $H^0(U_L,\cO(d))$ is a rational function $f(x,y,z)/y^r$,
  where $f$ is a homogeneous polynomial of degree $d+r$.
\item $U_C$: Since $C = \{xz-y^2 = 0\}$ is a conic, we identify $\cL
  \cong \cO(2)$ and take $s = p = xz-y^2$. Thus
  \begin{equation}
    H^*(U_C,\cO(d)) \cong \lim_{r\to \infty} H^*(\CP^2,\cO(d+2r))
  \end{equation}
  where the limit is formed with respect to multiplication by $p =
  xz-y^2$.  An element of $H^0(U_C,\cO(d))$ is a rational function
  $f(x,y,z)/p^r$, where $f$ is a homogeneous polynomial of degree
  $d+2r$.
\item $U_D$: Since $D = \{yp = xyz-y^3 = 0\}$ is a cubic, we identify
  $\cL \cong \cO(3)$ and take $s = yp$. Thus
  \begin{equation}
    H^*(U_D,\cO(d)) \cong \lim_{r \to \infty} H^*(\CP^2,\cO(d+3r))
  \end{equation}
  where the limit is formed with respect to multiplication by $yp$. An
  element of $H^0(U_D,\cO(d))$ is a rational function
  $f(x,y,z)/(yp)^r$, where $f$ is a homogeneous polynomial of degree
  $d+3r$.
\end{itemize}

For the purposes of computation, a useful simplification comes from
noting that the line bundles $\cO(d)$ may become isomorphic over
the complements.

\begin{itemize}
\item $U_L$: Since $U_L \cong \C^2$, all the line bundles $\cO(d)$ are
  isomorphic over it.
\item $U_C$: The complement of a smooth conic in $\CP^2$ has
  $H^2(U_C;\Z) \cong \Z/2\Z$, generated by $c_1(\cO(1))$. The defining
  section $p: \cO\to \cO(2)$ is an isomorphism over $U_C$, and so $\Pic
  (U_C) \cong\Z/2\Z$ as well.
\item $U_D$: Since $U_D \subset U_L$, all the line bundles $\cO(d)$
  are isomorphic over it as well.
\end{itemize}

\subsection{Wrapping}
\label{sec:wrapping}

In this subsection we describe the geometric setup for wrapped Floer
cohomology in the mirrors of $U_L,U_C$, and $U_D$. 

\subsubsection{Completions}
\label{sec:completions}

Wrapped Floer cohomology is formulated in terms of noncompact
manifolds containing noncompact Lagrangian submanifolds. These can be
defined as (partial) completions of compact manifolds with boundary.

The starting point for all three cases is the original Lefschetz
fibration $X(B)\to X(I)$ containing the Lagrangians
$\{L(d)\}_{d\in\Z}$. We define completions $\hat{X}_L$, $\hat{X}_C$,
and $\hat{X}_D$. This process involves replacing a boundary component
with an infinite end, either on the base or in the fiber of the
Lefschetz fibration.
% The symplectic form constructed in section
% \ref{sec:main-construction} has the defect that it blows up at the
% corners of $B$. We could potentially work with it directly, but the
% technically safest way to deal with it is to simply cut these corners
% off by restricting the fibration to a sub-annulus of $X(I)$. We remark
% that it is clear that the results of the previous sections all carry
% over to this manifold without change: nowhere was the behavior at the
% corners essentially used other than in motivating the restriction on how the
% Lagrangians $L(d)$ should behave near the corners.

% \begin{remark}
%   If one wanted a mirror interpretation of cutting off the corners of
%   $B$ and
%   completing, it would be blowing up the intersections $C \cap L$, and
%   then removing the total inverse image of either $L$, $C$, or $D$,
%   which is the same as just removing $L$, $C$, or $D$.
% \end{remark}

\begin{itemize}
\item $L$: The manifold $\hat{X}_L$ retains a boundary component at
  the top horizontal boundary, corresponding to the fiber of the
  superpotential $W_L = u$. The fibers of $X(B) \to X(I)$ are
  completed at the other, bottom, end. The base annulus $X(I)$ is
  completed to a cylinder $\hat{X}(I)$, and the fibration is extended
  over this cylinder.
\item $C$: The manifold $\hat{X}_C$ retains a boundary component at
  the bottom horizontal boundary, corresponding to the fiber of the
  superpotential $W_C = \frac{e^{-\Lambda}v^2}{uv-1}$. The fibers are
  completed at the other, top, end. The base annulus is completed to a
  cylinder $\hat{X}(I)$, and the fibration is extended over this cylinder.
\item $D$: The manifold $\hat{X}_D$ has no boundary, and the fibers are
  completed at both ends. The base annulus is completed to a
  cylinder $\hat{X}(I)$, and the fibration is extended over this cylinder.
\end{itemize}

The torus fibration on $X(B)$ extends to these completions, and yields
torus fibrations over completed bases $\hat{B}_L$,$\hat{B}_C$,
$\hat{B}_D$.
\begin{itemize}
\item $\hat{B}_L$ is a half-plane with singular affine structure given
  by removing the bottom boundary from $B$ and extending in that
  direction.
\item $\hat{B}_C$ is a half-plane with singular affine structure given
  by removing the top boundary from $B$ and extending in that
  direction.
\item $\hat{B}_D$ is an entire plane with singular affine structure
  given by removing all boundaries from $B$ and extending in all directions.
\end{itemize}

The Lagrangian submanifolds $L(d)$ are extended to $\hat{L}(d)$; In
all cases, we extend $L(d)$ into whatever ends are attached so as to
be invariant under the Liouville flow within the end. However, in the
case of $L$, respectively $C$, we still have the boundary condition
that $\hat{L}(d)$ is required to end on $\Sigma_1$ (the complex
hypersurface contained in the top boundary), respectively $\Sigma_0$
(contained in the bottom boundary).

\subsubsection{Hamiltonians}
\label{sec:hamiltonians}

The most crucial difference between the three cases comes from the
choices of Hamiltonians that are be used to perform the wrapping. The
Hamiltonians we consider are the sum of contributions from the base
and the fiber. 

Let $H_b: \hat{X}(B) \to \R$ be the pullback of a function on the base
cylinder which is a function of the radial coordinate $\eta = \log|w|$
only. Writing the symplectic form on the base as $d\rho \wedge
d\theta$, where $\rho$ is a function of $\eta$, we take $H_b$ to be a
convex function of $\rho$ on the compact part $X(I)$, and linear in
$\rho$ on the ends. We also require $H_b \geq 0$, with minimum on the
central circle $\eta = 0$. Since $dH_b$ vanishes on the fibers, $X_{H_b}$ is
horizontal.

The fiber Hamiltonian $H_f$ is chosen differently in each case. The main
constraint is that its differential must vanish at any boundary
component which may still be present. The construction is most
convenient if we assume the completion preserves the $S^1$-symmetry
that rotates the fibers. If $\mu$ denotes the moment map for this
action, we can take $H_f$ to be a function of
$\mu$. Since $X_{H_f}$ is tangent to the fibers, we have
\begin{equation}
  \{H_b,H_f\} = \omega(X_{H_b},X_{H_f}) = 0
\end{equation}
which allows us to compute the flow of $H_b+H_f$ term by term.

We use the same base Hamiltonian $H_b$ for all cases. The specific
choice of $H_f$ in each case is as follows.
\begin{itemize}
\item $L$: Let $H_{f,L} \geq 0$ be a function with a minimum at the
  top of the fiber, convex in $\mu$ on the compact part, and linear
  in $\mu$ on the bottom end. 
\item $C$: Let $H_{f,C} \geq 0$ be a function with a minimum at the
  bottom of the fiber, convex in $\mu$ on the compact part, and linear
  in $\mu$ on the top end.
\item $D$: Let $H_{f,D} \geq 0$ be a function with a minimum in the
  two-thirds of the way down from the top of the compact part of the fiber, convex in $\mu$ on the
  compact part, and linear in $\mu$ on the ends.
\end{itemize}

In each case, the total Hamiltonian $H = H_b + H_f$ achieves its minimum along a torus in $X(B)$, which is one of the fibers of the torus fibration over $B$. Since $L(r)$ is a section of the torus fibration, $H$ has a unique minimum on $L(r)$, and hence there is a unique constant chord for $L(r)$, which corresponds to the identity element $e_r \in HW^*(L(r),L(r))$. The cases are:
\begin{itemize}
\item $L$: The minimum is along the torus corresponding to the midpoint of the top edge of $B$, where the generator $y$ lies. All of the Lagrangians $L(r)$ for $r \in \Z$ intersect this torus at the same point. As an intersection point of $L(0)$ and $L(d)$, the minimum corresponds to $q_{a,i}$ for $(a,i) = (0,0)$.
\item $C$: The minimum is along the torus corresponding to the midpoint of the bottom edge of $B$, where the generator $p$ lies. All of the Lagrangians $L(r)$ for $r \in 2\Z$ intersect this torus at the same point. As an intersection point of $L(0)$ and $L(r)$, the minimum corresponds to $q_{a,i}$ for $(a,i)= (0,r/2)$.
\item $D$: The minimum is along the torus two-thirds of the way from the top of the middle fiber of the map $B \to I$, where the generator $yp$ lies. All of the Lagrangians $L(r)$ for $r \in 3\Z$ intersect this torus in the same point. As an intersection point of $L(0)$ and $L(r)$, the minimum corresponds to $q_{a,i}$ for $(a,i) = (0,r/3)$.
\end{itemize}

\begin{remark}
  The Hamiltonians we obtain as $H_b + H_f$ are not admissible in the
  usual sense, because they vanish at some boundaries, and, even in
  the case $D$, are not linear with respect to a cylindrical
  end. Closer to our situations are the \emph{Lefschetz admissible}
  Hamiltonians considered by Mark McLean \cite{mclean-lf-sh}, that are
  precisely those functions on the total space of a Lefschetz
  fibration that are the sum of admissible Hamiltonians on the base
  and fiber separately.
\end{remark}

\subsubsection{Generators}
\label{sec:generators}

Given Lagrangian submanifolds $L_1, L_2$ of $X$, equipped with
Hamiltonian $H$, and $r \in \R$, we get Floer cohomology complexes
$CF^*(L_1,L_2;rH)$ generated by time-$1$ trajectories of
$X_{rH}$ starting on $L_1$ and ending on $L_2$. As usual
the differential counts inhomogeneous pseudo-holomorphic strips. We
also have continuation maps
\begin{equation}
  CF^*(L_1,L_2;rH) \to CF^*(L_1,L_2;r'H), \quad r < r'
\end{equation}
given by counting strips where the inhomogeneous term interpolates
between $r'X_H$ and $rX_H$. At the homology level, the continuation
maps form an inductive system, and we define the \emph{wrapped Floer cohomology}
\begin{equation}
  HW^*(L_1,L_2) = \lim_{r\to \infty} HF^*(L_1,L_2;rH)
\end{equation}

Our purpose in this section is simply to set up an enumeration of
the generators of $CF^*(L(d_1),L(d_2);rH)$ in each of the three
cases. These generators can also be regarded as intersection points
$\phi_{rH}(L(d_1))\cap L(d_2)$. In order to make the situation as
convenient as possible for our later arguments, we refine our choice
of Hamiltonians so as to ensure that $\phi_{rH}(L(d))$ is actually
$L(d')$ for some $d'$; thus we can identify wrapped Floer cohomology
generators with intersection points of our original Lagrangians. This
is done by adjusting the slopes of our Hamiltonians on the ends.
\begin{itemize}
\item We take the base Hamiltonian $H_b$ so that the time--$1$ flow
  completes $1$ turn on the cylindrical ends of the base.
\item For cases $L$ and $C$, we take the fiber Hamiltonian $H_f$ so that the
  time--$1$ flow completes $1/2$ turn on the cylindrical end of the
  fiber.
\item For case $D$, we take the fiber Hamiltonian $H_f$ so that the
  time--$1$ flow completes $1/3$ turn at the top of the fiber, and
  $1/6$ turn at the bottom of the fiber.
\end{itemize}
In all cases the total Hamiltonian we use is $H = H_b + H_f$.

The way to understand the flow of $H$ is to first apply $H_f$, then
$H_b$. The flow of $H_f$ wraps $L(d)$ in the fiber,
while the flow of $H_b$, when it completes a loop in the base,
performs the monodromy of the Lefschetz fibration around that loop,
which undoes some of the wrapping due to $H_f$.

We can relate $\phi_H(L(d))$ to $L(d')$ as follows:
\begin{itemize}
\item $L$: we have $\phi_{rH}(L(d)) = L(d-r)$.
\item $C$: we have $\phi_{2rH}(L(d)) = L(d-2r)$. Note that the same
  \emph{cannot} be said with $r$ in place of $2r$; in that case the
  two Lagrangians intersect the bottom boundary (where no wrapping
  occurs) in different points.
\item $D$: we have $\phi_{3rH}(L(d)) = L(d-3r)$. Again the same cannot
  be said with $r$ in place of $3r$.
\end{itemize}

In order to identify generators with intersection points, we perturb
the boundary intersection points in a positive sense just as before. 
In the cases with boundary, where the Hamiltonian is supposed to
have a minimum at the boundary, it is useful to perform this
perturbation in an extra collar attached to the boundary, so that the
Lagrangians still intersect at the minimum of $H$ if they did prior to
the perturbation. Once this is done, we can identify
\begin{equation}
  CF^*(L(d_1),L(d_2);rH) \cong CF^*(\phi_{rH}(L(d_1)),L(d_2)) \cong CF^*(L(d_1-r),L(d_2))
\end{equation}
where $r\in \Z$ in case $L$, $r\in 2\Z$ in case $C$, and $r\in 3\Z$
in case $D$. This identification is compatible with the gradings, which shows that for $r > d_1 - d_2$, the Floer complex $CF^*(L(d_1),L(d_2);rH)$ is concentrated in degree zero, and so has vanishing differential. Recall that the generators of the last group are
identified with $B(\frac{1}{d_2-d_1+r}\Z)$.

Once we are in the range $d_2-d_1 + r > 0$, we find that as $r$
increases, new generators are created, none are destroyed, and the
generators that already exist are ``compressed'' toward the minimum of
$H$. This gives rise to naive inclusion maps $i :
CF^*(L(d_1),L(d_2);rH) \to CF^*(L(d_1),L(d_2);r'H)$ for $r < r'$,
where $r > d_1-d_2$. In terms of fractional integral points, this $i$
corresponds to the map $B(\frac{1}{d_2-d_1+r}\Z) \to
B(\frac{1}{d_2-d_1+r'}\Z)$ which is dilation by the appropriate factor
centered at the point corresponding to the minimum of $H$.

We can index the points of $\hat{B}(\frac{1}{d}\Z)$ with two
indices. The index $a\in \Z$ corresponds to the column lying at $\eta
= a/d$, while the index $i$ that indexes points within a column, and
which lies in $\{0,\dots,\left\lfloor\frac{d-|a|}{2}\right\rfloor\}$
in the compact case, is now unbounded in the positive direction in
case $L$, in the negative direction in case $C$, and in both
directions in case $D$. We use the notation $q_{a,i}$ for these
points.

\subsection{Continuation maps and products}
\label{sec:continuation-product}

We saw above that for $r > d_1-d_2$, the generators of $CF^*(L(d_1),L(d_2);rH)$ all
have degree $0$, so there are no differentials, and we can
identify these complexes with their homologies. In order to obtain the
wrapped Floer cohomology, we must determine the continuation
maps.

Let it be understood that we require $r \in \Z$ in case $L$, $r \in
2\Z$ in case $C$, and $r \in 3\Z$ in case $D$.

To get started, we consider the wrapped Floer cohomology of $L(d_1)$
with itself. Each complex $CF^*(L(d_1),L(d_1);rH)$ has a distinguished
element $e_r$, sitting at the minimum of $H$. The complexes are
filtered by action, and the continuation maps are non-increasing with
respect to this filtration. Under the $r \to r'$ continuation map,
$e_r \mapsto e_{r'}$; $e_r$ and $e_{r'}$ are the unique generators of
minimal action in their respective complexes, and in fact their
actions are equal, so the only strip is the constant map to the
minimum.

At this point we bring in the product structure.

\begin{proposition}
 Let $L_i$ one of the Lagrangians $L(d_i)$ for $i = 1, 2, 3$. Under the identification
  \begin{equation}
  CF^*(L_i,L_j;rH) \cong CF^*(\phi_{rH}(L_i),L_j) 
\end{equation}
The product 
\begin{equation}
  HF^*(L_2,L_3; rH) \otimes HF^*(L_1,L_2; sH) \to HF^*(L_1,L_3; (r+s)H)
\end{equation}
which counts inhomogeneous pseudo-holomorphic triangles corresponds to the product
\begin{equation}
  HF^*(\phi_{rH}(L_2),L_3) \otimes
  HF^*(\phi_{(r+s)H}(L_1),\phi_{rH}(L_2)) \to HF^*(\phi_{(r+s)H}(L_1),L_3)
\end{equation}
counting pseudo-holomorphic triangles with no inhomogeneous term.
\end{proposition}

\begin{proof}
  First we recall the equation for the product in the wrapped setting. The disk with three boundary punctures $S$ is equipped with strip-like ends: at the inputs they are parametrized by $(s,t) \in (0,\infty)\times [0,1]$, while at the output we have $(s,t) \in (-\infty,0)\times [0,1]$. $S$ is also equipped with a closed one form $\beta \in \Omega^1(S)$, whose restriction to $\partial S$ vanishes,  such that on the strip-like ends of $S$, $\beta$ has the form $r\,dt$, $s\,dt$ and $(r+s)\,dt$ for the two inputs and the output respectively. The curves we count are solutions of 
  \begin{equation}
    (du - X_H \otimes \beta)^{0,1} = 0
  \end{equation}
  Because $\beta$ is closed and $S$ has no first homology, we may find a primitive $\tau \in C^\infty(S)$ such that $d \tau = \beta$, and we assume $\tau = 0$ on the boundary corresponding to $L_3$.  Denoting by $\phi^t_H$ the flow of $X_H$ for time $t$, we can define a new map $v = \phi^{-\tau}_H \circ u$, meaning that for $p \in S$, we take $u(p)$ and flow it for time $-\tau(p)$ by $X_H$. This has the effect of collapsing each Hamiltonian chord to a single point. Then the map $v$ is pseudo-holomorphic without inhomogeneous term, but with a domain-dependent almost complex structure obtained by conjugating the original $J$ by $\phi_H^\tau$ at each point.

We must relate the products \begin{equation}
  HF^*(\phi_{rH}(L_2),L_3) \otimes
  HF^*(\phi_{(r+s)H}(L_1),\phi_{rH}(L_2)) \to HF^*(\phi_{(r+s)H}(L_1),L_3)
\end{equation}
defined with respect to this domain-dependent almost complex structure $\{J_p\}_{p\in S}$ to the operation that we studied in Section \ref{sec:degeneration}, where the complex structure $J$ was fixed. We fix attention to one strip-like end, say the one corresponding to $L_2$ and $L_3$. On this end the complex structure depends only on $t \in [0,1]$. There is a continuation map 
\begin{equation}
HF^*(\phi_{rH}(L_2),L_3;\{J_t\}) \to HF^*(\phi_{rH}(L_2),L_3;J)  
\end{equation}
counting strips with a domain-dependent almost complex structure that interpolates between $\{J_t\}$ and $J$. We claim that, in our situation, all such strips are constant. This means that the bases of intersection points on the two sides of the continuation map correspond. These continuation maps intertwine the product with respect to $\{J_p\}$ and the one defined using $J$, implying that these products agree when expressed in the basis of intersection points on either side of the correspondence.

It remains to prove the claim that all the continuation strips are constant. Let $u$ denotes such a strip. Because the Hamiltonian flow preserves the structure of the Lefschetz fibration, each $J_p$ for $p \in S$ is a complex structure with the property that the tangent space to a fiber is complex, and the projection to the base is holomorphic, once we equip the base with a suitable complex structure $j_p$. As $u$ is pseudo-holomorphic with respect to $\{J_p\}$, this implies that $\pi \circ u$ has the property that its differential at each point has rank either $0$ or $2$. Thus the image of $\pi \circ u$ is a set whose boundary is contained in the images of the Lagrangians $\phi_{rH}(L_2)$ and $L_3$. Because there are no topological strips in the base joining different intersection points of the base paths, we conclude that $\pi \circ u$ must be constant, and the image of $u$ is contained in a fiber. But in the fiber there are also no topological strips joining different intersection points. Hence $u$ itself is constant.
\end{proof}

When the differentials on the Floer complexes vanish, the identification of the products holds at the chain level
as well. Tracing the isomorphisms through, we find that
the product with $e_r$ induces the naive inclusion map on generators
\begin{equation}
  \mu^2(e_r,\cdot) = i : HF^*(L(d_1),L(d_2); sH) \to HF^*(L(d_1),L(d_2);(r+s)H)
\end{equation}

%  There is
% an identification between the product
% \begin{equation}
%   HF^*(L_2,L_3; rH) \otimes HF^*(L_1,L_2; sH) \to HF^*(L_1,L_2; (r+s)H)
% \end{equation}
% which counts inhomogeneous pseudo-holomorphic triangles, and the
% product
% \begin{equation}
%   HF^*(\phi_{rH}(L_2),L_3) \otimes
%   HF^*(\phi_{(r+s)H}(L_1),\phi_{rH}(L_2)) \to HF^*(\phi_{(r+s)H}(L_1),L_2)
% \end{equation}
% counting pseudo-holomorphic triangles, which holds at the homology
% level. 

Due to the compatibility of the product with the continuation maps,
and the fact that $e_r \mapsto e_{r'}$ under continuation, we find that
the naive inclusion maps commute with the continuation maps:
\begin{equation}
  \xymatrix{
    HF^*(L(d_1),L(d_2);sH)\ar[d]^{=} \ar[r]^{i}& HF^*(L(d_1),L(d_2); (s+r)H)\ar[d]^{\text{cont.}}\\
    HF^*(L(d_1),L(d_2);sH) \ar[r]^{i}& HF^*(L(d_1),L(d_2); (s+r')H)
  }
\end{equation}
Thus the continuation map agrees with the naive inclusion map, at
least on those generators which are in the image of
$HF^*(L(d_1),L(d_2);sH)$. It follows that the continuation maps agree
with the naive inclusion maps, at least for $r$ large enough depending
on a particular generator. Hence
\begin{equation}
  HW^*(L(d_1),L(d_2)) = \lim_{r\to \infty} HF^*(L(d_1),L(d_2);rH),
\end{equation}
where the limit is formed with respect to the continuation maps, or
with respect to the naive inclusion maps, or (what is equal) the
multiplications by the various elements $e_r$.

Spelling this out a bit more gives a precise correspondence with
section \ref{sec:algebraic-motivation}. Consider the isomorphism
\begin{equation}
  HF^*(L(d_1),L(d_1);rH) \cong HF^*(L(d_1-r),L(d_1)) \cong H^*(\CP^2, \cO(r))
\end{equation}
\begin{itemize}
\item $L$: For $r \in \Z$, this isomorphism identifies $e_r$ with $y^r$.
\item $C$: For $r \in 2\Z$, this isomorphism identifies $e_r$ with $p^{r/2}$.
\item $D$: For $r \in 3\Z$, this isomorphism identifies $e_r$ with $(yp)^{r/3}$.
\end{itemize}
Thus the directed systems computing wrapped Floer cohomology are
identified with those computing the cohomology of line bundles on the
divisor complements.

We can identify the basis of $HW^*(L(d_1),L(d_2))$ with
$\hat{B}(\frac{1}{d_2-d_1} \Z)$, where $\hat{B} =
\hat{B}_L,\hat{B}_C,\hat{B}_D$ is the completion of the affine
manifold. The sets $B(\frac{1}{d_2-d_1+r}\Z)$ embed in
$\hat{B}(\frac{1}{d_2-d_1}\Z)$ and this latter is their limit as $r
\to \infty$; the map
is dilation by $\frac{d_2-d_1+r}{d_2-d_1}$ centered at the minimum of
$H$.

Using the products we computed in section \ref{sec:degeneration}, we
can identify this basis $\hat{B}(\frac{1}{d}\Z)$ for $HW^*(L(0),L(d))$
with a basis of $H^*(U;\cO(d))$, and complete the proof of Theorem \ref{thm:wrapped-products}

\begin{itemize}
\item $L$: The point $q_{a,i}$ of $\hat{B}_L(\frac{1}{d}\Z)$
  corresponds to the function $x^{-a}p^iy^{d+a-2i}$ for $a \leq 0$,
  and $z^ap^iy^{d-a-2i}$ for $a \geq 0$. In this case $i \geq 0$ can
  be arbitrarily large, so the exponent of $y$ is allowed to be negative.
\item $C$: The point $q_{a,i}$ of $\hat{B}_C(\frac{1}{d}\Z)$
  corresponds to the function $x^{-a}p^iy^{d+a-2i}$ for $a \leq 0$,
  and $z^ap^iy^{d-a-2i}$ for $a \geq 0$. In this case $i \leq
  \left\lfloor\frac{d-|a|}{2}\right\rfloor$ can be negative, so the exponent of $p$ is
  allowed to be negative, while the exponent of $y$ is non-negative.
\item $D$:  The point $q_{a,i}$ of $\hat{B}_D(\frac{1}{d}\Z)$
  corresponds to the function $x^{-a}p^iy^{d+a-2i}$ for $a \leq 0$,
  and $z^ap^iy^{d-a-2i}$ for $a \geq 0$. In this case $i \in \Z$, so
  the exponents of $y$ and $p$ are allowed to be negative.

\end{itemize}

%\end{doublespace}
\bibliographystyle{amsplain}
\bibliography{symplectic}

\end{document}